\documentclass{tac}
\thirdleveltheorems
\usepackage[arrow, curve, frame, matrix, graph, tips]{xy}
\usepackage{amssymb}
\usepackage{amsmath}
\usepackage[colorlinks=true,urlcolor=blue,linkcolor=blue]{hyperref}
\usepackage{tabularx}

\author{Mark Weber}
\thanks{}
\address{Department of Mathematics, Macquarie University}

\title{Polynomials in categories with pullbacks}
\copyrightyear{2015}
\eaddress{mark.weber.math@gmail.com}

\newtheorem{thm}{\bf Theorem}
\newtheorem{prop}{\bf Proposition}
\newtheorem{lem}{\bf Lemma}
\newtheorem{cor}{\bf Corollary}

\newtheoremrm{defn}{\bf Definition}
\newtheoremrm{rem}{\bf Remark}
\newtheoremrm{rems}{\bf Remarks}
\newtheoremrm{exam}{\bf Example}
\newtheoremrm{exams}{\bf Examples}

\newcommand{\tn}[1]{\textnormal{#1}}
\newcommand{\tnb}[1]{\textnormal{\bf #1}}

\newcommand{\tensor}{\otimes}

\newcommand{\C}{\mathbb{C}}

\newcommand{\N}{\mathbb{N}}
\newcommand{\comp}{\circ}
\newcommand{\id}{\tn{id}}
\newcommand{\ca}[1]{\mathcal{#1}}
\newcommand{\ladj}{\dashv}
\newcommand{\iso}{\cong}

\newcommand{\Set}{\tnb{Set}}

\newcommand{\Cat}{\tnb{Cat}}
\newcommand{\Top}{\tnb{Top}}

\newcommand{\CAT}{\tnb{CAT}}

\newcommand{\op}{\tn{op}}
\newcommand{\co}{\tn{co}}

\newcommand{\TwoCAT}{\mathbf{2}{\tnb{-CAT}}}

\newcommand{\Span}[1]{{\tnb{Span}_{#1}}}
\newcommand{\Cospan}[1]{{\tnb{Cospan}_{#1}}}

\newcommand{\Polyc}[1]{{\tnb{Poly}_{#1}}}
\newcommand{\PFun}[1]{{\tnb{P}_{#1}}}

\newcommand{\PsAlg}[1]{\tn{Ps-}{#1}\tn{-Alg}}

\newcommand{\Alg}[1]{{#1}\tn{-Alg}}
\newcommand{\Algs}[1]{{#1}\tn{-Alg}_{\tn{s}}}

\newcommand{\Cart}{\tnb{CAT}_{\tn{pb}}}

\newcommand{\lft}[1]{{#1}^{\bullet}}
\newcommand{\rgt}[1]{{#1}_{\bullet}}

\renewcommand{\P}{\mathbb{P}}
\newcommand{\B}{\mathbb{B}}
\renewcommand{\S}{\mathbb{S}}

\newcommand{\SMCMnd}{\tnb{S}}
\newcommand{\BMCMnd}{\tnb{B}}
\newcommand{\MCMnd}{\tnb{M}}
\newcommand{\FCMnd}{\tnb{C}_{\tn{fin}}}
\newcommand{\FPMnd}{\tnb{P}_{\tn{fin}}}

\begin{document}

\maketitle

\begin{abstract}
The theory developed by Gambino and Kock, of polynomials over a locally cartesian closed category $\mathcal{E}$, is generalised for $\mathcal{E}$ just having pullbacks. The 2-categorical analogue of the theory of polynomials and polynomial functors is given, and its relationship with Street's theory of fibrations within 2-categories is explored. Johnstone's notion of ``bagdomain data'' is adapted to the present framework to make it easier to completely exhibit examples of polynomial monads.
\end{abstract}


\section{Introduction}
\label{sec:intro}

Thanks to unpublished work of Andr\'{e} Joyal dating back to the 1980's, polynomials admit a beautiful categorical interpretation. Given a multivariable polynomial function $p$ with natural number coefficients, like say
\begin{equation}\label{eq:expository-poly} p(w,x,y,z) = (x^3y + 2, 3x^2z + y) \end{equation}
one may break down its formation as follows. There is a set $\tnb{In} = \{w,x,y,z\}$ of ``input variables'' and a two element set $\tnb{Out}$ of ``output variables''. Rewriting $p(w,x,y,z) = (x^3y + 1 + 1,x^2z + x^2z + x^2z + y)$, there is a set
\[ \tnb{MSum} = \{x^3y, (1)_1, (1)_2, (x^2z)_1, (x^2z)_2, (x^2z)_3, y\} \]
of ``monomial summands'', and a set
\[ \tnb{UVar} = \{x_1,x_2,x_3,y_1,x_4,x_5,z_1,x_6,x_7,z_2,x_8,x_9,z_3,y_2\} \]
of ``usages of variables'', informally consisting of no $w$'s, nine $x$'s, two $y$'s and three $z$'s. The task of forming the polynomial $p$ can then be done in three steps. First one takes the input variables and duplicates or ignores them according to how often each variable is used. The book-keeping of this step is by means of the evident function $p_1:\tnb{UVar} \to \tnb{In}$, which in our example forgets the subscripts of elements of $\tnb{UVar}$. In the second step one performs all the multiplications, and this is book-kept by taking products over the fibres of the function $p_2:\tnb{UVar} \to \tnb{MSum}$ which sends each usage to the monomial summand in which it occurs, that is
\[ \begin{array}{ccccc} {x_1,x_2,x_3,y_1 \mapsto x^3y} && {x_4,x_5,z_1 \mapsto (x^2z)_1}
&& {x_6,x_7,z_2 \mapsto (x^2z)_2} \\ {x_8,x_9,z_3 \mapsto (x^2z)_3}
&& {y_2 \mapsto y.} && {} \end{array} \]
Finally one adds up the summands, and this is book-kept by summing over the fibres of the evident function $p_3:\tnb{MSum} \to \tnb{Out}$. Thus the polynomial $p$ ``is'' the diagram
\begin{equation}\label{eq:first-poly} \xygraph{!{0;(2,0):} {\tnb{In}}="p1" [r] {\tnb{UVar}}="p2" [r] {\tnb{MSum}}="p3" [r] {\tnb{Out}}="p4" "p1":@{<-}"p2"^-{p_1}:"p3"^-{p_2}:"p4"^-{p_3}} \end{equation}
in the category $\Set$. A categorical interpretation of the formula (\ref{eq:expository-poly}) from the diagram (\ref{eq:first-poly}) begins by regarding an $n$-tuple of variables as (the fibres of) a function into a given set of cardinality $n$. Duplication of variables is then interpretted by the functor $\Delta_{p_1} : \Set/\tnb{In} \to \Set/\tnb{UVar}$ given by pulling back along $p_1$, taking products by the functor $\Pi_{p_2} : \Set/\tnb{UVar} \to \Set/\tnb{MSum}$ and taking sums by applying the functor $\Sigma_{p_3}:\Set/\tnb{MSum} \to \Set/\tnb{Out}$ given by composing with $p_3$. Composing these functors gives
\[ \PFun{}(p) : \Set/\tnb{In} \to \Set/\tnb{Out} \]
the polynomial functor corresponding to the polynomial $p$.

Functors of the form $\Delta_{p_1}$, $\Pi_{p_2}$ and $\Sigma_{p_3}$ are part of the bread and butter of category theory. For any map $p_3$ in any category, one may define $\Sigma_{p_3}$ between the appropriate slices, and one requires only pullbacks in the ambient category to interpret $\Delta_{p_1}$ more generally. The functor $\Pi_{p_2}$ is by definition the right adjoint of $\Delta_{p_2}$, and its existence is a condition on the map $p_2$, called \emph{exponentiability}. Locally cartesian closed categories are by definition categories with finite limits in which all maps are exponentiable. Consequently a reasonable general categorical definition of polynomial is as a diagram
\begin{equation}\label{eq:poly-gen-intro} \xygraph{!{0;(2,0):} {X}="p1" [r] {A}="p2" [r] {B}="p3" [r] {Y}="p4" "p1":@{<-}"p2"^-{p_1}:"p3"^-{p_2}:"p4"^-{p_3}} \end{equation}
in some locally cartesian closed category $\ca E$. The theory polynomials and polynomial functors was developed at this generality in the beautiful paper \cite{GambinoKock-PolynomialFunctors} of Gambino and Kock. There the question of what structures polynomials in a locally cartesian closed $\ca E$ form was considered, and it was established in particular that polynomials can be seen as the arrows of certain canonical bicategories, with the process of forming the associated polynomial functor giving homomorphisms of bicategories.

In this paper we shall focus on the bicategory $\Polyc{\ca E}$ of polynomials and cartesian maps between them in the sense of \cite{GambinoKock-PolynomialFunctors}. Our desire to generalise the above setting comes from the existence of canonical polynomials and polynomial functors for the case $\ca E = \Cat$ and the wish that they sit properly within an established framework. While local cartesian closedness is a very natural condition of great importance to categorical logic, enjoyed for example by any elementary topos, it is not satisfied by $\Cat$. Avoiding the assumption of local cartesian closure may be useful also for applications in categorical logic. For example, the categories of classes considered in Algebraic Set Theory \cite{JoyalMoerdijk-AlgSetTheory} are typically not assumed to be locally cartesian closed, but the small maps are assumed to be exponentiable.

The natural remedy of this defect is to define a polynomial $p$ between $X$ and $Y$ in a category $\ca E$ with pullbacks to be a diagram as in (\ref{eq:poly-gen-intro}) such that $p_2$ is an exponentiable map. Since exponentiable maps are pullback stable and closed under composition, one obtains the bicategory $\Polyc{\ca E}$ together with the ``associated-polynomial-functor homomorphism'', as before. We describe this in Section \ref{sec:Poly-in-Cats}.

The main technical innovation of Sections \ref{sec:elementary} and \ref{sec:Poly-in-Cats} is to remove any reliance on type theory in the proofs, giving a completely categorical account of the theory. In establishing the bicategory structure on $\Polyc{\ca E}$ in Section 2 of \cite{GambinoKock-PolynomialFunctors}, the internal language of $\ca E$ is used in an essential way, especially in the proof of Proposition 2.9. Our development makes no use of the internal language. Instead we isolate the concept of a distributivity pullback in Section \ref{ssec:dpb} and prove some elementary facts about them. Armed with this technology we then proceed to give an elementary account of the bicategory of polynomials, and the homomorphism which encodes the formation of associated polynomial functors. Our treatment requires only pullbacks in $\ca E$.

Our second extension to the categorical theory of polynomials is motivated by the fact that $\Cat$ is a 2-category. Thus in Section \ref{ssec:2-categorical-version} we develop the theory of polynomials within a 2-category $\ca K$ with pullbacks, and the polynomial 2-functors that they determine. In this context the structure formed by polynomials is a degenerate kind of tricategory, called a 2-bicategory, which roughly speaking is a bicategory whose homs are 2-categories instead of categories. However except for this change, the theory works \emph{in the same way} as for categories. In fact our treatment of the 1-categorical version of the theory in Section \ref{sec:Poly-in-Cats} was tailored in order to make the previous sentence true (in addition to giving the desired generalisation).

A first source of examples of 2-categorical polynomials come from the 2-monads considered first by Street \cite{Street-FibrationIn2cats} whose algebras are fibrations. In Proposition \ref{prop:fib-monads-as-poly} these 2-monads are exhibited as being polynomial in general. Fibrations in a 2-category play another role in this work, because it is often the case that the maps participating in a polynomial may themselves be fibrations or opfibrations in the sense of Street. This has implications for the properties that the resulting polynomial 2-functor inherits. To this end, the general types of 2-functor that are compatible with fibrations are recalled from \cite{Weber-Fam2fun} in Section \ref{ssec:fam-2-fun}, and the polynomials that give rise to them are identified in Theorem \ref{thm:fibrational-polynomials}.

As explained in \cite{Bourke-Thesis,Lack-Codescent} certain 2-categorical colimits called codescent objects are important in 2-dimensional monad theory. Theorem \ref{thm:fibrational-polynomials} has useful consequences in \cite{Weber-CodescCrIntCat}, in which certain codescent objects which arise naturally from a morphism of 2-monads are considered. When these codescent objects arise from a situation conforming appopriately to the hypotheses of Theorem \ref{thm:fibrational-polynomials}, they acquire extra structure which facilitates their computation. Also of relevance to the computation of associated codescent objects, we have in Theorem \ref{thm:polynomials-sifted-colims} identified sufficient conditions on polynomials in $\Cat$ so that their induced polynomial 2-functors preserve all sifted colimits.

While the bicategorical composition of polynomials has been established in \cite{GambinoKock-PolynomialFunctors}, and more generally in Sections \ref{sec:Poly-in-Cats} and \ref{sec:Poly-in-2-Cats} of this paper, actually exhibiting explicitly a polynomial monad requires some effort due to the complicated nature of this composition. However one can often avoid the need to check monad axioms by using an alternative approach, based on Johnstone's notion of ``bagdomain data'' \cite{Johnstone-VariationsBagdomain}. The essence of this approach is described in Theorem \ref{thm:bagdomain-data} and its 2-categorical analogue Theorem \ref{thm:2-bagdomain-data}. These methods are then illustrated in Section \ref{ssec:exhibiting-examples}, where various fundamental examples of polynomial 2-monads on $\Cat$ are exhibited. In particular the 2-monads on $\Cat$ for symmetric and for braided monoidal categories are polynomial 2-monads.

Polynomial functors over some locally cartesian closed category $\ca E$ arise in diverse mathematical contexts as explained in \cite{GambinoKock-PolynomialFunctors}. They arise in computer science under the name of \emph{containers} \cite{AAG-Containers}. Tambara in \cite{Tambara-MultTransfer} studied polynomials over categories of finite $G$-sets motivated by representation theory and group cohomology. Very interesting applications of Tambara's work were found by Brun in \cite{Brun-WittTambara} to Witt vectors, and in \cite{Brun-WittSpectraCobordism} also to equivariant stable homotopy theory and cobordism. Moreover in \cite{BatJoyKockMas-Opetopes} one finds applications of polynomial functors to higher category theory.

Having generalised to the consideration of non-locally cartesian closed categories we have expanded the possible scope of applications. In this article we have described some basic examples of polynomial monads over $\Cat$. Further examples for $\Cat$ of relevance to operads are provided in \cite{BataninWeber-Guitart, Weber-CodescCrIntCat, Weber-OpPoly2Mnd}. The results of Section \ref{sec:Poly-in-Cats} apply also to polynomials over $\Top$ which were a part of the basic setting of the work of Joyal and Bisson \cite{BissonJoyal-DyerLashof} on Dyer-Lashof operations.

\emph{Notations}. We denote by $[n]$ the ordinal $\{0 < ... < n\}$ regarded as a category. The category of functors $\ca A \to \ca B$ and natural transformations between them is usually denoted as $[\ca A,\ca B]$, though in some cases we also use exponential notation $\ca B^{\ca A}$. For instance $\ca E^{[1]}$ is the arrow category of a category $\ca E$, and $\ca E^{[2]}$ is a category whose objects are composable pairs of arrows of $\ca E$. A 2-monad is a $\Cat$-enriched monad, and given a 2-monad $T$ on a 2-category $\ca K$, we denote by $\Algs T$ the 2-category of strict $T$-algebras and strict maps, $\Alg T$ the 2-category of strict algebras and strong maps{\footnotemark{\footnotetext{Which are $T$-algebra morphisms up to coherent isomorphism}}} and $\PsAlg T$ for the 2-category of pseudo-$T$-algebras and strong maps, following the usual notations of 2-dimensional monad theory \cite{BWellKellyPower-2DMndThy, Lack-Codescent}.

\section{Elementary notions}
\label{sec:elementary}

In this section we describe the elementary notions which underpin our categorical treatment of the bicategory of polynomials in Section \ref{sec:Poly-in-Cats}. In Section \ref{ssec:exponentiable-morphisms} we recall basic facts and terminology regarding exponentiable morphisms. In Section \ref{ssec:dpb} we introduce distributivity pullbacks, and prove various general facts about them.

\subsection{Exponentiable morphisms.}
\label{ssec:exponentiable-morphisms}

Given a morphism $f:X \to Y$ in a category $\ca E$, we denote by $\Sigma_f:\ca E/X \to \ca E/Y$ the functor given by composition with $f$. When $\ca E$ has pullbacks $\Sigma_f$ has a right adjoint denoted as $\Delta_f$, given by pulling back maps along $f$. When $\Delta_f$ has a right adjoint, denoted as $\Pi_f$, $f$ is said to be \emph{exponentiable}. A commutative square in $\ca E$ as on the left
\[ \xygraph{!{0;(3,0):} {\xybox{\xygraph{!{0;(1.2,0):} {A}="tl" [r] {B}="tr" [d] {D}="br" [l] {C}="bl" "tl":"tr"^-{f}:"br"^-{k}:@{<-}"bl"^-{g}:@{<-}"tl"^-{h}}}} [r]
{\xybox{\xygraph{!{0;(1.2,0):} {\ca E/A}="tl" [r] {\ca E/B}="tr" [d] {\ca E/D}="br" [l] {\ca E/C}="bl" "tl":"tr"^-{\Sigma_f}:@{<-}"br"^-{\Delta_k}:@{<-}"bl"^-{\Sigma_g}:"tl"^-{\Delta_h} [d(.5)r(.35)] :@{=>}[r(.3)]^-{\alpha}}}} [r]
{\xybox{\xygraph{!{0;(1.2,0):} {\ca E/A}="tl" [r] {\ca E/B}="tr" [d] {\ca E/D}="br" [l] {\ca E/C}="bl" "tl":@{<-}"tr"^-{\Delta_f}:"br"^-{\Pi_k}:"bl"^-{\Delta_g}:@{<-}"tl"^-{\Pi_h} [d(.5)r(.35)] :@{<=}[r(.3)]^-{\beta}}}}} \]
determines a natural transformation $\alpha$ as in the middle, as the mate of the identity $\Sigma_k\Sigma_f=\Sigma_g\Sigma_h$ via the adjunctions $\Sigma_h \ladj \Delta_h$ and $\Sigma_k \ladj \Delta_k$. We call $\alpha$ a \emph{left Beck-Chevalley cell} for the original square. There is another left Beck-Chevalley cell for this square, namely $\Sigma_h\Delta_f \to \Delta_g\Sigma_k$, obtained by mating the identity $\Sigma_k\Sigma_f=\Sigma_g\Sigma_h$ with the adjunctions $\Sigma_f \ladj \Delta_f$ and $\Sigma_g \ladj \Delta_g$. If in addition $h$ and $k$ are exponentiable maps, then taking right adjoints produces the natural transformation $\beta$ from $\alpha$, and we call this a \emph{right Beck-Chevalley cell} for the original square. There is another right Beck-Chevalley cell $\Delta_k\Pi_g \to \Pi_f\Delta_h$ when $f$ and $g$ are exponentiable. It is well-known that the original square is a pullback if and only if either associated left Beck-Chevalley cell is invertible, and when $h$ and $k$ are exponentiable, these conditions are also equivalent to the right Beck-Chevalley cell $\beta$ being an isomorphism. Under these circumstances we shall speak of the left or right Beck-Chevalley isomorphisms.

Clearly exponentiable maps are closed under composition and any isomorphism is exponentiable. Moreover, exponentiable maps are pullback stable. For given a pullback square as above in which $g$ is exponentiable, one has $\Sigma_h\Delta_f \iso \Delta_g\Sigma_k$, and since $\Sigma_h$ is comonadic, $\Delta_g$ has a right adjoint by the Dubuc adjoint triangle theorem \cite{Dubuc-KanExtensions}.

When $\ca E$ has a terminal object $1$ and $f$ is the unique map $X \to 1$, we denote by $\Sigma_X$, $\Delta_X$ and $\Pi_X$ the functors $\Sigma_f$, $\Delta_f$ and $\Pi_f$ (when it exists) respectively. In fact since $\Sigma_X:\ca E/X \to \ca E$ takes the domain of a given arrow into $X$, it makes sense to speak of it even when $\ca E$ doesn't have a terminal object. An object $X$ of a finitely complete category $\ca E$ is exponentiable when the unique map $X \to 1$ is exponentiable in the above sense (ie when $\Pi_X$ exists). A finitely complete category $\ca E$ is \emph{cartesian closed} when all its objects are exponentiable, and \emph{locally cartesian closed} when all its morphisms are exponentiable.

Note that as right adjoints the functors $\Delta_f$ and $\Pi_f$ preserve terminal objects. An object $h:A \to X$ of the slice category $\ca E/X$ is terminal if and only if $h$ is an isomorphism in $\ca E$, but there is also a canonical choice of terminal object for $\ca E/X$ -- the identity $1_X$. So for the sake of convenience we shall often assume below that $\Delta_f$ and $\Pi_f$ are chosen so that $\Delta_f(1_Y)=1_X$ and $\Pi_f(1_X)=1_Y$.

\subsection{Distributivity pullbacks.}
\label{ssec:dpb}
For $f:A \to B$ in $\ca E$ a category with pullbacks, $\Delta_f:\ca E/B \to \ca E/A$ expresses the process of pulling back along $f$ as a functor. One may then ask: what basic categorical process is expressed by the functor $\Pi_f:\ca E/A \to \ca E/B$, when $f$ is an exponentiable map?

Let us denote by $\varepsilon_f^{(1)}$ the counit of $\Sigma_f \ladj \Delta_f$, and when $f$ is exponentiable, by $\varepsilon_f^{(2)}$ the counit of $\Delta_f \ladj \Pi_f$. The components of these counits fit into the following pullbacks:
\begin{equation}\label{diag:Counits-BaseChange}
\xygraph{{\xybox{\xygraph{{Y}="tl" [r] {X}="tr" [d] {B}="br" [l] {A}="bl" "tl":"tr"^-{\varepsilon_{f,b}^{(1)}}:"br"^-{b}:@{<-}"bl"^-{f}:@{<-}"tl"^-{\Delta_fb} "tl":@{}"br"|-*{\scriptstyle{pb}}}}} [r(3.5)] 
{\xybox{\xygraph{{Q}="tl" [r] {P}="tm" [r] {A}="tr" [d] {B}="br" [l(2)] {R}="bl" "tl":"tm"^-{\varepsilon_{f,a}^{(2)}}:"tr"^-{a}:"br"^-{f}:@{<-}"bl"^-{\Pi_fa}:@{<-}"tl"^-{\varepsilon_{f,\Pi_fa}^{(1)}} "tl":@{}"br"|-*{\scriptstyle{pb}}}}}} 
\end{equation}
Now the universal property of $\varepsilon_f^{(1)}$, as the counit of the adjunction $\Sigma_f \ladj \Delta_f$, is equivalent to the square on the left being a pullback as indicated. An answer to the above question is obtained by identifying what is special about the diagram on the right in (\ref{diag:Counits-BaseChange}), that corresponds to the universal property of $\varepsilon_f^{(2)}$ as the counit of $\Delta_f \ladj \Pi_f$. To this end we make
\begin{defn}\label{def:pb-around}
Let $g:Z \to A$ and $f:A \to B$ be a composable pair of morphisms in a category $\ca E$. Then a \emph{pullback around} $(f,g)$ is a diagram
\[ \xygraph{{X}="tl" [r] {Z}="tm" [r] {A}="tr" [d] {B}="br" [l(2)] {Y}="bl" "tl":"tm"^-{p}:"tr"^-{g}:"br"^-{f}:@{<-}"bl"^-{r}:@{<-}"tl"^-{q} "tl":@{}"br"|-*{\scriptstyle{pb}}} \]
in which the square with boundary $(gp,f,r,q)$ is, as indicated, a pullback. A morphism $(p,q,r) \to (p',q',r')$ of pullbacks around $(f,g)$ consists of $s:X \to X'$ and $t:Y \to Y'$ such that $p's=p$, $qs=tq'$ and $r=r's$. The category of pullbacks around $(f,g)$ is denoted $\tn{PB}(f,g)$.
\end{defn}
For example the pullback on the right in (\ref{diag:Counits-BaseChange}) exhibits $(\varepsilon_{f,a}^{(2)}, \varepsilon_{f,\Pi_fa}^{(1)},\Pi_fa)$ as a pullback around $(f,a)$. One may easily observe directly that the universal property of $\varepsilon_{f,a}^{(2)}$ is equivalent to $(\varepsilon_{f,a}^{(2)}, \varepsilon_{f,\Pi_fa}^{(1)},\Pi_fa)$ being a terminal object of $\tn{PB}(f,a)$. Thus we make
\begin{defn}\label{def:dpb}
Let $g:Z \to A$ and $f:A \to B$ be a composable pair of morphisms in a category $\ca E$. Then a \emph{distributivity} pullback around $(f,g)$ is a terminal object of $\tn{PB}(f,g)$. When $(p,q,r)$ is a distributivity pullback, we denote this diagramatically as follows:
\[ \xygraph{{X}="tl" [r] {Z}="tm" [r] {A}="tr" [d] {B}="br" [l(2)] {Y}="bl" "tl":"tm"^-{p}:"tr"^-{g}:"br"^-{f}:@{<-}"bl"^-{r}:@{<-}"tl"^-{q} "tl":@{}"br"|-*{\scriptstyle{dpb}}} \]
and we say that this diagram exhibits $r$ as a distributivity pullback of $g$ \emph{along} $f$.
\end{defn}
Thus the answer to the question posed at the beginning of this section is: when $f:A \to B$ is an exponentiable map in $\ca E$ a category with pullbacks, the functor $\Pi_f:\ca E/A \to \ca E/B$ encodes the process of taking distributivity pullbacks along $f$.

For any $(p,q,r) \in \tn{PB}(f,g)$ one has a Beck-Chevalley isomorphism as on the left
\[ \begin{array}{lccr} {\Pi_{q}\Delta_p\Delta_g \iso \Delta_r\Pi_f} &&& {\delta_{p,q,r}:\Sigma_r\Pi_{q}\Delta_{p} \to \Pi_f\Sigma_g} \end{array} \]
which when you mate it by $\Sigma_r \ladj \Delta_r$ and $\Sigma_g \ladj \Delta_g$, gives a natural transformation $\delta_{p,q,r}$ as on the right in the previous display. When this is an isomorphism, it expresses a type of distributivity of ``sums'' over ``products'', and so the following proposition explains why we use the terminology \emph{distributivity} pullback.
\begin{prop}\label{prop:dist-pb}
Let $f$ be an exponentiable map in a category $\ca E$ with pullbacks. Then $(p,q,r)$ is a distributivity pullback around $(f,g)$ if and only if $\delta_{p,q,r}$ is an isomorphism.
\end{prop}
\begin{proof}
Since $(\varepsilon_{f,g}^{(2)}, \varepsilon_{f,\Pi_fg}^{(1)},\Pi_fg)$ is terminal in $\tn{PB}(f,g)$, one has unique morphisms $d$ and $e$ fitting into a commutative diagram
\[ \xygraph{!{0;(1.5,0):(0,.5)::} {Z}="d" [ru] {X}="me" [r] {Y}="e" [dr] {B}="c" [dl] {E}="e2" [l] {D}="me2" "d":@{<-}"me"^-{p}:"e"^-{q}:"c"^-{r} "d":@{<-}"me2"_-{\varepsilon^{(2)}_{f,g}}:"e2"_-{\varepsilon^{(1)}_{f,\Pi_fg}}:"c"_-{\Pi_fg} "me":"me2"_-{d} "e":"e2"^-{e} "me" [d(.6)r(.15)] "me":@{}"e2"|*{\scriptstyle{pb}}} \]
in which the middle square is a pullback by the elementary properties of pullbacks. Thus $(p,q,r)$ is a distributivity pullback if and only if $e$ is an isomorphism. Since the adjunctions $\Sigma_r \ladj \Delta_r$ and $\Sigma_g \ladj \Delta_g$ are cartesian, $\delta_{p,q,r}$ is cartesian, and so it is an isomorphism if and only if its component at $1_Z \in \ca E/Z$ is an isomorphism. Since $\Delta_p(1_Z)=1_X$ and $\Pi_{q}(1_{X})=1_{Y}$ one may easily witness directly that $(\delta_{p,q,r})_{1_X} = e$.
\end{proof}
When manipulating pullbacks in a general category, one uses the ``elementary fact'' that given a commutative diagram of the form
\[ \xygraph{{A}="tl" [r] {B}="tm" [r] {C}="tr" [d] {F}="br" [l] {E}="bm" [l] {D}="bl" "tl":"tm"(:"tr":"br":@{<-}"bm":@{<-}"bl":@{<-}"tl",:"bm") "tm" [d(.5)r(.5)] {\scriptstyle{pb}}} \]
then the front square is a pullback if and only if the composite square is. In the remainder of this section we identify three elementary facts about distributivity pullbacks.
\begin{lem}\label{lem:distpb-composition-cancellation}
(Composition/cancellation) Given a diagram of the form
\[ \xygraph{!{0;(1.5,0):(0,.5)::} {B_6}="p11" [d] {B_2}="p21" [d] {B}="p31" [d] {X}="p41" [r] {Y}="p42" [u(2)] {B_3}="p22" [u] {B_4}="p12" [r] {B_5}="p13" [d(3)] {Z}="p43"
"p11":"p12"^-{h_9}:"p13"^-{h_6}:"p43"^-{h_7}:@{<-}"p42"^-{g}:@{<-}"p41"^-{f}:@{<-}"p31"^-{h}:@{<-}"p21"^-{h_2}:@{<-}"p11"^-{h_8}
"p22"(:@{<-}"p21"^-{h_3},:@{<-}"p12"_-{h_5},:"p42"^-{h_4})
"p11" [d(.5)r(.5)] {\scriptstyle{pb}} "p31" [r(.5)] {\scriptstyle{dpb}} "p22" [d(.5)r(.5)] {\scriptstyle{pb}}} \]
in any category with pullbacks, then the right-most pullback is a distributivity pullback around $(g,h_4)$ if and only if the composite diagram is a distributivity pullback around $(gf,h)$.
\end{lem}
\begin{proof}
Let us suppose that right-most pullback is a distributivity pullback, and that $C_1$, $C_2$, $k_1$, $k_2$ and $k_3$ as in
\[ \xygraph{!{0;(1.5,0):(0,.5)::} {B_6}="p11" [d] {B_2}="p21" [d] {B}="p31" [d] {X}="p41" [r] {Y}="p42" [u(2)] {B_3}="p22" [u] {B_4}="p12" [r] {B_5}="p13" [d(3)] {Z}="p43"
"p11" [u(2)l] {C_1}="p01" [r(4)] {C_2}="p03" "p12" [l(.5)u(1.35)] {C_3}="p02"
"p11":"p12"^-{h_9}:"p13"^-{h_6}:"p43"|-{h_7}:@{<-}"p42"^-{g}:@{<-}"p41"^-{f}:@{<-}"p31"^-{h}:@{<-}"p21"^-{h_2}:@{<-}"p11"^-{h_8}
"p22"(:@{<-}"p21"^-{h_3},:@{<-}"p12"_-{h_5},:"p42"^-{h_4})
"p31":@{<-}@/^{2pc}/"p01"^-{k_1}:"p03"^-{k_2}:"p43"^-{k_3}
"p02"(:@/_{1pc}/@{.>}"p42"^(.7){k_4},:@{.>}"p03"|(.4){k_5},:@{<.}"p01"|-{k_6})
"p01":@{.>}@/_{1pc}/"p21"|-{k_7} "p02":@{.>}"p22"|(.3){k_8} "p02":@{.>}@/^{.5pc}/"p12"^-{k_9} "p03":@{.>}"p13"_-{k_{10}} "p01":@{.>}"p11"|-{k_{11}}
"p11" [d(.5)r(.5)] {\scriptstyle{pb}} "p31" [r(.5)] {\scriptstyle{dpb}} "p22" [d(.5)r(.5)] {\scriptstyle{dpb}}} \]
are given such that the square with boundary $(hk_1,gf,k_3,k_2)$ is a pullback. Then we must exhibit $r:C_1 \to B_6$ and $s:C_2 \to B_5$ unique such that $h_2h_8r=k_1$, $h_6h_9r=sk_2$ and $h_7s=k_3$. Form $C_3$, $k_4$ and $k_5$ by taking the pullback of $k_3$ along $g$, and then $k_6$ is unique such that $k_5k_6=k_2$ and $k_4k_6=fhk_1$. Clearly the square with boundary $(hk_1,f,k_4,k_6)$ is a pullback around $(f,h)$. From the universal property of the left-most distributivity pullback, one has $k_7$ and $k_8$ as shown unique such that $k_1=h_2k_7$, $h_3k_7=k_8k_6$ and $h_4k_8=k_4$. From the universal property of the right-most distributivity pullback, one has $k_9$ and $k_{10}$ as shown unique such that $k_8=h_5k_9$, $h_6k_9=k_{10}k_5$ and $h_7k_{10}=k_3$. Clearly $h_5k_9k_6=h_3k_7$ and so by the universal property of the top-left pullback square one has $k_{11}$ as shown unique such that $h_8k_{11}=k_7$ and $h_9k_{11}=k_9k_6$. Clearly $h_2h_8k_{11}=k_1$, $h_6h_9k_{11}=k_{10}k_2$ and $h_7k_{10}=k_3$ and so we have established the existence of maps $r$ and $s$ with the required properties.

As for uniqueness, let us suppose now that $r:C_1 \to B_6$ and $s:C_2 \to B_5$ are given such that $h_2h_8r=k_1$, $h_6h_9r=sk_2$ and $h_7s=k_3$. We must verify that $r=k_{11}$ and $s=k_{10}$. Since the right-most distributivity pullback is in particular a pullback, one has $k_9':C_3 \to B_4$ unique such that $h_4h_5k_9'=k_4$ and $h_6k_9'=sk_5$. Since $(h_4h_5,h_6)$ are jointly monic, and clearly $h_4h_5h_9r=h_4h_5k_9'k_6$ and $h_6h_9r=h_6k_9'k_6$, we have $h_9r=k_9'k_6$. By the universal property of the left-most distributivity pullback, it follows that $h_5k_9'=k_8$ and $h_8r=k_7$. Thus by the universal property of the left-most distributivity pullback, it follows that $k_9=k_9'$ and $k_{10}=s$. Since $(h_8,h_9)$ are jointly monic, $h_8k_{11}=k_7=h_8r$ and $h_9k_{11}=k_{9}k_6=h_9r$, we have $r=k_{11}$.

Conversely, suppose that the composite diagram is a distributivity pullback around $(gf,h)$, and that $C_1$, $C_2$, $k_1$, $k_2$ and $k_3$ as in
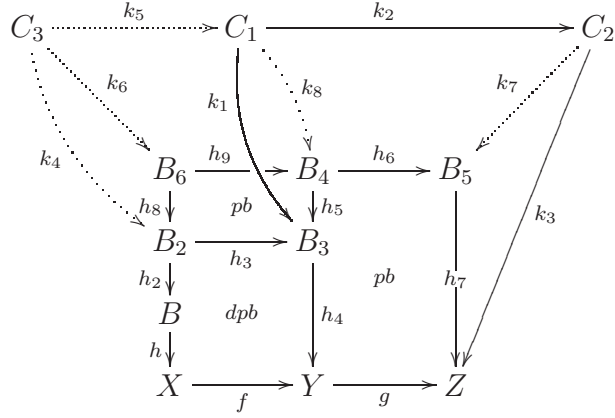
\[ \xygraph{!{0;(1.5,0):(0,.5)::} {B_6}="p11" [d] {B_2}="p21" [d] {B}="p31" [d] {X}="p41" [r] {Y}="p42" [u(2)] {B_3}="p22" [u] {B_4}="p12" [r] {B_5}="p13" [d(3)] {Z}="p43"
"p11" [u(2)l] {C_3}="p01" [r(4)] {C_2}="p03" "p12" [l(.5)u(2)] {C_1}="p02"
"p11":"p12"^(.35){h_9}|(.61)*{\hole}:"p13"^-{h_6}:"p43"|-{h_7}:@{<-}"p42"^-{g}:@{<-}"p41"^-{f}:@{<-}"p31"^-{h}:@{<-}"p21"^-{h_2}:@{<-}"p11"^-{h_8}
"p22"(:@{<-}"p21"^-{h_3},:@{<-}"p12"_-{h_5},:"p42"^-{h_4})
"p03":"p43"^-{k_3}
"p02"(:"p03"^(.4){k_2},:@{<.}"p01"_-{k_5})
"p01":@{.>}@/_{1pc}/"p21"_-{k_4} "p02":@/_{1pc}/"p22"_(.3){k_1} "p02":@{.>}@/^{.5pc}/"p12"^-{k_8} "p03":@{.>}"p13"_-{k_{7}} "p01":@{.>}"p11"^-{k_{6}}
"p11" [d(.5)r(.5)] {\scriptstyle{pb}} "p31" [r(.5)] {\scriptstyle{dpb}} "p22" [d(.5)r(.5)] {\scriptstyle{pb}}} \]
are given such that the square with boundary $(h_4k_1,g,k_3,k_2)$ is a pullback. We must give $r:C_1 \to B_4$ and $s:C_2 \to B_5$ unique such that $k_1=h_5r$, $h_6r=sk_2$ and $h_7s=k_3$. Pullback $k_1$ along $h_3$ to produce $C_3$, $k_4$ and $k_5$. This makes the square with boundary $(hh_2k_4,gf,k_3,k_2k_5)$ a pullback around $(gf,h)$. Thus one has $k_6$ and $k_7$ as shown unique such that $h_8k_6=k_4$, $h_6h_9k_6=k_7k_2k_1$ and $k_7h_7=k_3$. By universal property of the right pullback and since $gh_4k_1=h_7k_7k_2$, one has $k_8$ as shown unique such that $h_5k_8=k_1$ and $h_6k_8=k_7k_2$. By the uniquness part of the universal property of the left distributivity pullback, it follows that $h_5k_8=k_1$, and so we have established the existence of maps $r$ and $s$ with the required properties.

As for uniqueness let us suppose that we are given $r:C_1 \to B_4$ and $s:C_2 \to B_5$ such that $k_1=h_5r$, $h_6r=sk_2$ and $h_7s=k_3$. We must verify that $r=k_8$ and $s=k_7$. By the universal property of the top-left pullback one has $k_6'$ unique such that $h_8k_6'=k_4$ and $h_9k_6'=rk_5$. By the uniquness part of the universal property of the left distributivity pullback, it follows that $h_8k_6=h_8k_6'$ and $h_5k_8=h_5r$. Thus by the uniquness part of the universal property of the composite distributivity pullback, it follows that $k_6=k_6'$ and $s=k_7$. Since $(h_4h_5,h_6)$ are jointly monic, it follows that $r=k_8$.
\end{proof}
\begin{lem}\label{lem:distpb-cube}
(The cube lemma). Given a diagram of the form
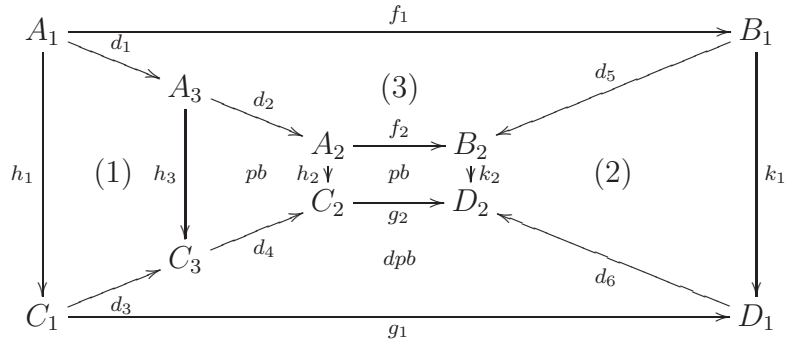
\[ \xygraph{!{0;(1.5,0):(0,.4)::} {A_2}="itl" ([ul] {A_3}="mtl" [ul] {A_1}="otl") [r] {B_2}="itr" ([u(2)r(2)] {B_1}="otr") [d] {D_2}="ibr" ([d(2)r(2)] {D_1}="obr")  [l] {C_2}="ibl" ([dl] {C_3}="mbl" [dl] {C_1}="obl") "otl":"otr"^-{f_1}:"obr"^-{k_1}:@{<-}"obl"^-{g_1}:@{<-}"otl"^-{h_1} "itl":"itr"^-{f_2}:"ibr"^-{k_2}:@{<-}"ibl"^-{g_2}:@{<-}"itl"^-{h_2} "mtl":"mbl"_-{h_3} "otl":"mtl"^-{d_1}:"itl"^-{d_2} "obl":"mbl"_-{d_3}:"ibl"_-{d_4} "otr":"itr"_-{d_5} "obr":"ibr"^-{d_6} "itl":@{}"ibr"|-*{\scriptstyle{pb}}="c" "c"  ([l] {\scriptstyle{pb}} [l] {(1)}, [r(1.5)] {(2)}) "itl" [r(.5)u] {(3)} "ibl" [r(.5)d] {\scriptstyle{dpb}}} \]
in any category with pullbacks, in which regions (1) and (2) commute, region (3) is a pullback around $(f_2,d_2)$, the square with boundary $(f_1,k_1,g_1,h_1)$ is a pullback and the bottom distributivity pullback is around $(g_2,d_4)$. Then regions (1) and (2) are pullbacks if and only if region (3) is a distributivity pullback around $(f_2,d_2)$.
\end{lem}
\begin{proof}
Let us suppose that (1) and (2) are pullbacks and $p$, $q$ and $r$ are given as in
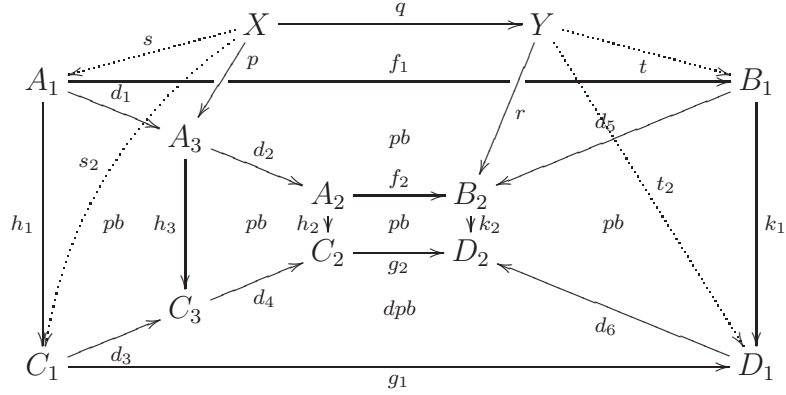
\[ \xygraph{!{0;(1.5,0):(0,.4)::} {A_2}="itl" ([ul] {A_3}="mtl" [ul] {A_1}="otl") [r] {B_2}="itr" ([u(2)r(2)] {B_1}="otr") [d] {D_2}="ibr" ([d(2)r(2)] {D_1}="obr")  [l] {C_2}="ibl" ([dl] {C_3}="mbl" [dl] {C_1}="obl") "otl":"otr"^-{f_1}|(.25)*{\hole}|(.665)*{\hole}:"obr"^-{k_1}:@{<-}"obl"^-{g_1}:@{<-}"otl"^-{h_1} "itl":"itr"^-{f_2}:"ibr"^-{k_2}:@{<-}"ibl"^-{g_2}:@{<-}"itl"^-{h_2} "mtl":"mbl"_-{h_3} "otl":"mtl"^-{d_1}:"itl"^-{d_2} "obl":"mbl"_-{d_3}:"ibl"_-{d_4} "otr":"itr"_-{d_5} "obr":"ibr"^-{d_6} "itl":@{}"ibr"|-*{\scriptstyle{pb}}="c" "c"  ([l] {\scriptstyle{pb}} [l] {\scriptstyle{pb}}, [r(1.5)] {\scriptstyle{pb}}) "itl" [r(.5)u] {\scriptstyle{pb}} "ibl" [r(.5)d] {\scriptstyle{dpb}}
"mtl" [u(2)r(.5)] {X}="x" [r(2)] {Y}="y" "mtl":@{<-}"x"_(.75){p}:"y"^-{q}:"itr"^-{r}
"x":@{.>}@/_{1.5pc}/"obl"_-{s_2} "y":@{.>}"obr"^-{t_2} "x":@{.>}"otl"_-{s} "y":@{.>}"otr"_-{t}} \]
such that the square with boundary $(q,r,f_2,d_2p)$ is a pullback. Then one can use the bottom distributivity pullback to induce $s_2$ and $t_2$ as shown, and then the pullbacks (1) and (2) to induce $s$ and $t$, and these clearly satisfy $d_1s=p$, $f_1s=tq$ and $r=d_5t$. On the other hand given $s':X \to A_1$ and $t':Y \to B_1$ satisfying these equations, define $s_2'=h_1s'$ and $t_2'=k_1t'$. But then by the uniqueness part of the universal property of the bottom distributivity pullback it follows that $s_2'=s_2$ and $t_2'=t_2$, and from the uniqueness parts of the universal properties of the pullbacks (1) and (2), it follows that $s=s'$ and $t=t'$, thereby verifying that $s$ and $t$ are unique satisfying the aforementioned equations.

For the converse suppose that (3) is a distributivity pullback. Note that (2) being a pullback implies that (1) is by elementary properties of pullbacks, so we must show that (2) is a pullback. To that end consider $s$ and $t$ as in
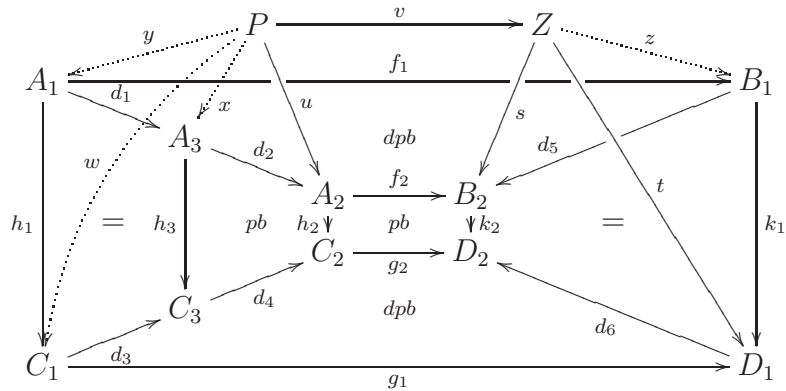
\[ \xygraph{!{0;(1.5,0):(0,.4)::} {A_2}="itl" ([ul] {A_3}="mtl" [ul] {A_1}="otl") [r] {B_2}="itr" ([u(2)r(2)] {B_1}="otr") [d] {D_2}="ibr" ([d(2)r(2)] {D_1}="obr")  [l] {C_2}="ibl" ([dl] {C_3}="mbl" [dl] {C_1}="obl") "otl":"otr"^-{f_1}|(.33)*{\hole}|(.665)*{\hole}|(.75)*{\hole}:"obr"^-{k_1}:@{<-}"obl"^-{g_1}:@{<-}"otl"^-{h_1} "itl":"itr"^-{f_2}:"ibr"^-{k_2}:@{<-}"ibl"^-{g_2}:@{<-}"itl"^-{h_2} "mtl":"mbl"_-{h_3} "otl":"mtl"^-{d_1}:"itl"^-{d_2} "obl":"mbl"_-{d_3}:"ibl"_-{d_4} "otr":"itr"_(.7){d_5}|(.5)*{\hole} "obr":"ibr"^-{d_6} "itl":@{}"ibr"|-*{\scriptstyle{pb}}="c" "c"  ([l] {\scriptstyle{pb}} [l] {=}, [r(1.5)] {=}) "itl" [r(.5)u] {\scriptstyle{dpb}} "ibl" [r(.5)d] {\scriptstyle{dpb}}
"mtl" [u(2)r(.5)] {P}="p" [r(2)] {Z}="z" "z"(:"itr"^-{s},:"obr"^-{t}) "p"(:"itl"^-{u},:"z"^-{v})
"p":@{.>}@/_{1.5pc}/"obl"_-{w} "p":@{.>}"mtl"^(.65){x} "p":@{.>}"otl"_-{y} "z":@{.>}"otr"^-{z}} \]
such that $k_2s=d_6t$, and then pullback $s$ and $f_2$ to produce $P$, $u$ and $v$. Using the fact that the bottom distributivity pullback is a mere pullback, one has $w$ unique such that $d_4d_3w=h_2u$ and $g_1w=tv$. Using the inner left pullback, one has $x$ unique such that $h_3x=d_3w$ and $d_2x=u$. Using the distributivity pullback (3), one has $y$ and $z$ unique such that $d_1y=x$, $f_1y=zv$ and $s=d_5z$. By the uniqueness part of the universal property of the bottom distributivity pullback, it follows that $t=k_1z$. Thus we have constructed $z$ satisfying $s=d_5z$ and $t=k_1z$. On the other hand given $z':Z \to B_1$ such that $s=d_5z'$ and $t=k_1z'$, one has $y':P \to A_1$ unique such that $d_2d_1y=u$ and $f_1y=zv$, using the fact that the top distributivity pullback is a mere pullback. Then from the uniqueness part of the universal property of that distributivity pullback, it follows that $y=y'$ and $z=z'$. Thus as required $z$ is unique satisfying $s=d_5z$ and $t=k_1z$.
\end{proof}
\begin{lem}\label{lem:distpb-sections}
(Sections of distributivity pullbacks). Let
\[ \xygraph{{D}="tl" [r] {A}="tm" [r] {B}="tr" [d] {C}="br" [l(2)] {E}="bl" "tl":"tm"^-{p}:"tr"^-{g}:"br"^-{f}:@{<-}"bl"^-{r}:@{<-}"tl"^-{q} "tl":@{}"br"|*{\scriptstyle{dpb}}} \]
be a distributivity pullback around $(f,g)$ in any category with pullbacks. Three maps
\[ \begin{array}{lcccr} {s_1:B \to A} && {s_2:B \to D} && {s_3:C \to E} \end{array} \]
which are sections of $g$, $gp$ and $r$ respectively, and are natural in the sense that $s_1=ps_2$ and $qs_2=s_3f$, are determined uniquely by the either of the following: (1) the section $s_1$; or (2) the section $s_3$.
\end{lem}
\begin{proof}
Given $s_1$ a section of $g$, induce $s_2$ and $s_3$ uniquely as shown:
\[ \xygraph{{D}="tl" [r] {A}="tm" [r] {B}="tr" [d] {C}="br" [l(2)] {E}="bl" "tl":"tm"^-{p}:"tr"^-{g}:"br"^-{f}:@{<-}"bl"^-{r}:@{<-}"tl"^-{q} "tl":@{}"br"|*{\scriptstyle{dpb}}
"tl" [l] {B}="tll" [d] {C}="bll" "tll"(:@/^{2pc}/"tm"^-{s_1},:"bll"_-{f}:@/_{2pc}/"br"_-{1})
"tll":@{.>}"tl"^-{s_2} "bll":@{.>}"bl"^-{s_3}} \]
using the universal property of the distributivity pullback. On the other hand given the section $s_3$, one induces $s_2$ using the fact that the distributivity pullback is a mere pullback, and then put $s_1=ps_2$.
\end{proof}
We often assume that in a given category $\ca E$ with pullbacks, some choice of all pullbacks, and of all existing distributivity pullbacks, has been fixed. Moreover we make the following harmless assumptions, for the sake of convenience, on these choices once they have been made. First we assume that the chosen pullback of an identity along any map is an identity. This ensures that $\Delta_{1_X}=1_{\ca E/X}$ and that $\Delta_f(1_B)=1_A$ for any $f:A \to B$. Similarly we assume that all diagrams of the form
\[ \xygraph{ {\xybox{\xygraph{{\bullet}="tl" [r] {\bullet}="tm" [r] {\bullet}="tr" [d] {\bullet}="br" [l(2)] {\bullet}="bl" "tl":"tm"^-{1}:"tr"^-{1}:"br"^-{f}:@{<-}"bl"^-{1}:@{<-}"tl"^-{f} "tl":@{}"br"|-*{\scriptstyle{dpb}}}}} [r(4)]
{\xybox{\xygraph{{\bullet}="tl" [r] {\bullet}="tm" [r] {\bullet}="tr" [d] {\bullet}="br" [l(2)] {\bullet}="bl" "tl":"tm"^-{1}:"tr"^-{g}:"br"^-{1}:@{<-}"bl"^-{g}:@{<-}"tl"^-{1} "tl":@{}"br"|-*{\scriptstyle{dpb}}}}}} \]
are among our chosen distributivity pullbacks. This has the effect of ensuring that $\Pi_f(1_A)=1_B$ for any exponentiable $f:A \to B$, and that $\Pi_{1_X}=1_{\ca E/X}$.

\section{Polynomials in categories}
\label{sec:Poly-in-Cats}

This section contains our general theory of polynomials and polynomial functors. In Section \ref{ssec:bicats-of-polys} we give an elementary account of the composition of polynomials, culminating in Theorem \ref{thm:bicat-of-polynomials}, in which polynomials in a category $\ca E$ with pullbacks are exhibited as the 1-cells of the bicategory $\Polyc {\ca E}$. Then in Section \ref{ssec:polynomial-functors}, we study the process of forming the associated polynomial functor, exhibiting this as the effect on 1-cells of the homomorphism $\PFun {\ca E} : \Polyc{\ca E} \to \CAT$ in Theorem \ref{thm:polnomial-functor-homomorphism}. At this generality, the homs of the bicategory $\Polyc {\ca E}$ have pullbacks, and the hom functors of $\PFun {\ca E}$ preserve them. This gives the sense in which the theory of polynomial functors could be iterated, and this is described in Section \ref{ssec:enriched-bicat-poly}. The organisation of this section has been chosen to facilitate its generalisation to the theory of polynomials in 2-categories, in Section \ref{ssec:2-categorical-version}.

\subsection{Bicategories of polynomials.}
\label{ssec:bicats-of-polys}
Let $\ca E$ be a category with pullbacks. In this section we give a direct description of a bicategory $\Polyc{\ca E}$, whose objects are those of $\ca E$, and whose one cells are polynomials in $\ca E$ in the following sense. For $X$, $Y$ in $\ca E$, a \emph{polynomial} $p$ from $X$ to $Y$ in $\ca E$ consists of three maps
\[ \xygraph{{X}="p1" [r] {A}="p2" [r] {B}="p3" [r] {Y}="p4" "p1":@{<-}"p2"^-{p_1}:"p3"^-{p_2}:"p4"^-{p_3}} \]
such that $p_2$ is exponentiable. Let $p$ and $q$ be polynomials in $\ca E$ from $X$ to $Y$. A \emph{cartesian morphism} $f:p \to q$ is a pair of maps $(f_0,f_1)$ fitting into a commutative diagram
\[ \xygraph{{X}="d" [ru] {A}="me" [r] {B}="e" [dr] {Y}="c" [dl] {B'}="e2" [l] {A'}="me2" "d":@{<-}"me"^-{p_1}:"e"^-{p_2}:"c"^-{p_3} "d":@{<-}"me2"_-{q_1}:"e2"_-{q_2}:"c"_-{q_3} "me":"me2"_-{f_0} "e":"e2"^-{f_1} "me":@{}"e2"|-*{\scriptstyle{pb}}} \]
We call $f_0$ the \emph{$0$-component} of $f$, and $f_1$ the \emph{$1$-component} of $f$. With composition inherited in the evident way from $\ca E$, one has a category $\Polyc{\ca E}(X,Y)$ of polynomials from $X$ to $Y$ and cartesian morphisms between them. These are the homs of our bicategory $\Polyc{\ca E}$.

In order to describe the bicategorical composition of polynomials, we introduce the concept of a \emph{subdivided composite} of a given composable sequence of polynomials. This enables us to give a direct description of $n$-ary composition for $\Polyc{\ca E}$, and then to describe the sense in which coherence for this bicategory ``follows from universal properties''.

Consider a composable sequence of polynomials in $\ca E$ of length $n$, that is to say, polynomials
\[ \xygraph{!{0;(1.5,0):(0,1)::} {X_{i{-}1}}="p1" [r] {A_i}="p2" [r] {B_i}="p3" [r] {X_i}="p4" "p1":@{<-}"p2"^-{p_{i1}}:"p3"^-{p_{i2}}:"p4"^-{p_{i3}}} \]
in $\ca E$, where $0 < i \leq n$. We denote such a sequence as $(p_i)_{1{\leq}i{\leq}n}$, or more briefly as $(p_i)_i$. 
\begin{defn}\label{def:subdivided-composite}
Let $(p_i)_{1{\leq}i{\leq}n}$ be a composable sequence of polynomials of length $n$. A \emph{subdivided composite} over $(p_i)_i$ consists of objects $(Y_0,...,Y_n)$, morphisms
\[ \begin{array}{lcccr} {q_1:Y_0 \to X_0} && {q_{2,i}:Y_{i{-}1} \to Y_i} && {q_3:Y_n \to X_n} \end{array} \]
for $0 < i \leq n$, and morphisms
\[ \begin{array}{lccr} {r_i:Y_{i{-}1} \to A_i} &&& {s_i:Y_i \to B_i} \end{array} \]
for $0 < i \leq n$, such that $p_{11}r_1=q_1$, $p_{n3}s_n=q_3$ and
\[ \xygraph{{\xybox{\xygraph{!{0;(1.5,0):(0,.6667)::} {Y_i}="tl" [r] {A_{i{+}1}}="tr" [d] {X_i}="br" [l] {B_i}="bl" "tl":"tr"^-{r_{i{+}1}}:"br"^-{p_{i{+}1,1}}:@{<-}"bl"^-{p_{i3}}:@{<-}"tl"^-{s_i}:@{}"br"|-{=}}}}
[r(4)]
{\xybox{\xygraph{!{0;(1.5,0):(0,.6667)::} {Y_{i{-}1}}="tl" [r] {Y_i}="tr" [d] {B_i}="br" [l] {A_i}="bl" "tl":"tr"^-{q_{2i}}:"br"^-{s_i}:@{<-}"bl"^-{p_{i2}}:@{<-}"tl"^-{r_i}:@{}"br"|-*{\scriptstyle{pb}}}}}} \]
\end{defn}
For example a subdivided composite over $(p_1,p_2,p_3)$, that is when $n=3$, assembles into a commutative diagram like this:
\[ \xygraph{{\bullet}="x0" [r] {\bullet}="a1" [r] {\bullet}="b1" [r] {\bullet}="x1" [r] {\bullet}="a2" [r] {\bullet}="b2" [r] {\bullet}="x2" [r] {\bullet}="a3" [r] {\bullet}="b3" [r] {\bullet}="x3" "a1" [u] {\bullet}="y0" [r(2)] {\bullet}="y1" [r(3)] {\bullet}="y2" [r(2)] {\bullet}="y3"
"x0":@{<-}"a1"_-{p_{11}}:"b1"_-{p_{12}}:"x1"_-{p_{13}}:@{<-}"a2"_-{p_{21}}:"b2"_-{p_{22}}:"x2"_-{p_{23}}:@{<-}"a3"_-{p_{31}}:"b3"_-{p_{32}}:"x3"_-{p_{33}} "y0":"y1"^-{q_{21}}:"y2"^-{q_{22}}:"y3"^-{q_{23}} "y0":"a1"_-{r_1} "y1"(:"b1"_-{s_1},:"a2"^-{r_2}) "y2"(:"b2"_-{s_2},:"a3"^-{r_3}) "y3":"b3"^-{s_3} "y0":"x0"_-{q_1} "y3":"x3"^-{q_3} "a1" [r(.5)u(.5)] {\scriptstyle{pb}} "a2" [r(.5)u(.5)] {\scriptstyle{pb}} "a3" [r(.5)u(.5)] {\scriptstyle{pb}}} \]
We denote a general subdivided composite over $(p_i)_i$ simply as $(Y,q,r,s)$.
\begin{defn}\label{def:morphism-subdivided-composite}
Let $(p_i)_{1{\leq}i{\leq}n}$ be a composable sequence of polynomials of length $n$. A \emph{morphism} $(Y,q,r,s) \to (Y',q',r',s')$ \emph{of subdivided composites} consists of morphisms $t_i:Y_i \to Y_i'$ for $0 \leq i \leq n$, such that $q_1=q_1't_0$, $q_{2i}'t_{i{-}1} = t_iq_{2i}$, $q_3=q_3't_n$, $r_i = r_i't_{i{-}1}$ and $s_i = s_i't_i$. With compositions inherited from $\ca E$, one has a category $\tn{SdC}(p_i)_i$ of subdivided composites over $(p_i)_i$ and morphisms between them.
\end{defn}
Given a subdivided composite $(Y,q,r,s)$ over $(p_i)_i$, note that the morphisms $q_{2i}$ are exponentiable since exponentiable maps are pullback stable, and that the composite $q_2:Y_0 \to Y_n$ defined as $q_2=q_{2n}...q_{21}$ is also exponentiable, since exponentiable maps are closed under composition. Thus we make
\begin{defn}\label{def:ass-poly}
The \emph{associated polynomial} of a given subdivided composite $(Y,q,r,s)$ over $(p_i)_i$ is defined to be
\[ \xygraph{{X_0}="p1" [r] {Y_0}="p2" [r] {Y_n}="p3" [r] {X_n}="p4" "p1":@{<-}"p2"^-{q_1}:"p3"^-{q_2}:"p4"^-{q_3}} \]
The process of taking associated polynomials is the object map of a functor
\[ \tn{ass} : \tn{SdC}(p_i)_i \longrightarrow \Polyc{\ca E}(X_0,X_n). \]
\end{defn}
Having made the necessary definitions, we now describe the canonical operations on subdivided composites which give rise to the bicategorical composition of polynomials. Let $n > 0$ and $(p_i)_{1{\leq}i{\leq}n}$ be a composable sequence of polynomials in $\ca E$. One has evident forgetful functors $\tn{res}_0$ and $\tn{res}_n$ as in
\[ \xygraph{!{0;(3.5,0):(0,1)::} {\tn{SdC}(p_i)_{1{<}i{\leq}n}}="p0" [r] {\tn{SdC}(p_i)_{1{\leq}i{\leq}n}}="p1" [r] {\tn{SdC}(p_i)_{1{\leq}i{<}n}}="p2"
"p0":@<1ex>@{<-}"p1"^-{\tn{res}_0}|{}="tl":@<1ex>"p2"^-{\tn{res}_n}|{}="tr"
"p0":@<-1ex>@{.>}"p1"_-{(-) \cdot p_{1}}|{}="bl":@{<.}@<-1ex>"p2"_-{p_{n} \cdot (-)}|{}="br" "tl":@{}"bl"|-{\perp} "tr":@{}"br"|-{\perp}} \]
and we now give a description of the right adjoints of these forgetful functors.

For $(Y,q,r,s)$ a subdivided composite over $(p_i)_{1{\leq}i{<}n}$, we construct the subdivided composite
\[ p_n \cdot (Y,q,r,s) := (p_n \cdot Y,p_n \cdot q,p_n \cdot r,p_n \cdot s) \]
over $(p_i)_{1{\leq}i{\leq}n}$ as follows. First we form the diagram on the left
\[ \xygraph{{\xybox{\xygraph{{X_{n{-}1}}="p1" [r] {A_n}="p2" [r(2)] {B_n}="p3" [r] {X_n}="p4" "p1" [u] {Y_{n{-}1}}="s" "p2" [u] {C}="mdpb" [u] {(p_n \cdot Y)_{n{-}1}}="tldpb" [r(2)] {(p_n \cdot Y)_n}="trdpb"
"p1":@{<-}"p2"_-{p_{n1}}:"p3"_-{p_{n2}}:"p4"_-{p_{n3}} "mdpb"(:"s"_-{}:"p1"_-{q_3},:"p2"^-{},:@{<-}"tldpb"_-{}:"trdpb"^-{(p_n \cdot q)_{2,n}}(:"p3"|-{},:"p4"^-{(p_n \cdot q)_3}|-{}="codeq")) "tldpb":"s"_-{\varepsilon_{n-1}}|{}="teq"
"mdpb" ([r] {\scriptstyle{dpb}}, [l(.5)d(.5)] {\scriptstyle{pb}}) "p3":@{}"codeq"|-*{=} "mdpb":@{}"teq"|-*{=}}}}
[r(4.75)]
{\xybox{\xygraph{!{0;(2,0):(0,1)::} {(p_n \cdot Y)_{n-k-1}}="tl" [r] {(p_n \cdot Y)_{n-k}}="tr" [d] {Y_{n-k}}="br" [l] {Y_{n-k-1}}="bl" "tl":"tr"^(.4){(p_n \cdot q)_{2,n-k}}:"br"^-{\varepsilon_{n-k}}:@{<-}"bl"^-{q_{2,n-k}}:@{<-}"tl"^-{\varepsilon_{n-k-1}} "tl":@{}"br"|*{\scriptstyle{pb}}}}}} \]
and then for $1 \leq k < n$ we form pullbacks as on the right in the previous display. Finally we define
\[ \begin{array}{lcccr} {(p_n \cdot q)_1 = q_1\varepsilon_0} && {(p_n \cdot r)_i = r_i\varepsilon_{i-1}} && {(p_n \cdot s)_i = s_i\varepsilon_i.} \end{array} \]
The $\varepsilon_i$ are the components of a morphism
\[ \varepsilon_{(Y,q,r,s)} : \tn{res}_n (p_n \cdot (Y,q,r,s)) \longrightarrow (Y,q,r,s) \]
of subdivided composites. The $n=4$ case of this construction is depicted in the diagram:
\[ \xygraph{!{0;(.75,0):(0,1)::} {\bullet}="x0" [r] {\bullet}="a1" [r] {\bullet}="b1" [r] {\bullet}="x1" [r] {\bullet}="a2" [r] {\bullet}="b2" [r] {\bullet}="x2" [r] {\bullet}="a3" [r] {\bullet}="b3" [r] {\bullet}="x3" "a1" [u] {\bullet}="y0" [r(2)] {\bullet}="y1" [r(3)] {\bullet}="y2" [r(2)] {\bullet}="y3"
"x3" [r] {\bullet}="a4" [r] {\bullet}="b4" [r] {\bullet}="x4"
"y3" [u(.5)r] {\bullet}="c" [u(.5)r] {\bullet}="z3" [r] {\bullet}="z4" 
"x0":@{<-}"a1"_-{p_{11}}:"b1"_-{p_{12}}:"x1"_-{p_{13}}:@{<-}"a2"_-{p_{21}}:"b2"_-{p_{22}}:"x2"_-{p_{23}}:@{<-}"a3"_-{p_{31}}:"b3"_-{p_{32}}:"x3"_-{p_{33}}:@{<-}"a4"_-{p_{41}}:"b4"_-{p_{42}}:"x4"_-{p_{43}}
"y0":"y1"|-{q_{21}}="bz0":"y2"|-{q_{22}}="bz1":"y3"|-{q_{23}}="bz2"
"bz0" [u] {\bullet}="z0" "bz1" [u] {\bullet}="z1" "bz2" [u] {\bullet}="z2"
"y0":"a1"|-{r_1} "y1"(:"b1"|-{s_1},:"a2"|-{r_2}) "y2"(:"b2"|-{s_2},:"a3"|-{r_3}) "y3":"b3"|-{s_3} "y0":"x0"|-{q_1} "y3":"x3"|-{q_3}
"c"(:"y3",:"a4",:@{<-}"z3":"z4"^-{(p_4{\cdot}q)_{24}}(:"b4",:"x4"^-{(p_4{\cdot}q)_3}))
"z0":"y0"|-{\varepsilon_0} "z1":"y1"|-{\varepsilon_1} "z2":"y2"|-{\varepsilon_2}
"z0":"z1"^-{(p_4{\cdot}q)_{21}}:"z2"^-{(p_4{\cdot}q)_{22}}:"z3"^-{(p_4{\cdot}q)_{23}}
"a1" [r(.5)u(.5)] {\scriptstyle{pb}} "a2" [r(.5)u(.5)] {\scriptstyle{pb}} "a3" [r(.5)u(.5)] {\scriptstyle{pb}}
"z0":@{}"y1"|-*{\scriptstyle{pb}} "z1":@{}"y2"|-*{\scriptstyle{pb}} "z2":@{}"y3"|-*{\scriptstyle{pb}} "c":@{}"x3"|-*{\scriptstyle{pb}} "z3":@{}"b4"_-*{\scriptstyle{dpb}}} \]
%
\begin{lem}\label{lem:poly-comp-subdivided-composite}
The morphisms $\varepsilon_{(Y,q,r,s)}$ just described are the components of the counit of an adjunction $\tn{res}_n \ladj p_n \cdot (-)$.
\end{lem}
\begin{proof}
Let $(Y',q',r',s')$ be a subdivided composite over $(p_i)_{1{\leq}i{\leq}n}$, then for $t$ as in
\[ \xygraph{!{0;(2,0):(0,.5)::} {\tn{res}_n (p_n \cdot (Y,q,r,s))}="p1" [r(2)] {(Y,q,r,s)}="p2" [dl] {\tn{res}_n (Y',q',r',s')}="p3" "p1":"p2"^-{\varepsilon_{(Y,q,r,s)}}:@{<-}"p3"^-{t}:@{.>}"p1"^-{\tn{res}_n(t')}} \]
we must give $t'$ unique so that the above triangle commutes. The following commutative diagram assembles this given data in the case $n=4$.
\[ \xygraph{!{0;(.75,0):(0,1)::} {\bullet}="x0" [r] {\bullet}="a1" [r] {\bullet}="b1" [r] {\bullet}="x1" [r] {\bullet}="a2" [r] {\bullet}="b2" [r] {\bullet}="x2" [r] {\bullet}="a3" [r] {\bullet}="b3" [r] {\bullet}="x3" "a1" [u] {\bullet}="y0" [r(2)] {\bullet}="y1" [r(3)] {\bullet}="y2" [r(2)] {\bullet}="y3"
"x3" [r] {\bullet}="a4" [r] {\bullet}="b4" [r] {\bullet}="x4"
"y3" [u(.5)r] {\bullet}="c" [u(.5)r] {\bullet}="z3" [r] {\bullet}="z4" 
"x0":@{<-}"a1"_-{p_{11}}:"b1"_-{p_{12}}:"x1"_-{p_{13}}:@{<-}"a2"_-{p_{21}}:"b2"_-{p_{22}}:"x2"_-{p_{23}}:@{<-}"a3"_-{p_{31}}:"b3"_-{p_{32}}:"x3"_-{p_{33}}:@{<-}"a4"_-{p_{41}}:"b4"_-{p_{42}}:"x4"_-{p_{43}}
"y0":"y1"|-{q_{21}}="bz0":"y2"|-{q_{22}}="bz1":"y3"|-{q_{23}}="bz2"
"bz0" [u] {\bullet}="z0" "bz1" [u] {\bullet}="z1" "bz2" [u] {\bullet}="z2"
"y0":"a1"|-{r_1} "y1"(:"b1"|-{s_1},:"a2"|-{r_2}) "y2"(:"b2"|-{s_2},:"a3"|-{r_3}) "y3":"b3"|-{s_3} "y0":"x0"|-{q_1} "y3":"x3"|-{q_3}
"c"(:"y3",:"a4",:@{<-}"z3":"z4"^-{(p_4{\cdot}q)_{24}}(:"b4",:"x4"|-{(p_4{\cdot}q)_3}))
"z0":"y0"|-{\varepsilon_0} "z1":"y1"|-{\varepsilon_1} "z2":"y2"|-{\varepsilon_2}
"z0":"z1"^-{(p_4{\cdot}q)_{21}}:"z2"^-{(p_4{\cdot}q)_{22}}:"z3"^-{(p_4{\cdot}q)_{23}}
"a1" [r(.5)u(.5)] {\scriptstyle{pb}} "a2" [r(.5)u(.5)] {\scriptstyle{pb}} "a3" [r(.5)u(.5)] {\scriptstyle{pb}}
"z0":@{}"y1"|-*{\scriptstyle{pb}} "z1":@{}"y2"|-*{\scriptstyle{pb}} "z2":@{}"y3"|-*{\scriptstyle{pb}} "c":@{}"x3"|-*{\scriptstyle{pb}} "z3":@{}"b4"_-*{\scriptstyle{dpb}}
"z0" [u] {\bullet}="w0" "z1" [u] {\bullet}="w1" "z2" [u] {\bullet}="w2" "z3" [ul] {\bullet}="w3" "z4" [u] {\bullet}="w4"
"w0":"w1"^-{q_{21}'}:"w2"^-{q_{22}'}:"w3"^-{q_{23}'}:"w4"^-{q_{24}'}
"w0":@{.>}@/_{1pc}/"y0"|-{t_0} "w1":@{.>}@/_{1pc}/"y1"|(.75){t_1} "w2":@{.>}@/_{1pc}/"y2"|(.75){t_2} "w3":@{.>}@/_{1pc}/"y3"|(.75){t_3} "w4":@/^{1pc}/"x4"^-{q_3'} "w3":@{.>}"a4"|(.55){r_4'} "w4":@{.>}@/_{1pc}/"b4"|-{s_4'}} \]
Since $q_{n-1}t_{n-1}=p_{n1}r_n'$ one induces $u:Y_{n-1}' \to C$ using the defining pullback of $C$, and then one induces $t_{n-1}':Y_{n-1}' \to Y_{n-1}$ and $t_n':Y_n' \to Y_n$ from the maps $u$ and $s_n'$ using the distributivity pullback. The rest of the $t_i'$ are induced inductively as follows. For $0 < i < n$ given $t_i':Y_i' \to Y_i$, one induces $t_{i-1}'$ using the maps $t_{i-1}$ and $t_i'$ and the pullback which defines $(p_n \cdot Y)_{i-1}$. By construction the $t_i'$ are the components of the required unique map $t'$.
\end{proof}
For $(Y,q,r,s)$ a subdivided composite over $(p_i)_{1{<}i{\leq}n}$, we construct the subdivided composite
\[ (Y,q,r,s) \cdot p_1 := (Y \cdot p_1,q \cdot p_1,r \cdot p_1,s \cdot p_1) \]
over $(p_i)_{1{\leq}i{\leq}n}$ as follows. First one takes the pullback on the left, then for $0<i<n$ the distributivity pullbacks as in the middle,
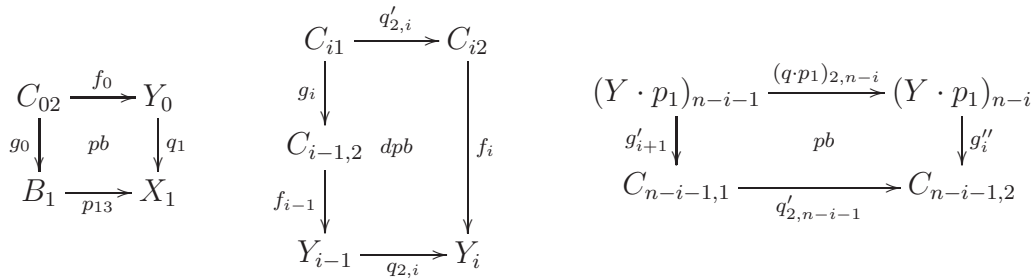
\[ \xygraph{{\xybox{\xygraph{!{0;(1.25,0):(0,.8)::} {C_{02}}="tl" [r] {Y_0}="tr" [d] {X_1}="br" [l] {B_1}="bl" "tl":"tr"^-{f_0}:"br"^-{q_1}:@{<-}"bl"^-{p_{13}}:@{<-}"tl"^-{g_0} "tl":@{}"br"|-*{\scriptstyle{pb}}}}}
[r(3)]
{\xybox{\xygraph{!{0;(1.5,0):(0,.75)::} {C_{i1}}="tl" [r] {C_{i2}}="tr" [d(2)] {Y_i}="br" [l] {Y_{i-1}}="bl" [u] {C_{i-1,2}}="ml" "tl":"tr"^-{q_{2,i}'}:"br"^-{f_i}:@{<-}"bl"^-{q_{2,i}}:@{<-}"ml"^-{f_{i-1}}:@{<-}"tl"^-{g_i} "tl":@{}"br"|-*{\scriptstyle{dpb}}}}}
[r(4.5)]
{\xybox{\xygraph{!{0;(3,0):(0,.3333)::} {(Y \cdot p_1)_{n-i-1}}="tl" [r] {(Y \cdot p_1)_{n-i}}="tr" [d] {C_{n-i-1,2}}="br" [l] {C_{n-i-1,1}}="bl" "tl":"tr"^-{(q \cdot p_1)_{2,n-i}}:"br"^-{g_i''}:@{<-}"bl"^-{q_{2,n-i-1}'}:@{<-}"tl"^-{g_{i+1}'} "tl":@{}"br"|-*{\scriptstyle{pb}}}}}} \]
and then for $0<i<n$ one takes the pullbacks as on the right in the previous display, setting $(q \cdot p_1)_{2,n}=q_{2,n-1}'$, $g_1''=g_{n-1}$, $g_{i+1}''=g_{n-i+1}g_{i+1}'$ for $i+1<n$, $g_{n-1}''=g_0g_1g_{n-1}'$ and $q_{02}'=p_{12}$. Finally one defines
\[ \begin{array}{lccr} {(q \cdot p_1)_1 = p_{11}g_n'} &&& {(q \cdot p_1)_3 = q_3f_{n-1}.} \end{array} \]
In the case $n=4$ one obtains a diagram like this:
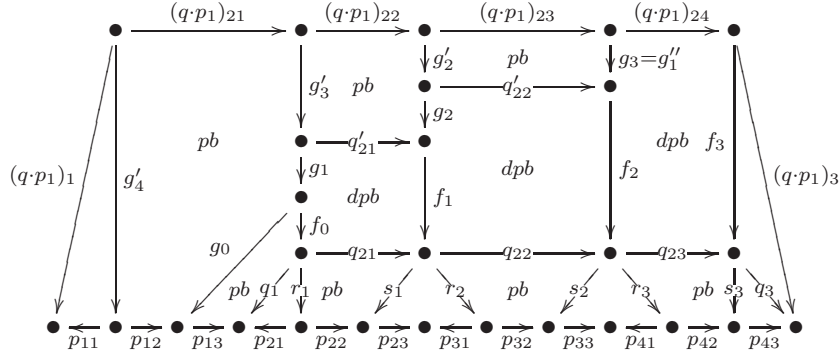
\[ \xygraph{!{0;(.65,0):(0,1.2)::} {\bullet}="x0" [r] {\bullet}="a1" [r] {\bullet}="b1" [r] {\bullet}="x1" [r] {\bullet}="a2" [r] {\bullet}="b2" [r] {\bullet}="x2" [r] {\bullet}="a3" [r] {\bullet}="b3" [r] {\bullet}="x3" [r] {\bullet}="a4" [r] {\bullet}="b4" [r] {\bullet}="x4" "a2" [u] {\bullet}="y1" [r(2)] {\bullet}="y2" [r(3)] {\bullet}="y3" [r(2)] {\bullet}="y4"
"x0":@{<-}"a1"_-{p_{11}}:"b1"_-{p_{12}}:"x1"_-{p_{13}}:@{<-}"a2"_-{p_{21}}:"b2"_-{p_{22}}:"x2"_-{p_{23}}:@{<-}"a3"_-{p_{31}}:"b3"_-{p_{32}}:"x3"_-{p_{33}}:@{<-}"a4"_-{p_{41}}:"b4"_-{p_{42}}:"x4"_-{p_{43}} "y1":"y2"|-{q_{21}}:"y3"|-{q_{22}}:"y4"|-{q_{23}} "y1":"a2"|-{r_1} "y2"(:"b2"|-{s_1},:"a3"|-{r_2}) "y3"(:"b3"|-{s_2},:"a4"|-{r_3}) "y4":"b4"|-{s_3} "y1":"x1"|-{q_1} "y4":"x4"|-{q_3} "a2" [r(.5)u(.5)] {\scriptstyle{pb}} "a3" [r(.5)u(.5)] {\scriptstyle{pb}} "a4" [r(.5)u(.5)] {\scriptstyle{pb}}
"y1" [u(.75)] {\bullet}="c11" [u(.75)] {\bullet}="c12" [u(1.5)] {\bullet}="c14"
"y2" [u(1.5)] {\bullet}="c22" [u(.75)] {\bullet}="c23" [u(.75)] {\bullet}="c24"
"y3" [u(2.25)] {\bullet}="c33" [u(.75)] {\bullet}="c34" "y4" [u(3)] {\bullet}="c44" "a1" [u(4)] {\bullet}="c04"
"c04":"a1"^-{g_4'} "c14":"c12"^-{g_3'}:"c11"^-{g_1}:"y1"^-{f_0} "c24":"c23"^-{g_2'}:"c22"^-{g_2}:"y2"^-{f_1} "c34":"c33"^-{g_3=g_1''}:"y3"^-{f_2} "c44":"y4"_-{f_3}
"c04":"c14"^-{(q \cdot p_1)_{21}}:"c24"^-{(q \cdot p_1)_{22}}:"c34"^-{(q \cdot p_1)_{23}}:"c44"^-{(q \cdot p_1)_{24}} "c23":"c33"|-{q_{22}'} "c12":"c22"|-{q_{21}'} "c11":"b1"_-{g_0} "c04":"x0"_-{(q \cdot p_1)_1} "c44":"x4"^-{(q \cdot p_1)_3}
"x1" [u(.5)] {\scriptstyle{pb}} "c12" [l(1.5)] {\scriptstyle{pb}} "c12":@{}"y2"|-*{\scriptstyle{dpb}} "c23":@{}"y3"|-*{\scriptstyle{dpb}} "c34":@{}"y4"|-*{\scriptstyle{dpb}} "c14":@{}"c22"|-*{\scriptstyle{pb}} "c24":@{}"c33"|-*{\scriptstyle{pb}}} \]
The equations $\varepsilon_0'=f_0g_1g_3'$, $\varepsilon_{n-1}'=f_{n-1}$ and $\varepsilon_i'=f_ig_{n-i-1}''$ for $0<i<n-1$, define the components of the morphism
\[ \varepsilon_{Y,q,r,s}' : \tn{res}_0 ((Y,q,r,s) \cdot p_1) \longrightarrow (Y,q,r,s).  \]
\begin{lem}\label{lem:subdivided-composite-comp-poly}
The morphisms $\varepsilon_{Y,q,r,s}'$ just described are the components of the counit of an adjunction $\tn{res}_0 \ladj (-) \cdot p_1$.
\end{lem}
\begin{proof}
Given $(Y,q,r,s)$ in $\tn{SdC}(p_i)_{1{\leq}i{\leq}n}$ and $t$ as in
\[ \xygraph{!{0;(2.5,0):(0,.4)::} {\tn{res}_0 ((Y,q,r,s) \cdot p_1)}="tl" [r(2)] {(Y,q,r,s)}="tr" [dl] {\tn{res}_0(Y',q',r',s')}="b" "tl":"tr"^-{\varepsilon_{Y,q,r,s}'}:@{<-}"b"^-{t}:@{.>}"tl"^-{\tn{res}_0(t')}} \]
we must exhibit $t'$ as shown unique so that the above diagram commutes. In the case $n=4$ the data $(Y',q',r',s')$ and $t$ fit into the following diagram:
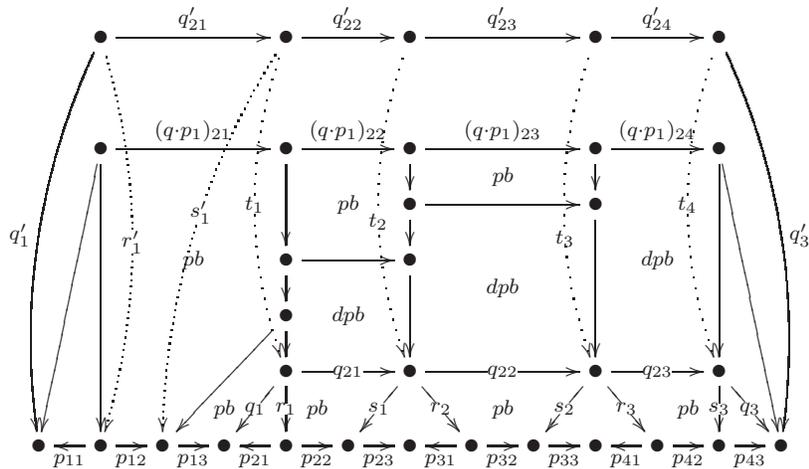
\[ \xygraph{!{0;(.65,0):(0,1.2)::} {\bullet}="x0" [r] {\bullet}="a1" [r] {\bullet}="b1" [r] {\bullet}="x1" [r] {\bullet}="a2" [r] {\bullet}="b2" [r] {\bullet}="x2" [r] {\bullet}="a3" [r] {\bullet}="b3" [r] {\bullet}="x3" [r] {\bullet}="a4" [r] {\bullet}="b4" [r] {\bullet}="x4" "a2" [u] {\bullet}="y1" [r(2)] {\bullet}="y2" [r(3)] {\bullet}="y3" [r(2)] {\bullet}="y4"
"x0":@{<-}"a1"_-{p_{11}}:"b1"_-{p_{12}}:"x1"_-{p_{13}}:@{<-}"a2"_-{p_{21}}:"b2"_-{p_{22}}:"x2"_-{p_{23}}:@{<-}"a3"_-{p_{31}}:"b3"_-{p_{32}}:"x3"_-{p_{33}}:@{<-}"a4"_-{p_{41}}:"b4"_-{p_{42}}:"x4"_-{p_{43}} "y1":"y2"|-{q_{21}}:"y3"|-{q_{22}}:"y4"|-{q_{23}} "y1":"a2"|-{r_1} "y2"(:"b2"|-{s_1},:"a3"|-{r_2}) "y3"(:"b3"|-{s_2},:"a4"|-{r_3}) "y4":"b4"|-{s_3} "y1":"x1"|-{q_1} "y4":"x4"|-{q_3} "a2" [r(.5)u(.5)] {\scriptstyle{pb}} "a3" [r(.5)u(.5)] {\scriptstyle{pb}} "a4" [r(.5)u(.5)] {\scriptstyle{pb}}
"y1" [u(.75)] {\bullet}="c11" [u(.75)] {\bullet}="c12" [u(1.5)] {\bullet}="c14"
"y2" [u(1.5)] {\bullet}="c22" [u(.75)] {\bullet}="c23" [u(.75)] {\bullet}="c24"
"y3" [u(2.25)] {\bullet}="c33" [u(.75)] {\bullet}="c34" "y4" [u(3)] {\bullet}="c44" "a1" [u(4)] {\bullet}="c04"
"c04":"a1"^-{} "c14":"c12"^-{}:"c11"^-{}:"y1"^-{} "c24":"c23"^-{}:"c22"^-{}:"y2"^-{} "c34":"c33"^-{}:"y3"^-{} "c44":"y4"_-{}
"c04":"c14"^-{(q \cdot p_1)_{21}}:"c24"^-{(q \cdot p_1)_{22}}:"c34"^-{(q \cdot p_1)_{23}}:"c44"^-{(q \cdot p_1)_{24}} "c23":"c33"|-{} "c12":"c22"|-{} "c11":"b1"_-{} "c04":"x0"|-{} "c44":"x4"|-{}
"x1" [u(.5)] {\scriptstyle{pb}} "c12" [l(1.5)] {\scriptstyle{pb}} "c12":@{}"y2"|-*{\scriptstyle{dpb}} "c23":@{}"y3"|-*{\scriptstyle{dpb}} "c34":@{}"y4"|-*{\scriptstyle{dpb}} "c14":@{}"c22"|-*{\scriptstyle{pb}} "c24":@{}"c33"|-*{\scriptstyle{pb}}
"c04" [u(1.5)] {\bullet}="z0" "c14" [u(1.5)] {\bullet}="z1" "c24" [u(1.5)] {\bullet}="z2" "c34" [u(1.5)] {\bullet}="z3" "c44" [u(1.5)] {\bullet}="z4"
"x0":@{<-}@/^{1pc}/"z0"^-{q_1'}:"z1"^-{q_{21}'}:"z2"^-{q_{22}'}:"z3"^-{q_{23}'}:"z4"^-{q_{24}'}:@/^{1pc}/"x4"^-{q_3'}
"z1":@{.>}@/_{1pc}/"y1"|-{t_1} "z2":@{.>}@/_{1pc}/"y2"|(.55){t_2} "z3":@{.>}@/_{1pc}/"y3"|(.6){t_3} "z4":@{.>}@/_{1pc}/"y4"|-{t_4} "z0":@{.>}@/^{1pc}/"a1"|-{r_1'} "z1":@{.>}@/_{1pc}/"b1"|(.45){s_1'}} \]
Using the pullback that defines $C_{02}$ and the maps $s_1'$ and $t_1$, one induces $Y_1' \to C_{02}$. Using the distributivity pullbacks one induces successively the morphisms $Y_i' \to C_{i1}$ and $Y_{i+1}' \to C_{i2}$ for $0<i<n$. In the case $i=n-1$ we denote these maps as $t_{n-1}'$ and $t_n'$ respectively. The components $t_i'$ for $0 \leq i < n-1$ are then induced from this data and the pullbacks that define the objects $(Y \cdot p_1)_i$. By construction the $t_i'$ are the components of the required unique map $t'$.
\end{proof}
\begin{prop}\label{prop:terminal-subdivided-composites}
For any composable sequence $(p_i)_{1{\leq}i{\leq}n}$ of polynomials in a category $\ca E$ with pullbacks, the category $\tn{SdC}(p_i)_i$ has a terminal object.
\end{prop}
\begin{proof}
We proceed by induction on $n$. In the case $n=0$, observe that a subdivided composite consists just of the data $Y_0$, $q_1:Y_0 \to X_0$ and $q_3:Y_0 \to X_0$, and that $\tn{SdC}()$ is the category $\Span{\ca E}(X_0,X_0)$ of endospans of $X_0$. The identity endospan is terminal. For the inductive step apply either of the functors $p_n \cdot (-)$ or $(-) \cdot p_1$ which as right adjoints, preserve terminal objects.
\end{proof}
\begin{defn}\label{def:composition-of-polynomials}
Let $\ca E$ be a category with pullbacks. A \emph{composite} of a composable sequence $(p_i)_{1{\leq}i{\leq}n}$ of polynomials in $\ca E$, is defined to be the associated polynomial of a terminal object in the category $\tn{SdC}(p_i)_i$. When such a composite has been chosen, it is denoted as $p_n \comp ... \comp p_1$.
\end{defn}
Let us consider now some degenerate cases of Definition \ref{def:composition-of-polynomials}.
\begin{itemize}
\item $n=0$: Choosing identity spans as terminal nullary subdivided composites (see the proof of Proposition \ref{prop:terminal-subdivided-composites}), nullary composition of polynomials gives polynomials whose constituent maps are all identities. That is,
\[ \xygraph{{X}="p1" [r] {X}="p2" [r] {X}="p3" [r] {X}="p4" "p1":@{<-}"p2"^-{1_X}:"p3"^-{1_X}:"p4"^-{1_X}} \]
is the ``identity polynomial on $X$'' as one would hope.
\item $n=1$: One may identify $\tn{SdC}(p)$ as the slice $\Polyc{\ca E}(X_0,X_1)/p$, and thus choose $1_p$ as the terminal unary subdivided composite over $(p)$. Thus the unary composite of a given polynomial $p$ is just $p$.
\item $n=2$: applying $p_2 \cdot (-)$ to $p_1$, or $(-) \cdot p_1$ to $p_2$, gives the same subdivided composite, namely
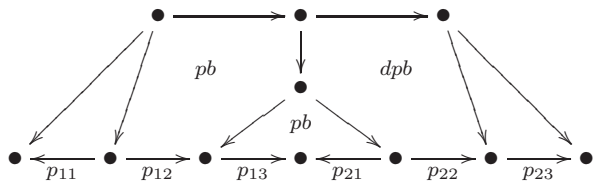
\[ \xygraph{!{0;(1,0):(0,.75)::} {\bullet}="b1" [r] {\bullet}="b2" [r] {\bullet}="b3" [r] {\bullet}="b4" [r] {\bullet}="b5" [r] {\bullet}="b6" [r] {\bullet}="b7" "b4" [u] {\bullet}="p1" [u] {\bullet}="dl" ([r(1.5)] {\bullet}="dr", [l(1.5)] {\bullet}="p2")
"b1":@{<-}"b2"_-{p_{11}}:"b3"_-{p_{12}}:"b4"_-{p_{13}}:@{<-}"b5"_-{p_{21}}:"b6"_-{p_{22}}:"b7"_-{p_{23}} "dl":"p1"_-{}(:"b3"_-{},:"b5"^-{}) "b2":@{<-}"p2"_-{}:"dl"_-{}:"dr"_-{}:"b6"^(.7){} "b1":@{<-}"p2"^-{} "dr":"b7"^-{}
"b3" [u(1.25)] {\scriptstyle{pb}} "b5" [u(1.25)] {\scriptstyle{dpb}} "b4" [u(.5)] {\scriptstyle{pb}}} \]
which is terminal by the case $n=1$ and since the functors $p_2 \cdot (-)$ and $(-) \cdot p_1$, as right adjoints, preserve terminal objects. Thus the associated composite of the above is the binary composite $p_2 \comp p_1$, and this agrees with the binary composition of polynomials given in \cite{GambinoKock-PolynomialFunctors}.
\end{itemize}
\begin{lem}\label{lem:for-poly-composition}
Let $n>0$ and $(p_i)_{1{\leq}i{\leq}n}$ be a composable sequence of polynomials in a category $\ca E$ with pullbacks. Then one has canonical isomorphisms
\[ \xygraph{{\xybox{\xygraph{!{0;(3,0):(0,.3333)::} {\tn{SdC}(p_i)_{1{\leq}i{<}n}}="tl" [r] {\Polyc{\ca E}(X_0,X_{n-1})}="tr" [d] {\Polyc{\ca E}(X_0,X_n)}="br" [l] {\tn{SdC}(p_i)_{1{\leq}i{\leq}n}}="bl" "tl":"tr"^-{\tn{ass}}:"br"^-{p_n \comp (-)}:@{<-}"bl"^-{\tn{ass}}:@{<-}"tl"^-{p_n \cdot (-)} "tl":@{}"br"|-*{\iso}}}}
[r(6)]
{\xybox{\xygraph{!{0;(3,0):(0,.3333)::} {\tn{SdC}(p_i)_{1{<}i{\leq}n}}="tl" [r] {\Polyc{\ca E}(X_1,X_n)}="tr" [d] {\Polyc{\ca E}(X_0,X_n)}="br" [l] {\tn{SdC}(p_i)_{1{\leq}i{\leq}n}}="bl" "tl":"tr"^-{\tn{ass}}:"br"^-{(-) \comp p_1}:@{<-}"bl"^-{\tn{ass}}:@{<-}"tl"^-{(-) \cdot p_1} "tl":@{}"br"|-*{\iso}}}}} \]
\end{lem}
\begin{proof}
The canonical isomorphism on the left follows from the definitions and the elementary properties of pullbacks. The canonical isomorphism on the right follows from the definitions, and iterated application of Lemma \ref{lem:distpb-composition-cancellation}.
\end{proof}
In order to make explicit the horizontal composition of 2-cells in $\Polyc {\ca E}$ we consider a horizontally composable sequence of morphisms of polynomials of length $n$, that is to say diagrams
\[ \xygraph{!{0;(1.5,0):(0,.5)::} {X_{i-1}}="p0" [ur] {A_i}="p1" [r] {B_i}="p2" [dr] {X_i}="p3" [dl] {B'_i}="p4" [l] {A'_i}="p5" "p0":@{<-}"p1"^-{p_{i1}}:"p2"^-{p_{i2}}:"p3"^-{p_{i3}}:@{<-}"p4"^-{q_{i3}}:@{<-}"p5"^-{q_{i2}}:"p0"^-{q_{i1}} "p1":"p5"_{f_{0i}} "p2":"p4"^{f_{1i}} "p1":@{}"p4"|-{\tn{pb}}} \]
in $\ca E$, for $0 < i \leq n$. We denote such a sequence as $(f_{0i},f_{1i})_i : (p_i)_i \to (q_i)_i$ since it is a morphism of the category $\prod_{i=1}^n \Polyc{\ca E}(X_{i-1},X_i)$. The process of vertically stacking subdivided composites and their morphisms on top of $(f_{0i},f_{1i})_i$ gives a functor
\[ \tn{SdC}(f_{0i},f_{1i})_i : \tn{SdC}(p_i)_i \longrightarrow \tn{SdC}(q_i)_i. \]
The assignation $(f_{0i},f_{1i})_i \mapsto \tn{SdC}(f_{0i},f_{1i})_i$ is functorial, and natural in the evident sense with respect to the restriction and associated polynomial functors defined above. For any choice $t_1$ and $t_2$ of terminal object of $\tn{SdC}(p_i)_i$ and $\tn{SdC}(q_i)_i$ respectively, one has composites
\[ \begin{array}{lccr} {p_n \comp ... \comp p_1 = \tn{ass}(t_1)} &&& {q_n \comp ... \comp q_1 = \tn{ass}(t_2)} \end{array} \]
by Definition \ref{def:composition-of-polynomials}, and a unique morphism $u_{t_1,t_2} : \tn{SdC}(f_{0i},f_{1i})_i(t_1) \to t_2$.
\begin{defn}\label{defn:hor-comp-2-cells}
Let $\ca E$ be a category with pullbacks and $(f_{0i},f_{1i})_i : (p_i)_i \to (q_i)_i$ be a horizontally composable sequence of polynomial morphisms of length $n$. Then in the context just described, the 2-cell
\[ f_n \comp ... \comp f_1 : p_n \comp ... \comp p_1 \longrightarrow q_n \comp ... \comp q_1  \]
is defined to be $\tn{ass}(u_{t_1,t_2})$.
\end{defn}
In the case $n = 2$ the original data and the chosen terminal subdivided composites comprise the solid parts of the diagram,
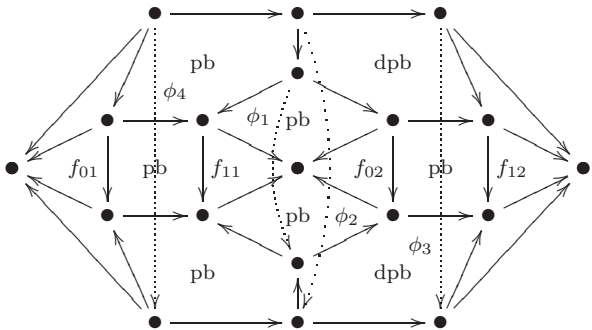
\[ \xygraph{!{0;(1,0):(0,.5)::} {\bullet}="p0" [ur] {\bullet}="p1" [r] {\bullet}="p2" [dr] {\bullet}="p3" [ur] {\bullet}="p4" [r] {\bullet}="p5" [dr] {\bullet}="p6" {\bullet}="q6" [dl] {\bullet}="q5" [l] {\bullet}="q4" [ul] {\bullet}="q3" [dl] {\bullet}="q2" [l] {\bullet}="q1" [ul] {\bullet}="q0"
"p0":@{<-}"p1"^-{}:"p2"^-{}:"p3"^-{}:@{<-}"p4"^-{}:"p5"^-{}:"p6"^-{}:@{<-}"q5"^-{}:@{<-}"q4"^-{}:"q3"^-{}:@{<-}"q2"^-{}:@{<-}"q1"^-{}:"p0"^-{}
"p1":"q1"_-{f_{01}} "p2":"q2"^-{f_{11}} "p4":"q4"_-{f_{02}} "p5":"q5"^-{f_{12}} "p1":@{}"q2"|-{\tn{pb}} "p4":@{}"q5"|-{\tn{pb}}
"p3" [u(2)] {\bullet}="p7" [u(1.25)] {\bullet}="p8" ([r(1.5)] {\bullet}="p9",[l(1.5)] {\bullet}="p10")
"p8"(:"p7"(:"p2",:"p4"),:"p9"(:"p5",:"p6")) "p10"(:"p0",:"p1",:"p8") "p3" [u] {\scriptsize{\tn{pb}}} ([u(1.2)l] {\scriptsize{\tn{pb}}}, [u(1.2)r] {\scriptsize{\tn{dpb}}})
"q3" [d(2)] {\bullet}="q7" [d(1.25)] {\bullet}="q8" ([r(1.5)] {\bullet}="q9",[l(1.5)] {\bullet}="q10")
"q8"(:"q7"(:"q2",:"q4"),:"q9"(:"q5",:"q6") "q10"(:"q0",:"q1",:"q8") "q3" [d] {\scriptsize{\tn{pb}}} ([d(1.2)l] {\scriptsize{\tn{pb}}}, [d(1.2)r] {\scriptsize{\tn{dpb}}})
"p7":@{.>}@/_{.8pc}/"q7"_(.25){\phi_1} "p8":@{.>}@/^{1pc}/"q8"^(.65){\phi_2} "p9":@{.>}"q9"_(.75){\phi_3} "p10":@{.>}"q10"^(.25){\phi_4}
} \]
and one then induces $\phi_1$ using the pullback defining its codomain, $\phi_2$ and $\phi_3$ are then induced by the universal property of the bottom distributivity pullback, and finally $\phi_4$ is induced by the bottom pullback. One verifies easily that $(\phi_4,\phi_2,\phi_3)$ a morphism of subdivided composites, thus it is $u_{t_1,t_2}$, and so by Definition \ref{defn:hor-comp-2-cells} the composite $f_2 \comp f_1$ is given by $(\phi_4,\phi_3)$.
\begin{thm}\label{thm:bicat-of-polynomials}
Let $\ca E$ be a category with pullbacks. One has a bicategory $\Polyc{\ca E}$, whose objects are those of $\ca E$, whose hom from $X$ to $Y$ is $\Polyc{\ca E}(X,Y)$, horizontal composition of 1-cells is given by Definition \ref{def:composition-of-polynomials}, and horizontal composition of 2-cells is given by Definition \ref{defn:hor-comp-2-cells}.
\end{thm}
\begin{proof}
By induction on $n$, using the fact that the functors $p_n \cdot (-)$ and $(-) \cdot p_1$ preserve terminal objects, and Lemma \ref{lem:for-poly-composition}, it follows that any iterated binary composite of polynomials of length $n$, is a composite in the sense of Definition \ref{def:composition-of-polynomials}. That is, such an iterated composite is the associated polynomial of a terminal subdivided polynomial, which arises from the composable sequence of polynomials that participates in the given iterated binary composite. Hence between any two alternative brackettings of a given composite, there is a unique isomorphism of their underlying subdivided composites, giving rise to a ``coherence'' isomorphism of the composites themselves upon application of ``$\tn{ass}$''. Any diagram of such coherence isomorphisms must commute, since it is the image by the appropriate ``$\tn{ass}$'' functor, of a diagram whose vertices are all terminal subdivided composites. Thanks to our conventions regarding chosen pullbacks and chosen distributivity pullbacks of identities described in Section \ref{ssec:dpb}, the unit coherence isomorphisms here turn out to be identities.

The functoriality of horizontal composition comes from the functoriality of $(f_{0i},f_{1i})_i \mapsto \tn{SdC}(f_{0i},f_{1i})_i$ and the naturality of $\tn{SdC}(f_{0i},f_{1i})_i$ with respect to the ``$\tn{ass}$'' functors. It remains to verify the naturality of the coherence isomorphisms identified in the previous paragraph. To this end we suppose that a horizontally composable sequence $(f_{0i},f_{1i})_i : (p_i)_i \to (q_i)_i$ of morphisms of polynomials of length $n$, and binary brackettings $\beta_1$ and $\beta_2$ of $n$ things is given. Let us denote by $t_{\beta_1}(p_i)_i$, $t_{\beta_2}(p_i)_i$, $t_{\beta_1}(q_i)_i$ and $t_{\beta_2}(q_i)_i$ the terminal subdivided composites witnessing the iterated binary composites of $(p_i)_i$ and $(q_i)_i$ via the given brackettings. In $\tn{SdC}(q_i)_i$ one has the diagram
\[ \xygraph{!{0;(4,0):(0,.25)::} {\tn{SdC}(f_{0i},f_{1i})_i(t_{\beta_1}(p_i)_i)}="p0" [r] {\tn{SdC}(f_{0i},f_{1i})_i(t_{\beta_2}(p_i)_i)}="p1" [d] {t_{\beta_2}(q_i)_i}="p2" [l] {t_{\beta_1}(q_i)_i}="p3" "p0":@<1ex>"p1"^-{}:"p2"^-{}:@{<-}@<-1ex>"p3"^-{}:@{<-}"p0"^-{}:@<-1ex>@{<-}"p1" "p2":@<1ex>"p3"} \]
in which the top horizontal arrows are the effect of applying $\tn{SdC}(f_{0i},f_{1i})_i$ to the unique morphisms, and the other morphisms are determined uniquely and both squares commute because $t_{\beta_1}(q_i)_i$ and $t_{\beta_2}(q_i)_i$ are terminal. Applying $\tn{ass} : \tn{SdC}(q_i)_i \to \Polyc{\ca E}(X_0,X_n)$ to this diagram gives the squares witnessing the naturality of the coherence morphisms.
\end{proof}
A span in $\ca E$ as on the left
\[ \xygraph{{\xybox{\xygraph{{X}="p0" [r] {Z}="p1" [r] {Y}="p2" "p0":@{<-}"p1"^-{s}:"p2"^-{t}}}}
[r(4)]
{\xybox{\xygraph{{X}="p0" [r] {Z}="p1" [r] {Z}="p2" [r] {Y}="p3" "p0":@{<-}"p1"^-{s}:"p2"^-{1_Z}:"p3"^-{t}}}}} \]
may be identified as a polynomial in which the middle map is an identity as on the right. Polynomial composition of spans coincides exactly with span composition, giving us a strict inclusion
\[ \Span{\ca E} \hookrightarrow \Polyc{\ca E} \]
of bicategories which is the identity on objects and locally fully faithful. For a given map $f:X \to Y$ in $\ca E$, we denote by $\lft{f}:X \to Y$ and $\rgt{f}:Y \to X$ the polynomials
\[ \xygraph{{\xybox{\xygraph{{X}="p1" [r] {X}="p2" [r] {X}="p3" [r] {Y}="p4" "p1":@{<-}"p2"^-{1}:"p3"^-{1}:"p4"^-{f}}}} [r(4)] {\xybox{\xygraph{{Y}="p1" [r] {X}="p2" [r] {X}="p3" [r] {X}="p4" "p1":@{<-}"p2"^-{f}:"p3"^-{1}:"p4"^-{1}}}}} \]
respectively. These are spans, it is well known that one has $\lft{f} \ladj \rgt{f}$ and that this is part of the basic data of the proarrow equipment $(\ca E,\Span{\ca E})$ \cite{Wood-Proarrows-I, Wood-Proarrows-II}. By the above strict inclusion, this extends to another proarrow equipment $(\ca E,\Polyc{\ca E})$, and all this at the generality of a category $\ca E$ with pullbacks. It is worth noting that polynomial composites of the form $\lft{f} \comp p$ and $q \comp \rgt{g}$ are particularly easy, these being
\[ \xygraph{{\xybox{\xygraph{{\bullet}="p1" [r] {\bullet}="p2" [r] {\bullet}="p3" [r] {\bullet}="p4" "p1":@{<-}"p2"^-{p_1}:"p3"^-{p_2}:"p4"^-{fp_3}}}} [r(4)] {\xybox{\xygraph{{\bullet}="p1" [r] {\bullet}="p2" [r] {\bullet}="p3" [r] {\bullet}="p4" "p1":@{<-}"p2"^-{gq_1}:"p3"^-{q_2}:"p4"^-{q_3}}}}} \]
respectively.

The homs of $\Polyc{\ca E}$ interact well with the slices of $\ca E$. For all $X$ and $Y$ one has obvious forgetful functors
\[ \xygraph{!{0;(2,0):} {\ca E/X}="l" [r] {\Polyc{\ca E}(X,Y)}="m" [r] {\ca E/Y}="r" "l":@{<-}"m"^-{l_{X,Y}}:"r"^-{r_{X,Y}}} \]
and we refer to these as the \emph{left} and \emph{right projections of the homs of} $\Polyc{\ca E}$. From the above descriptions of composites of the form $\lft{f} \comp p$ and $q \comp \rgt{g}$, one obtains immediately the sense in which these forgetful functors are natural.
\begin{lem}\label{lem:naturality-Poly-hom-projections}
For all $f:Y \to Z$ and $g:X \to W$ one has
\[ \begin{array}{rclccrcl} {\Sigma_gl_{X,Y}} & {=} & {l_{W,Y}((-) \comp \rgt{g})} &&& {\Sigma_fr_{X,Y}} & {=} & {r_{X,Z}(\lft{f} \comp (-))} \\ {l_{X,Y}} & {=} & {l_{X,Z}(\lft{f} \comp (-))} &&& {r_{X,Y}} & {=} & {r_{W,Y}((-) \comp \rgt{g})} \end{array} \]
\end{lem}

\subsection{Polynomial functors.}
\label{ssec:polynomial-functors}
Let $\ca E$ be a category with pullbacks. In this section we define a homomorphism of bicategories
\[ \begin{array}{lccr} {\PFun{\ca E} : \Polyc{\ca E} \longrightarrow \CAT} &&& {X \mapsto \ca E/X} \end{array} \]
with object map as indicated, in Theorem \ref{thm:polnomial-functor-homomorphism}. Given a polynomial $p:X \to Y$ in $\ca E$, the functor $\PFun{\ca E}(p):\ca E/X \to \ca E/Y$ is defined to be the composite $\Sigma_{p_3}\Pi_{p_2}\Delta_{p_1}$, which for the sake of brevity, will also be denoted as $p(-):\ca E/X \to \ca E/Y$. In more elementary terms the effect of $p(-)$ on an object $x:C \to X$ of $\ca E/X$ is described by the following commutative diagram:
\[ \xygraph{!{0;(1.5,0):(0,.5)::} {X}="p1" [r] {A}="p2" [r] {B}="p3" [r] {Y}="p4" "p1" [u] {C}="s" "p2" [u] {C_2}="mdpb" [u] {C_3}="tldpb" [r] {C_4}="trdpb"
"p1":@{<-}"p2"_-{p_1}:"p3"_-{p_2}:"p4"_-{p_3} "mdpb"(:"s"_-{}:"p1"_-{x},:"p2"^-{},:@{<-}"tldpb"_-{}:"trdpb"^-{}(:"p3"|-{},:"p4"^-{p(x)}|-{}="codeq"))
"mdpb" ([r(.5)] {\scriptstyle{dpb}}, [l(.5)d(.5)] {\scriptstyle{pb}})} \]
Similarly one may, by exploiting the universal property of the pullback and distributivity pullback in this description, induce the maps which provide the arrow map of $p(-)$. These explicit descriptions together with Lemma \ref{lem:naturality-Poly-hom-projections} enables us to catalogue all the ways one can use the composition of $\Polyc{\ca E}$ to describe the functor $p(-)$, and we record this in
\begin{lem}\label{lem:all-descriptions-associated-polynomial-functor}
Let $p:X \to Y$ be a polynomial in $\ca E$.
\begin{enumerate}
\item  Given $x:C \to X$ in $\ca E/X$, one has
\[ p(x) = r_{Z,Y}(p \comp \lft{x} \comp \rgt{g}) \]
for all $Z$ and $g:C \to Z$.
\item  Given $x_1:C_1 \to X$, $x_2:C_2 \to X$ and $h:C_1 \to C_2$ over $X$, one has
\[ p(h) = r_{Z,Y}(p \comp h' \comp \rgt{g}) \]
for all $Z$ and $g:C_2 \to Z$, where $h':x_2^{\bullet} \to x_1^{\bullet}\rgt{h}$ is the mate of the identity $x_2^{\bullet}\lft{h}=x_1^{\bullet}$ via $\lft{h} \ladj \rgt{h}$.
\end{enumerate}
\end{lem}
To prove Theorem \ref{thm:polnomial-functor-homomorphism} we exhibit an analogous result, Lemma \ref{lem:all-descriptions-2cell-map-P}, giving all the ways of expressing $\PFun{\ca E}$'s 2-cell map in terms of composition in $\Polyc {\ca E}$. Preliminary to this result we reconcile two ways of describing $\PFun{\ca E}$'s 2-cell map -- that given in \cite{GambinoKock-PolynomialFunctors} versus a direct description in terms of morphisms of induced from pullbacks and distributivity pullbacks, in Lemma \ref{lem:explicit-polymap->cart-transformation}. Lemmas \ref{lem:left-BC} and \ref{lem:right-BC} are preliminary to Lemma \ref{lem:explicit-polymap->cart-transformation}. The reader not interested in such technical details is encouraged to skip ahead to the statement of Theorem \ref{thm:polnomial-functor-homomorphism} below.

The description of $\PFun {\ca E}$'s 2-cell map given in \cite{GambinoKock-PolynomialFunctors}, is to associate to a given cartesian morphism $f:p \to q$ between polynomials from $X$ to $Y$, the following natural transformation
\begin{equation}\label{diag:2-cell-map-of-P-external}
\xygraph{!{0;(1.25,0):(0,.8)::} {\ca E/X}="d" [ru] {\ca E/A}="me" [r] {\ca E/B}="e" [dr] {\ca E/Y}="c" [dl] {\ca E/B'}="e2" [l] {\ca E/A'}="me2" "d":"me"^-{\Delta_{p_1}}:"e"^-{\Pi_{p_2}}:"c"|(.4){}="pd2"^-{\Sigma_{p_3}} "d":"me2"_-{\Delta_{q_1}}:"e2"|-{}="pc2"_-{\Pi_{q_2}}:"c"|(.4){}="pc2"_-{\Sigma_{q_3}} "me2":"me"|-{\Delta_{f_0}} "e2":"e"|-{\Delta_{f_1}} "d" [r(.6)] {\iso} [r(.9)] {\iso} "pd2":@{}"pc2"|(.35){}="d2"|(.65){}="c2" "d2":@{=>}"c2"}
\end{equation}
in which the isomorphism in the middle is a Beck-Chevalley isomorphism, and $\Sigma_{p_3}\Delta_{f_1} \to \Sigma_{q_3}$ is the mate of the identity via $\Sigma_{f_1} \ladj \Delta_{f_1}$. The advantage of this description is that the functoriality of the resulting hom functor
\[ (\PFun{\ca E})_{X,Y} : \Polyc{\ca E}(X,Y) \longrightarrow \CAT(\ca E/X,\ca E/Y) \]
is evident. We will show that the component at $x:C \to X$ of $(\PFun{\ca E})_{X,Y}(f)$ is given by the map $f_{4,x}$, constructed in
\begin{equation}\label{diag:2-cell-map-of-P-internal}
\xygraph{!{0;(1.5,0):(0,.375)::} {X}="X" ([ur] {A}="A" [r] {B}="B" [dr] {Y}="Y", [dr] {A'}="Ap" [r] {B'}="Bp") "X" [l] {C}="C" ([u(2)r] {C_2}="C2" [ur] {C_3}="C3" [r] {C_4}="C4",[d(2)r] {C'_2}="Cp2" [dr] {C'_3}="Cp3" [r] {C'_4}="Cp4")
"X":@{<-}"A"^-{p_1}:"B"^-{p_2}:"Y"^-{p_3} "X":@{<-}"Ap"_-{q_1}:"Bp"_-{q_2}:"Y"_-{q_3} "C":@{<-}"C2":@{<-}"C3":"C4"^-{}:"Y"^-{p(x)} "C":@{<-}"Cp2":@{<-}"Cp3":"Cp4"_-{}:"Y"_-{q(x)} "C":"X"|-{x} "C2":"A" "Cp2":"Ap" "C4":"B" "Cp4":"Bp" "A":"Ap"^-{f_0} "B":"Bp"^-{f_1} "C2":@{.>}@/_{1.5pc}/"Cp2"|(.35){f_{2,x}} "C3":@{.>}@/_{1.5pc}/"Cp3"|-{f_{3,x}} "C4":@{.>}@/_{1.5pc}/"Cp4"|(.6){f_{4,x}}
"C2" [d] {\scriptstyle{pb}} "Cp2" [u] {\scriptstyle{pb}} "C3" [r(.25)d] {\scriptstyle{dpb}} "Cp3" [r(.25)u] {\scriptstyle{dpb}} "A" [r(.5)d] {\scriptstyle{pb}}}
\end{equation}
To construct this diagram one induces $f_{2,x}$ using $f_0$ and the bottom pullback, and then it follows that the square $(C_2,A,A',C_2')$ is a pullback. One can then induce $f_{3,x}$ and $f_{4,x}$ using the bottom distributivity pullback. Lemma \ref{lem:distpb-cube} ensures that $(C_4,B,B',C_4')$ is a pullback, and elementary properties of pullbacks ensure that $(C_2,C_3,C_3',C_2')$ and $(C_3,C_4,C_4',C_3')$ are also pullbacks.

Our next task is to explain why (\ref{diag:2-cell-map-of-P-internal}) really does describe the components of (\ref{diag:2-cell-map-of-P-external}). This verification begins by unpacking, for a given commuting square as shown on the right
\[ \xygraph{!{0;(3,0):} {\xybox{\xygraph{!{0;(1.2,0):} {A}="tl" [r] {B}="tr" [d] {D}="br" [l] {C}="bl" "tl":"tr"^-{f}:"br"^-{k}:@{<-}"bl"^-{g}:@{<-}"tl"^-{h}}}} [r]
{\xybox{\xygraph{!{0;(1.2,0):} {\ca E/A}="tl" [r] {\ca E/B}="tr" [d] {\ca E/D}="br" [l] {\ca E/C}="bl" "tl":"tr"^-{\Sigma_f}:@{<-}"br"^-{\Delta_k}:@{<-}"bl"^-{\Sigma_g}:"tl"^-{\Delta_h} [d(.5)r(.35)] :@{=>}[r(.3)]^-{\alpha}}}} [r]
{\xybox{\xygraph{!{0;(1.2,0):} {\ca E/A}="tl" [r] {\ca E/B}="tr" [d] {\ca E/D}="br" [l] {\ca E/C}="bl" "tl":@{<-}"tr"^-{\Delta_f}:"br"^-{\Pi_k}:"bl"^-{\Delta_g}:@{<-}"tl"^-{\Pi_h} [d(.5)r(.35)] :@{<=}[r(.3)]^-{\beta}}}}} \]
the left Beck-Chevalley cell $\alpha$, and when $h$ and $k$ are exponentiable, the right Beck-Chevalley cell $\beta$, in elementary terms. Since $\alpha$ is obtained by taking the mate of the identity $\Sigma_{g}\Sigma_{h}=\Sigma_{k}\Sigma_{f}$ via the adjunctions $\Sigma_h \ladj \Delta_h$ and $\Sigma_k \ladj \Delta_k$, it follows that $\alpha$ is uniquely determined by the equation $\Sigma_k\varepsilon^{(1)}_f = (\varepsilon^{(1)}_g\Sigma_k)(\Sigma_g\alpha)$. On the other hand one has the commutative diagram
\begin{equation}\label{eq:left-BC}
\xygraph{!{0;(1,0):(0,1.2)::} {A_3}="t" [dr] {A_2}="tl" [r] {A}="tm" [r] {B}="tr" [d] {D}="br" [l] {C}="bm" [l] {C_2}="bl" "tl":"tm"^-{\Delta_hx}:"tr"^-{f}:"br"^-{k}:@{<-}"bm"^-{g}:@{<-}"bl"^-{x}:@{<-}"tl"|(.6){\varepsilon^{(1)}_{h,x}}:@{.>}"t"_-{\alpha_x}(:@/^{1pc}/"tr"^-{\Delta_k(gx)},:@/_{1pc}/"bl"_-{\varepsilon^{(1)}_{k,gx}}) "tm":"bm"^-{h} "tl":@{}"bm"|-*{\scriptstyle{pb}}}
\end{equation}
and the above equation, for the component $x$, is witnessed by the commutativity of the bottom triangle. Thus
\begin{lem}\label{lem:left-BC}
The components of $\alpha$ are induced as in (\ref{eq:left-BC}).
\end{lem}
Moreover we can see directly from (\ref{eq:left-BC}) that if the original square is a pullback, then $\alpha$ is invertible, and the converse follows by considering the case $x=1_C$. The right Beck-Chevalley cell $\beta$ may be obtained by taking the mate of $\Delta_f\Delta_k \iso \Delta_h\Delta_g$ via $\Delta_h \ladj \Pi_h$ and $\Delta_k \ladj \Pi_k$. Thus it is uniquely determined by the commutativity of
\begin{equation}\label{eq:right-BC-external}
\xygraph{!{0;(2,0):(0,.5)::} {\Delta_h\Delta_g\Pi_k}="tl" [r] {\Delta_h\Pi_h\Delta_f}="tr" [d] {\Delta_f}="br" [l] {\Delta_f\Delta_k\Pi_k}="bl" "tl":"tr"^-{\Delta_h\beta}:"br"^-{\varepsilon^{(2)}_h\Delta_f}:@{<-}"bl"^-{\Delta_f\varepsilon^{(2)}_k}:@{<-}"tl"^-{\tn{coh.}\Pi_k}} 
\end{equation}
whereas inside $\ca E$ for all $B_2$ and $x:B_2 \to B$ we have
\begin{equation}\label{eq:right-BC}
\xygraph{!{0;(2,0):(0,.5)::} {A}="a" [r] {B}="b" [d] {D}="d" [l] {C}="c" [l] {C_3}="c3" [u] {A_3}="a3" [u(.5)r(.5)] {A_2}="a2" [r(2)] {B_2}="b2" [u(.5)r(.5)] {B_3}="b3" [d(3)] {D_2}="d2" [l(3.5)] {C_2}="c2" [u(3)] {A_4}="a4"
"a":"b"_-{f}:"d"_-{k}:@{<-}"c"_-{g}:@{<-}"a"_-{h} "a":@{<-}"a2":@{<-}"a3":"c3":"c":@{<-}"c2":@{<-}"a4"(:"a2":"b2",:"b3"(:"b2":"b"^-{x},:"d2"(:"d",:@{<-}"c2"))) "a4":@{.>}"a3" "c2":@{.>}"c3"|-{\beta_x} "a":@{}"d"|-{}="centre" "centre" ([u(.75)] {\scriptstyle{pb}} [u(.5)] {\scriptstyle{pb}}, [d] {\scriptstyle{pb}}, [l] {\scriptstyle{dpb}}, [r] {\scriptstyle{dpb}})}
\end{equation}
constructed as follows. Take the distributivity pullback of $x$ along $k$ and then pullback the result along $g$. Pullback $x$ along $f$ and then take the distributivity pullback of the result along $h$. Then form the top pullback and induce the morphism $A_4 \to C_2$. By the elementary properties of pullbacks, it follows that the squares $(A_4,B_3,D_2,C_2)$ and $(A_4,A,C,C_2)$ are pullbacks. From this last we induce the dotted arrows using the left distributivity pullback.
\begin{lem}\label{lem:right-BC}
The components of $\beta$ are induced as in (\ref{eq:right-BC}).
\end{lem}
\begin{proof}
Note that the square $(A_4,A_3,C_3,C_2)$ is also a pullback, and so one can identify the commuting triangle $(A_4,A_3,A_2)$ with (\ref{eq:right-BC-external}), once one has understood that $A_4 \to A_2$ is the (appropriate component of) the lower composite in (\ref{eq:right-BC-external}).
\end{proof}
From this construction and Lemma \ref{lem:distpb-cube} one may witness directly that if the original square is a pullback, then $\beta$ is invertible, and the converse is easily witnessed by considering the case $x=1_B$. With these details sorted out we can now proceed to the proof of
\begin{lem}\label{lem:explicit-polymap->cart-transformation}
The component at $x:C \to X$ of the natural transformation described in (\ref{diag:2-cell-map-of-P-external}) is the morphism $f_{4,x}$ described in (\ref{diag:2-cell-map-of-P-internal}).
\end{lem}
\begin{proof}
The proof consists of unpacking the definition of $(\PFun{\ca E})_{X,Y}(f)_x$ with reference to the diagram (\ref{diag:2-cell-map-of-P-internal}), keeping track of the canonical isomorphisms that participate in the definition. All of this may be witnessed in
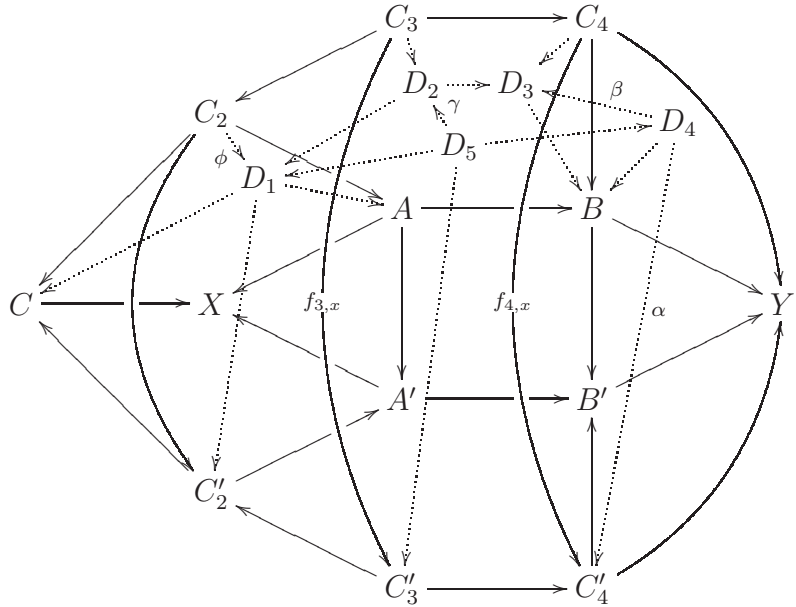
\[ \xygraph{!{0;(2,0):(0,.5)::} {X}="X" ([ur] {A}="A" [r] {B}="B" [dr] {Y}="Y", [dr] {A'}="Ap" [r] {B'}="Bp") "X" [l] {C}="C" ([u(2)r] {C_2}="C2" [ur] {C_3}="C3" [r] {C_4}="C4",[d(2)r] {C'_2}="Cp2" [dr] {C'_3}="Cp3" [r] {C'_4}="Cp4")
"X":@{<-}"A"^-{}|(.61)*{\hole}:"B"^-{}|(.63)*{\hole}:"Y"^-{} "X":@{<-}"Ap"_-{}|(.61)*{\hole}:"Bp"_-{}|(.63)*{\hole}:"Y"_-{} "C":@{<-}"C2":@{<-}"C3":"C4"^-{}:@/^{1.5pc}/"Y"^-{} "C":@{<-}"Cp2":@{<-}"Cp3":"Cp4"_-{}:@/_{1.5pc}/"Y"_-{} "C":"X"|(.58)*{\hole} "C2":"A"|(.67)*{\hole} "Cp2":"Ap"|(.67)*{\hole} "C4":"B" "Cp4":"Bp" "A":"Ap"^-{} "B":"Bp"^-{} "C2":@/_{2.5pc}/"Cp2"|(.35){} "C3":@/_{2.5pc}/"Cp3"|-{f_{3,x}} "C4":@/_{2.5pc}/"Cp4"|-{f_{4,x}}
"X" [r(.25)u(1.3)] {D_1}="D1" "D1"(:@{.>}"Cp2",:@{.>}"A",:@{.>}"C"
"A" [u(1.3)r(.1)] {D_2}="D2" [r(.5)] {D_3}="D3" "D1":@{<.}"D2":@{.>}"D3":@{.>}"B" "C2":@{.>}"D1"_-{\phi} "C3":@{.>}"D2" "C4":@{.>}"D3"
"B" [r(.45)u(.9)] {D_4}="D4" "D4"(:@{.>}"B",:@{.>}"Cp4"^(.4){\alpha})
"A" [u(.65)r(.3)] {D_5}="D5" "D5"(:@{.>}"D1",:@{.>}"Cp3",:@{.>}"D4") "D5":@{.>}"D2"_-{\gamma} "D4":@{.>}"D3"_(.4){\beta}} \]
in which the solid arrows appeared already in (\ref{diag:2-cell-map-of-P-internal}), and the dotted arrows are constructed as follows. Form $D_1$ by pulling back $f_0$ and $C_2' \to A'$, then $D_1 \to C$ is the composite for the triangle $(D_1,C_2',C)$. The form $D_2$, $D_3$ and $D_4$ by taking the distributivity pullback of $D_1 \to A$ along $p_2$. Form $D_4$ by pulling back $f_1$ and $C_4' \to B'$. The construction of the rest of the data proceeds in the same way as for (\ref{eq:right-BC}) as shown in
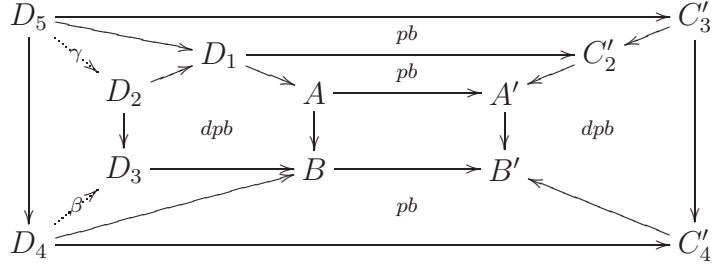
\[ \xygraph{!{0;(2,0):(0,.4)::} {A}="a" [r] {A'}="b" [d] {B'}="d" [l] {B}="c" [l] {D_3}="c3" [u] {D_2}="a3" [u(.5)r(.5)] {D_1}="a2" [r(2)] {C_2'}="b2" [u(.5)r(.5)] {C_3'}="b3" [d(3)] {C_4'}="d2" [l(3.5)] {D_4}="c2" [u(3)] {D_5}="a4"
"a":"b"_-{}:"d"_-{}:@{<-}"c"_-{}:@{<-}"a"_-{} "a":@{<-}"a2":@{<-}"a3":"c3":"c":@{<-}"c2":@{<-}"a4"(:"a2":"b2",:"b3"(:"b2":"b"^-{},:"d2"(:"d",:@{<-}"c2"))) "a4":@{.>}"a3"|-{\gamma} "c2":@{.>}"c3"|-{\beta} "a":@{}"d"|-{}="centre" "centre" ([u(.75)] {\scriptstyle{pb}} [u(.5)] {\scriptstyle{pb}}, [d] {\scriptstyle{pb}}, [l] {\scriptstyle{dpb}}, [r] {\scriptstyle{dpb}})} \]
Thus the arrow labelled as $\beta$ is by Lemma \ref{eq:right-BC} the right Beck-Chevalley isomorphism, and $\gamma$ is also invertible. Clearly $\phi$ is an isomorphism witnessing the pseudo-functoriality of $\Delta_{(-)}$, and considering (\ref{eq:left-BC}) for the square
\[ \xygraph{{B}="tl" [r] {Y}="tr" [d] {Y}="br" [l] {B'}="bl" "tl":"tr"^-{p_3}:"br"^-{1_Y}:@{<-}"bl"^-{q_3}:@{<-}"tl"^-{f_0}} \]
the arrow labelled $\alpha$ is evidently the appropriate component of the left Beck-Chevalley cell by Lemma \ref{lem:left-BC}. Thus $(\PFun{\ca E})_{X,Y}(f)_x$ is by definition the composite
\[ \xygraph{{C_4}="p1" [r] {D_3}="p2" [r] {D_4}="p3" [r] {C_4'}="p4" "p1":"p2"^-{}:"p3"^-{\beta^{-1}}:"p4"^-{\alpha}} \]
and to finish the proof we must show that this composite is $f_{4,x}$. Provisionally let us denote by $\xi$ this composite, and by $\zeta$ the composite
\[ \xygraph{{C_3}="p1" [r] {D_2}="p2" [r] {D_5}="p3" [r] {C_3'.}="p4" "p1":"p2"^-{}:"p3"^-{\gamma^{-1}}:"p4"^-{}} \]
Observe that the squares
\[ \xygraph{!{0;(3,0):} {\xybox{\xygraph{{C_2}="tl" [r] {C_3}="tr" [d] {C_3'}="br" [l] {C_2'}="bl" "tl":@{<-}"tr"^-{}:@{.>}"br"^-{\zeta}:"bl"^-{}:@{<-}"tl"^-{}}}} [r]
{\xybox{\xygraph{{C_3}="tl" [r] {C_4}="tr" [d] {C_4'}="br" [l] {C_3'}="bl" "tl":"tr"^-{}:@{.>}"br"^-{\xi}:@{<-}"bl"^-{}:@{<.}"tl"^-{\zeta}}}} [r]
{\xybox{\xygraph{{C_4}="tl" [r] {B}="tr" [d] {B'}="br" [l] {C_4'}="bl" "tl":"tr"^-{}:"br"^-{f_1}:@{<-}"bl"^-{}:@{<.}"tl"^-{\xi}}}}} \]
are commutative, and so by the uniqueness aspect of the universal property of the bottom distributivity pullback, it follows that $\zeta=f_{3,x}$ and $\xi=f_{4,x}$.
\end{proof}
The importance of this alternative description is that it can, in various ways, be written in terms of composition in the bicategory $\Polyc{\ca E}$ whose composition and coherence we understand. These ways are described in the following result, which follows immediately from Lemma \ref{lem:naturality-Poly-hom-projections} and Lemma \ref{lem:explicit-polymap->cart-transformation}.
\begin{lem}\label{lem:all-descriptions-2cell-map-P}
Let $p$ and $q:X \to Y$ be polynomials in $\ca E$ and $f:p \to q$ be a cartesian morphism between them. Then given $x:C \to X$ in $\ca E/X$, one has
\[ \PFun{\ca E}(f)_x = f_{4,x} = r_{Z,Y}(p \comp \lft{x} \comp \rgt{g}) \]
for all $Z$ and $g:C \to Z$.
\end{lem}
The fact that the one and 2-cell maps of $\PFun{\ca E}$ have, by Lemmas \ref{lem:all-descriptions-associated-polynomial-functor} and \ref{lem:all-descriptions-2cell-map-P}, been described in terms of the bicategory structure of $\Polyc{\ca E}$, is the reason why they give a homomorphism of bicategories. We expand on this further in the proof of
\begin{thm}\label{thm:polnomial-functor-homomorphism}
Let $\ca E$ be a category with pullbacks. With the object map $X \mapsto \ca E/X$, arrow map $p \mapsto \Sigma_{p_3}\Pi_{p_2}\Delta_{p_1}$, and 2-cell map depicted in (\ref{diag:2-cell-map-of-P-external}), one has a homomorphism
\[ \PFun{\ca E} : \Polyc{\ca E} \longrightarrow \CAT \]
of bicategories.
\end{thm}
\begin{proof}
It remains to exhibit the coherence isomorphisms and verify the coherence axioms. We assume a canonical choice of all pullbacks and existing distributivity pullbacks as explained at the end of Section \ref{ssec:dpb}. In particular this implies that identities in $\Polyc{\ca E}$ are strict, making $\PFun{\ca E}(1_X) = 1_{\ca E/X}$ for all $X \in \ca E$ by Lemma \ref{lem:all-descriptions-associated-polynomial-functor}. Let $p:X \to Y$ and $q:Y \to Z$ be polynomials. For $x:C \to X$ in $\ca E/X$ one has the associativity isomorphism
\[ \alpha_{q,p,\lft{x}}^{-1} : q \comp (p \comp \lft{x}) \iso (q \comp p) \comp \lft{x} \]
and so the component of the coherence isomorphism
\[ \pi_{q,p,x} : \PFun{\ca E}(q)\PFun{\ca E}(p)(x) \iso \PFun{\ca E}(q \comp p)(x) \]
is defined to be $l_{C,Z}(\alpha_{q,p,\lft{x}}^{-1})$. Naturality in $q$, $p$ and $x$ is clear by definition. Note that by Lemma \ref{lem:naturality-Poly-hom-projections} there are many other descriptions of this same component, namely
\[ \pi_{q,p,x} = r_{D,Z}(\alpha_{q,p,\lft{x}}^{-1} \comp \rgt{g}) \]
for any $D$ and $g:C \to D$. Using this and Lemmas \ref{lem:all-descriptions-associated-polynomial-functor} and  \ref{lem:all-descriptions-2cell-map-P}, one can exhibit any component of any bicategorical homomorphism coherence diagram, as the image of a diagram of coherence isomorphisms in $\Polyc{\ca E}$, by a right projection of one of $\Polyc{\ca E}$'s homs. By Theorem \ref{thm:bicat-of-polynomials} all such diagrams commute.
\end{proof}
\begin{defn}\label{defn:polyfunctor}
A \emph{polynomial functor} over $\ca E$ is a functor which is isomorphic to a composite of functors of the form $\Sigma_f$, $\Delta_g$ and $\Pi_h$, where $f$ and $g$ can be arbitrary morphisms of $\ca E$, and $h$ can be an exponentiable morphism of $\ca E$. 
\end{defn}
It follows from Theorem \ref{thm:polnomial-functor-homomorphism} that a functor between slices of $\ca E$ is polynomial if and only if it is in the essential image of $\PFun{\ca E}$.
\begin{defn}\label{defn:polymnd}
Let $X \in \ca E$. A monad $T$ on $\ca E/X$ is a \emph{polynomial monad} if it is isomorphic to a monad of the form $\PFun{\ca E}(p)$, where $p$ is a monad on $X$ in $\Polyc{\ca E}$. A morphism $\phi : S \to T$ of monads on $\ca E/X$ is a \emph{polynomial monad morphism} if $\phi$ factors as
\[ \xygraph{!{0;(1.5,0):(0,1)::} {S}="p0" [r] {\PFun{\ca E}(q)}="p1" [r(1.25)] {\PFun{\ca E}(p)}="p2" [r] {T}="p3" "p0":"p1"^-{\iota_1}:"p2"^-{\PFun{\ca E}(\phi')}:"p3"^-{\iota_2}} \]
where $\iota_1$ and $\iota_2$ are isomorphisms of monads, and $\phi' : q \to p$ is a morphism of monads on $X$ in $\Polyc{\ca E}$.
\end{defn}
We conclude this section by observing that the hom functors of $\PFun{\ca E}$ are faithful and conservative.
\begin{prop}\label{prop:faithfulness-P-E-XY}
For any category $\ca E$ with pullbacks and objects $X,Y \in \ca E$, the hom functor $(\PFun{\ca E})_{X,Y}$ is faithful and conservative.
\end{prop}
\begin{proof}
Considering the instance of (\ref{diag:2-cell-map-of-P-internal}) in which $x=1_X$, it is clear that $f_{4,1_X} = f_1$, and so by Lemma \ref{lem:explicit-polymap->cart-transformation} $(\PFun{\ca E})_{X,Y}(f)$ uniquely determines $f_1$. Let us now consider the case $x=q_1$. In that case $C=A'$ and the morphisms $C'_2 \to C$ and $C'_2 \to A'$, namely the projections of the pullback defining $C'_2$, have a common section $s_1:A' \to C'_2$. Applying Lemma \ref{lem:distpb-sections} to the bottom distributivity pullback, one obtains the sections $s_2:A' \to C_3'$ and $s_3:B' \to C_4'$ satisfying the naturality conditions of that lemma. Since the square $(C_4,B,B',C_4')$ is a pullback, one induces unique $s_4:B \to C_4$ which is a section of the given map $C_4 \to B$ and satisfies $f_{4,x}s_4=s_3f_1$. Applying Lemma \ref{lem:distpb-sections}, this time to the top distributivity pullback, one induces the natural sections $s_5:A \to C_3$ and $s_6:A \to C_2$. From the naturality conditions of the sections so constructed and the commutativities in the original general diagram (\ref{diag:2-cell-map-of-P-internal}), it follows easily that $f_0$ is equal to the composite
\[ \xygraph{{A}="l" [r] {C_2}="m" [r(1.25)] {C=A',}="r" "l":"m"^-{s_6}:"r"^-{}} \]
which by construction and Lemma \ref{lem:explicit-polymap->cart-transformation}, is determined uniquely by $(\PFun{\ca E})_{X,Y}(f)$, and so $(\PFun{\ca E})_{X,Y}$ is faithful. If $(\PFun{\ca E})_{X,Y}(f)$ is invertible, then $f_1$, which we saw is the component at $1_X$ of this natural transformation, must also be invertible. Since $f_0$ is a pullback along $q_2$ of $f_1$, $f_0$ is also invertible, and so $(\PFun{\ca E})_{X,Y}$ is conservative.
\end{proof}
Example 2.10 of \cite{GambinoKock-PolynomialFunctors} shows that $(\PFun{\ca E})_{X,Y}$ is not full in general, and this is discussed further in Remark \ref{rem:GamKock}. 

\subsection{Enrichment over $\Cart$.}
\label{ssec:enriched-bicat-poly}
Recall from \cite{Bourke-CatPB} that the category $\Cart$ of categories with pullbacks and pullback preserving functors is cartesian closed. The product in $\Cart$ is as in $\CAT$, and the internal hom $[X,Y]$ is the category of pullback preserving functors $X \to Y$ and cartesian transformations between them. A \emph{$\Cart$-bicategory} is a bicategory $\ca B$ whose homs have pullbacks and whose compositions
\[ \tn{comp}_{X,Y,Z} : \ca B(Y,Z) \times \ca B(X,Y) \to \ca B(X,Z) \]
preserve them. The basic example is $\Cart$ itself. A \emph{homomorphism} $F:\ca B \to \ca C$ of $\Cart$-bicategories is a homomorphism of their underlying bicategories whose hom functors preserve pullbacks. The point of this section is to show that for any category $\ca E$ with pullbacks, the homomorphism $\PFun{\ca E}$ is in fact a homomorphism of $\Cart$-bicategories.

For all $f:A \to B$ in a category $\ca E$ with pullbacks, it is easy to witness directly that the adjunction $\Sigma_f \ladj \Delta_f$ lives in $\Cart$. So polynomial functors preserve pullbacks, and the diagram (\ref{diag:2-cell-map-of-P-external}) may be regarded as living in $\Cart$. In other words, $\PFun {\ca E}$ sends 2-cells in $\Polyc {\ca E}$ to cartesian transformations. Thus the hom maps of $\PFun{\ca E}$ may be regarded as landing in the homs of $\Cart$, that is, one can write
\[ (\PFun{\ca E})_{X,Y} : \Polyc{\ca E}(X,Y) \to \Cart(\ca E/X,\ca E/Y). \]
In fact these functors themselves live in $\Cart$. To see this we first we note that
\begin{lem}\label{lem:Poly-E-XY-pullbacks}
For any category $\ca E$ with pullbacks and objects $X,Y \in \ca E$, the category $\Polyc{\ca E}(X,Y)$ has pullbacks, and a commutative square in $\Polyc{\ca E}(X,Y)$ is a pullback if and only if its $1$-component is a pullback in $\ca E$.
\end{lem}
\begin{proof}
One has a canonical inclusion 
\[ \Polyc{\ca E}(X,Y) \longrightarrow \ca E^{\leftarrow\to\to} \]
of $\Polyc{\ca E}(X,Y)$ into the functor category. In general, given a category $\C$, an arrow $\alpha$ in $\C$, and a square
\[ \xygraph{{P}="tl" [r] {B}="tr" [d] {C}="br" [l] {A}="bl" "tl":"tr"^-{q}:"br"^-{g}:@{<-}"bl"^-{f}:@{<-}"tl"^-{p}} \]
in $[\C,\ca E]$, then if the naturality squares of $f$ and $g$ at $\alpha$ are pullbacks, then so are those for $p$ and $q$, by the elementary properties of pullback squares. Since exponentiable maps are pullback stable, and one may choose pullbacks in $\ca E$ so that identity arrows are pullback stable, it follows that pullbacks in $\Polyc{\ca E}(X,Y)$ exist and are formed as in $\ca E^{\leftarrow\to\to}$. Thus it follows in particular that a commutative square in $\Polyc{\ca E}(X,Y)$ as in the statement is a pullback if and only if its $0$ and $1$-components are pullbacks in $\ca E$. But from the elementary properties of pullbacks, if the $1$-component square is a pullback then so is the $0$-component.
\end{proof}
\noindent and so one has
\begin{prop}\label{prop:P-E-XY-pres-ref-pbs}
For any category $\ca E$ with pullbacks and objects $X,Y \in \ca E$, the functor $(\PFun{\ca E})_{X,Y}$ preserves and reflects pullbacks.
\end{prop}
\begin{proof}
Since by Proposition \ref{prop:faithfulness-P-E-XY} $(\PFun{\ca E})_{X,Y}$ is conservative, it suffices to show that it preserves pullbacks. But by the elementary properties of pullbacks, a square in $\Cart(\ca E/X,\ca E/Y)$ is a pullback if and only if its component at $1_X$ is a pullback. Since by Lemma \ref{lem:explicit-polymap->cart-transformation} the component at $1_X$ of $(\PFun{\ca E})_{X,Y}(f)$ is just $f_1$, the result follows from Lemma \ref{lem:Poly-E-XY-pullbacks}.
\end{proof}
\begin{thm}\label{thm:pb-enrichment-poly}
Let $\ca E$ be a category with pullbacks. Then $\Polyc{\ca E}$ is a $\Cart$-bicategory and
\[ \PFun{\ca E} : \Polyc{\ca E} \longrightarrow \Cart \]
is a homomorphism of $\Cart$-bicategories.
\end{thm}
\begin{proof}
By Proposition \ref{prop:P-E-XY-pres-ref-pbs} it suffices to show that the composition functors of $\Polyc{\ca E}$ preserve pullbacks. One has for each $X,Y,Z \in \ca E$, an isomorphism
\[ \xygraph{!{0;(5,0):(0,.2)::} {\Polyc{\ca E}(Y,Z) \times \Polyc{\ca E}(X,Y)}="tl" [r] {\Polyc{\ca E}(X,Z)}="tr" [d] {\Cart(\ca E/X,\ca E/Z)}="br" [l] {\Cart(\ca E/Y,\ca E/Z) \times \Cart(\ca E/X,\ca E/Y)}="bl" "tl":"tr"^-{\comp}:"br"^-{(\PFun{\ca E})_{X,Z}}:@{<-}"bl"^-{\comp}:@{<-}"tl"^-{(\PFun{\ca E})_{Y,Z} \times (\PFun{\ca E})_{X,Y}} "tl" [d(.5)r(.6)] {\iso}} \]
The vertical functors preserve and reflect pullbacks by Proposition \ref{prop:P-E-XY-pres-ref-pbs}, the bottom one preserves pullbacks since $\Cart$ is a $\Cart$-bicategory by cartesian closedness, and so the composition functor for $\Polyc{\ca E}$ preserves pullbacks as required.
\end{proof}
\begin{rem}\label{rem:GamKock}
In the case where $\ca E$ is locally cartesian closed, the work of Gambino and Kock \cite{GambinoKock-PolynomialFunctors} tells us more. In that case $\ca E$ is in particular a monoidal category via its cartesian product, and it acts as a monoidal category on its slices. Moreover polynomial functors over such $\ca E$ acquire a canonical strength. Then by Proposition 2.9 of \cite{GambinoKock-PolynomialFunctors}, the image of $\PFun{\ca E}$ consists of the slices of $\ca E$, polynomial functors over $\ca E$ and \emph{strong} cartesian transformations between them.
\end{rem}
\begin{rem}\label{rem:iterability}
Since for any category $\ca E$ with pullbacks the homs of $\Polyc{\ca E}$ also have pullbacks, the above result can be applied to any of those homs in place of $\ca E$, giving a sense in which the theory of polynomials may be iterated. Such iteration is implicit in unpublished work of Martin Hyland on fibrations, type theory and the Dialectica interpretation. Building on this, the thesis \cite{VonGlehn-Thesis} of von Glehn provides polynomial models of type theory, and shows the possibility of a more sophisticated kind of iteration.
\end{rem}
We conclude this section by describing the sense in which the hom functors of $\PFun{\ca E}$ are fibrations. First we require some preliminary definitions. Given a functor $F : \ca A \to \ca B$, a morphism $f : X \to Y$ of $\ca A$ is \emph{$F$-cartesian} when for all $g : Z \to X$ and $h : FZ \to FX$ such that $F(g)h = Ff$, there exists a unique $k : Z \to X$ such that $Fk = h$ and $fk = g$. When $F$ is a fibration in the bicategorical sense of Street \cite{Street-FibrationInBicats}, we say that it is a \emph{bi-fibration}. This property on $F$ can be formulated in elementary terms as follows: every $f : B \to FA$ factors as
\[ \xygraph{{B}="p0" [r] {FC}="p1" [r] {FB}="p2" "p0":"p1"^-{g}:"p2"^-{Fh}} \]
where $h$ is $F$-cartesian and $g$ is an isomorphism.
\begin{prop}\label{prop:homfunctors-bi-fibs}
Suppose that $\ca E$ is a category with finite limits, and let $X$ and $Y \in \ca E$. Then
\[ (\PFun{\ca E})_{X,Y} : \Polyc{\ca E}(X,Y) \longrightarrow \Cart(\ca E/X,\ca E/Y) \]
is a bi-fibration, and every morphism of $\Polyc{\ca E}(X,Y)$ is $(\PFun{\ca E})_{X,Y}$-cartesian.
\end{prop}
\begin{proof}
We begin by verifying that $(\phi_0,\phi_1) : (p_1,p_2,p_3) \to (q_1,q_2,q_3)$ as in
\[ \xygraph{!{0;(1.5,0):(0,.5)::} {X}="p0" [ur] {E_1}="p1" [r] {B_1}="p2" [dr] {Y}="p3" [dl] {B_2}="p4" [l] {E_2}="p5" "p0":@{<-}"p1"^-{p_1}:"p2"^-{p_2}:"p3"^-{p_3}:@{<-}"p4"^-{q_3}:@{<-}"p5"^-{q_2}:"p0"^-{q_1} "p1":"p5"_-{\phi_0} "p2":"p4"_-{\phi_1} "p1":@{}"p4"|-{\tn{pb}}} \]
is $(\PFun{\ca E})_{X,Y}$-cartesian. Denoting by $1$ the terminal object of $\ca E$, note that $\PFun(\phi_0,\phi_1) = \phi_1$. Given $(\psi_0,\psi_1) : (r_1,r_2,r_3) \to (q_1,q_2,q_3)$, and a cartesian natural transformation $\alpha : \PFun{\ca E}(r_1,r_2,r_3) \to \PFun{\ca E}(q_1,q_2,q_3)$ such that $\PFun{\ca E}(\phi_0,\phi_1)\alpha = \PFun{\ca E}(\psi_0,\psi_1)$, we must exhibit $(\beta_0,\beta_1) : (r_1,r_2,r_3) \to (p_1,p_2,p_3)$ unique such that $\PFun{\ca E}(\beta_0,\beta_1) = \alpha$ and $(\phi_0,\phi_1)(\beta_0,\beta_1) = (\psi_0,\psi_1)$ in $\Polyc{\ca E}$. But the first of these equations forces $\beta_1 = \alpha_1$, and the second equation forces $\beta_0$ to be induced as in
\[ \xygraph{!{0;(1.5,0):(0,.5)::}
{X}="p0" [ur] {E_1}="p1" [r] {B_1}="p2" [dr] {Y}="p3" [dl] {B_2}="p4" [l] {E_2}="p5" "p0":@{<-}"p1"^-{}:"p2"|(.45)*=<4pt>{}:"p3"|(.4)*=<4pt>{}:@{<-}"p4"^-{q_3}:@{<-}"p5"^-{q_2}:"p0"^-{q_1} "p1":"p5"_-{\phi_0} "p2":"p4"_-{\phi_1} "p1":@{}"p4"|-{\tn{pb}}
"p1" [r(.5)u(1.5)] {E_3}="q1" [r] {B_3}="q2" "p0":@/^{1pc}/@{<-}"q1"^-{r_3}:"q2"^-{r_2}:"p3"^-{r_3}
"q1"(:@{.>}"p1"|-{\beta_0},:@/^{.75pc}/"p5"^(.525){\psi_0}) "q2"(:"p2"_-{\alpha_1},:@/^{.75pc}/"p4"^(.65){\psi_1})} \]
The statement that $(\PFun{\ca E})_{X,Y}$ is a bi-fibration is a reformulation of the fact that if $P : \ca E/X \to \ca E/Y$ is a polynomial functor and $\phi : Q \to P$ is a cartesian transformation, then $Q$ is also polynomial. This appeared as Lemma 2.10 of \cite{GambinoKock-PolynomialFunctors}, and the proof given there works at the present generality.
\end{proof}
\begin{cor}\label{cor:inherit-polymnd-along-cart-monad-morphism}
Let $\ca E$ be a category with finite limits, $X \in \ca E$, $P$ and $Q$ be monads on $\ca E/X$, and $\phi : Q \to P$ be a monad morphism. If $P$ is a polynomial monad and $\phi$ is a cartesian monad morphism, then $Q$ is a polynomial monad and $\phi$ is a polynomial monad morphism.
\end{cor}
\begin{proof}
It suffices to show that if $p$ is a monad on $X$ in $\Polyc{\ca E}$ and $\phi : Q to \PFun{\ca E}(p)$ is a cartesian monad morphism, then $\phi$ factors as
\[ \xygraph{{Q}="p0" [r(1.2)] {\PFun{\ca E}(q)}="p1" [r(1.5)] {\PFun{\ca E}(p)}="p2" "p0":"p1"^-{\iota}:"p2"^-{\PFun{\ca E}(\psi)}} \]
where $\iota$ is an isomorphism of monads on $\ca E/X$ and $\psi : q \to p$ is a morphism of monads on $X$ in $\Polyc{\ca E}$. At the level of endomorphisms this follows from Proposition \ref{prop:homfunctors-bi-fibs}, so it suffices to exhibit a monad structure on $q$ making $\iota$ and $\psi$ into monad morphisms. The unit $\eta^q : 1 \to q$ is defined to be the unique 2-cell such that $\psi\eta = \eta^p$ and $\PFun{\ca E}(\eta^q) = \eta^{\PFun{\ca E}(q)} = \iota\eta^Q$, by the $(\PFun{\ca E})_{X,X}$-cartesianness of $\psi$. The multiplication $\mu^q : q \comp q \to q$ is defined similarly as the unique 2-cell such that $\psi\mu^q = \mu^p(\psi \comp \psi)$ and $\PFun{\ca E}(\mu^q) = \iota\mu^Q(\iota \comp \iota)^{-1}$. The monad axioms for $(q,\eta^q,\mu^q)$ are deduced from those of $p$ and the uniqueness aspects of $\psi$'s $(\PFun{\ca E})_{X,X}$-cartesianness, and with respect to this structure, $\iota$ and $\psi$ are monad morphisms essentially by definition.
\end{proof}

\section{Polynomials in 2-categories}
\label{sec:Poly-in-2-Cats}

We now develop the 2-categorical aspects of the theory of polynomials. In Section \ref{ssec:2-categorical-version} we directly generalise Section \ref{sec:Poly-in-Cats} to the setting of 2-categories. One sense in which the study of polynomial 2-functors is richer than its 1-dimensional counterpart, is that one can consider whether such 2-functors are compatible with the theory of fibrations internal to the 2-category in which the corresponding polynomial lives. We review briefly the theory of fibrations in a 2-category in Section \ref{ssec:fib-2-monads}, and in Section \ref{ssec:fam-2-fun} recall from \cite{Weber-Fam2fun} the general theory of familial 2-functors, which are those 2-functors (not necessarily polynomial) which are compatible with the 2-categorical theory of fibrations. Then in Section \ref{ssec:fam-from-poly} we give conditions on polynomials and morphisms thereof, so that the resulting polynomial 2-functors and morphisms thereof are familial. Since 2-dimensional monad theory \cite{BWellKellyPower-2DMndThy} is most usefully applied to sifted colimit preserving 2-monads as explained at the beginning of Section \ref{ssec:sifted-colim-preservation}, we give conditions on a polynomial in $\Cat$ so that ensure that its corresponding polynomial 2-functor preserves sifted colimits in Theorem \ref{thm:polynomials-sifted-colims}.

\subsection{Polynomial 2-functors.}
\label{ssec:2-categorical-version}
In this section we extend the developments of Section \ref{sec:Poly-in-Cats} to the setting of 2-categories. Let $\ca K$ be a 2-category with pullbacks. Recall that when one speaks of pullbacks, or more generally any weighted limit in a 2-category, the universal property has a 2-dimensional aspect. That is, a square $S$ in $\ca K$ is by definition a pullback in $\ca K$ if and only if for all $X \in \ca K$, the square $\ca K(X,S)$ in $\CAT$ is a pullback in $\CAT$. On objects this is the usual universal property of a pullback as in ordinary category theory, and on arrows this is the ``2-dimensional aspect''. Recall also \cite{Kelly-EnrichedCatsBook} that if $\ca K$ admits tensors with $[1]$, then the usual universal property implies this 2-dimensional aspect, but in the absence of tensors, one must verify the 2-dimensional aspect separately.

Similarly when we speak of distributivity pullbacks in $\ca K$ we will also demand that these satisfy a 2-dimensional universal property. Let $g:Z \to A$ and $f:A \to B$ be in $\ca K$. We describe first the 2-category $\tn{PB}(f,g)$ of pullbacks around $(f,g)$. The underlying category of $\tn{PB}(f,g)$ is described as in Definition \ref{def:pb-around}. Let $(s,t)$ and $(s',t'):(p,q,r) \to (p',q',r')$ be morphisms in $\tn{PB}(f,g)$. Then a 2-cell between them consists of 2-cells $\sigma:s \to s'$ and $\tau:t \to t'$ of $\ca K$, such that $p'\sigma = 1_p$, $q\sigma = \tau q'$ and $1_r = r'\sigma$. Compositions for $\tn{PB}(f,g)$ are inherited from $\ca K$. One thus defines a \emph{distributivity pullback} around $(f,g)$ in $\ca K$ to be a terminal object of the 2-category $\tn{PB}(f,g)$.

The meaning of distributivity pullbacks in this 2-categorical environment is the same as in the discussion of Section \ref{ssec:dpb}. First note that $\Sigma_f:\ca K/A \to \ca K/B$ is a 2-functor, and that by virtue of the 2-dimensional universal property of pullbacks in $\ca K$, $\Delta_f:\ca K/B \to \ca K/A$ is a 2-functor and $\Sigma_f \ladj \Delta_f$ is a 2-adjunction. To say that all distributivity pullbacks along $f$ exist in $\ca K$ is to say that $\Delta_f$ has a right $2$-adjoint, denoted $\Pi_f$ as before, and this right adjoint encodes the process of taking distributivity pullbacks along $f$. Such morphisms $f$ in $\ca K$ are said to be exponentiable, and as in the 1-dimensional case, exponentiable maps are closed under composition and are stable by pullback along arbitrary maps. Morever Lemmas \ref{lem:distpb-composition-cancellation}, \ref{lem:distpb-cube} and \ref{lem:distpb-sections} remain valid in our 2-categorical environment. The verification of this is just a matter of using the 2-dimensional aspects of pullbacks and distributivity pullbacks to induce the necessary 2-cells, in exact imitation of how one induced the arrows during these proofs in Section \ref{ssec:dpb}.

Polynomials in $\ca K$ and cartesian morphisms between them are defined as in Section \ref{ssec:bicats-of-polys}. Given polynomials $p$ and $q:X \to Y$, and cartesian morphisms $f$ and $g:p \to q$, a 2-cell $\phi:p \to q$ consists of 2-cells $\phi_0:f \to g_0$ and $\phi_1:f_1 \to g_1$ such that $p_1=q_1\phi_0$, $q_2\phi_0=\phi_1p_2$ and $q_3\phi_1=p_3$. With compositions inherited from $\ca K$ one has a 2-category $\Polyc{\ca K}(X,Y)$ together with left and right projections
\[ \xygraph{!{0;(3,0):} {\ca K/X}="l" [r] {\Polyc{\ca K}(X,Y)}="m" [r] {\ca K/Y.}="r" "l":@{<-}"m"^-{l_{X,Y}}:"r"^-{r_{X,Y}}} \]
For a composable sequence $(p_i)_i$ of polynomials as in
\[ \xygraph{{X_{i-1}}="p1" [r] {A_i}="p2" [r] {B_i}="p3" [r] {X_i}="p4" "p1":@{<-}"p2"^-{p_{i1}}:"p3"^-{p_{i2}}:"p4"^-{p_{i3}}} \]
one defines the 2-category $\tn{SdC}(p_i)_i$ of subdivided composites over $(p_i)_i$ as follows. The objects and arrows are defined as in Definitions \ref{def:subdivided-composite} and \ref{def:morphism-subdivided-composite}. Given morphisms $t$ and $t':(Y,q,r,s) \to (Y',q',r',s')$ of subdivided composites, a 2-cell $\tau:t \to t'$ consists of 2-cells $\tau_i:t_i \to t_i'$ in $\ca K$ for $0 \leq i \leq n$, such that $q_1=q_1'\tau_0$, $q_{2i}'\tau_{i−1}=\tau_iq_{2i}$, $q_3=q_3'\tau_n$, $r_i=r_i'\tau_{i−1}$ and $s_i=s_i'\tau_i$. Compositions in $\tn{SdC}(p_i)_i$ are inherited from $\ca K$. The process of taking the associated polynomial of a subdivided composite, as described in Definition \ref{def:ass-poly}, is 2-functorial. The forgetful functors $\tn{res}_n$ and $\tn{res}_0$ become 2-functors. The fact that Lemmas \ref{lem:poly-comp-subdivided-composite} and \ref{lem:subdivided-composite-comp-poly} remain valid in our 2-categorical environment, is once again a matter of using the 2-dimensional aspects of pullbacks and distributivity pullbacks to induce the necessary 2-cells in the same way that the arrows during these proofs were induced in the 1-dimensional case. Thus these 2-categories of subdivided composites admit terminal objects, and so one may define the composition of polynomials as in Definition \ref{def:composition-of-polynomials}. Moreover composition is 2-functorial.

Lemma \ref{lem:explicit-polymap->cart-transformation} gives a direct description of the arrow map of the hom functors of $\PFun{\ca E}$ as being induced by the universal properties of pullbacks and distributivity pullbacks. Thus in our 2-categorical setting, with the 2-dimensional aspects of these universal properties available, we can do the same one dimension higher and induce directly the components of the modification induced by a 2-cell between maps of polynomials. Thus we have 2-functors
\[ (\PFun{\ca K})_{X,Y} : \Polyc{\ca K}(X,Y) \to \TwoCAT(\ca K/X,\ca K/Y) \]
for all objects $X$ and $Y$ of a 2-category $\ca K$ with pullbacks. We now describe the structure that polynomials in a 2-category form.
\begin{defn}\label{def:2-bicats}
A \emph{2-bicategory} consists of a bicategory $\ca B$ whose hom categories are endowed with 2-cells making them 2-categories and the composition functors
\[ \tn{comp}_{X,Y,Z} : \ca B(Y,Z) \times \ca B(X,Y) \to \ca B(X,Z) \]
are endowed with 2-cell maps making them into 2-functors. In addition we ask that the coherence isomorphisms of $\ca B$ be natural with respect to the 3-cells. 
\end{defn}
Since a 2-bicategory $\ca B$ is a degenerate sort of tricategory, we shall call the 2-cells in its homs \emph{3-cells} of $\ca B$.
\begin{exams}\label{exs:basic-examples-2-bicats}
\begin{enumerate}
\item Given a 2-category $\ca K$ with pullbacks, the bicategory $\Span {\ca K}$ has the additional structure of a 2-bicategory in which a 3-cell $f \to g$ consists of a 2-cell $\phi$ as in
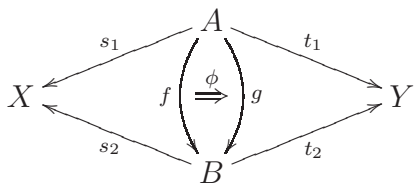
\[ \xygraph{!{0;(2,0):(0,.4)::} {X}="p0" [ur] {A}="p1" [dr] {Y}="p2" [dl] {B}="p3" "p0":@{<-}"p1"^-{s_1}:"p2"^-{t_1}:@{<-}"p3"^-{t_2}:"p0"^-{s_2}
"p1":@/_{1pc}/"p3"_-{f}|-{}="pd" "p1":@/^{1pc}/"p3"^-{g}|-{}="pc"
"pd":@{}"pc"|(.25){}="d"|(.75){}="c" "d":@{=>}"c"^-{\phi}} \]
such that $s_2\phi = \id$ and $t_2\phi = \id$.
\item Dually one has a 2-bicategory $\Cospan {\ca K}$ of cospans in any 2-category $\ca K$ with pushouts.
\item Any strict 3-category such as $\TwoCAT$ is a 2-bicategory.
\end{enumerate}
\end{exams}
\begin{defn}\label{def:hom-2-bicat}
A \emph{homomorphism} $F:\ca B \to \ca C$ of 2-bicategories is a homomorphism of their underlying bicategories whose hom functors are endowed with 2-cell maps making them into 2-functors, and whose coherence data is natural with respect to 3-cells.
\end{defn}
\begin{thm}\label{thm:2-bicat-of-polynomials}
Let $\ca K$ be a 2-category with pullbacks. One has a 2-bicategory $\Polyc{\ca K}$ whose objects are those of $\ca K$, whose hom between $X$ and $Y \in \ca K$ is $\Polyc{\ca K}(X,Y)$, and whose compositions are defined as above. Moreover with object, arrow and 2-cell maps defined as in the categorical case, and 3-cell map as defined above, one has a homomorphism
\[ \PFun{\ca K} : \Polyc{\ca K} \longrightarrow \TwoCAT \]
of 2-bicategories.
\end{thm}
\begin{proof}
By the way that things have been set up, Lemma \ref{lem:for-poly-composition} and Theorem \ref{thm:bicat-of-polynomials} lift to our 2-categorical setting, with the extra naturality of the coherences coming from the 2-dimensional aspect of all the universal properties being used. Thus $\Polyc{\ca K}$ is a 2-bicategory. Since in the proof of Theorem \ref{thm:polnomial-functor-homomorphism} the coherences of $\PFun{\ca E}$ were obtained from associativity coherences in $\Polyc{\ca E}$, the extra naturality enjoyed by the associativities in $\Polyc{\ca K}$ gives the extra naturality required for the coherences of $\PFun{\ca K}$. Thus $\PFun{\ca K}$ is a homomorphism of 2-bicategories.
\end{proof}
The notions \emph{polynomial 2-functor}, \emph{polynomial 2-monad} and \emph{morphism of polynomial 2-monads} are defined as in Definitions \ref{defn:polyfunctor} and \ref{defn:polymnd}. Similarly one may speak of \emph{polynomial pseudo monads} when the coherences are themselves also in the essential image of $\PFun{\ca K}$.
\begin{exam}\label{ex:Fam-Ps-Mnd}
Suppose that a morphism $U:E \to B$ in a 2-category $\ca K$ with finite limits is both exponentiable and a classifying discrete opfibration in the sense of \cite{Weber-2Toposes}. One can then define a discrete opfibration in $\ca K$ to be \emph{$U$-small} when it arises by pulling back $U$. By definition $U$-small discrete opfibrations are pullback stable in $\ca K$. Often they are also closed under composition, and when this is the case, one can consider the full sub-2-bicategory $\ca S_U$ of $\Polyc{\ca K}$ consisting of those polynomials
\[ \xygraph{{I}="p1" [r] {A}="p2" [r] {C}="p3" [r] {J}="p4" "p1":@{<-}"p2"^-{s}:"p3"^-{p}:"p4"^-{t}} \]
such that $p$ is a $U$-small discrete opfibration. The condition of being a classifying discrete opfibration then implies that the polynomial
\begin{equation}\label{eq:cdo}
\xygraph{{1}="p1" [r] {E}="p2" [r] {B}="p3" [r] {1}="p4" "p1":@{<-}"p2"^-{}:"p3"^-{U}:"p4"^-{}}
\end{equation}
is a biterminal object in $\ca S_U(1,1)$. As such it admits a canonical pseudo-monad structure. When $\ca K = \Cat$ and $U$ is the forgetful functor $\Set_{\tn{f},\bullet} \to \Set_{\tn{f}}$ from the category of finite pointed sets to that of finite sets, a $U$-small discrete opfibration is one with finite fibres, and these are evidently closed under composition. By the finitary analogue of \cite{Weber-Fam2fun} Corollary 5.12, the endofunctor associated to (\ref{eq:cdo}) is the finite ``Fam'' construction, which associates to a category its finite coproduct completion. Replacing $U$ by $U^{\op}$ gives a polynomial pseudo monad on $\Cat$ whose underlying endofunctor gives the finite product completion of a category. To summarise, the finite coproduct and finite product completion pseudo monads on $\Cat$ are polynomial pseudo monads.
\end{exam}
The 2-categorical analogues of Proposition \ref{prop:homfunctors-bi-fibs} and Corollary \ref{cor:inherit-polymnd-along-cart-monad-morphism} are also valid, with proofs adapted from the 1-dimensional case in the same way as above. We record these results as follows.
\begin{prop}\label{prop:poly-along-cart-2cat-case}
Let $\ca K$ be a 2-category with finite limits and let $X$ and $Y \in \ca K$. Let $P$ be a polynomial 2-functor $\ca K/X \to \ca K/Y$, and $T$ be a polynomial 2-monad on $\ca K/X$. Let $\phi : Q \to P$ be a 2-natural transformation, and $\psi : S \to T$ be a morphism of 2-monads.
\begin{enumerate}
\item If $\phi$ is cartesian then $Q$ is a polynomial 2-functor.
\item If $\psi$ is cartesian then $S$ is a polynomial 2-monad and $\psi$ is a morphism of polynomial 2-monads.
\end{enumerate}
\end{prop}

\subsection{Fibrations 2-monads.}
\label{ssec:fib-2-monads}
The 2-monads whose algebras are fibrations and opfibrations to be recalled here, play two roles in this work: (1) as part of the background to the discussion of familial 2-functors below in Section \ref{ssec:fam-2-fun}, and (2) as examples of polynomial 2-monads in Proposition \ref{prop:fib-monads-as-poly}. Fibrations internal to a 2-category were introduced by Street in \cite{Street-FibrationIn2cats}.

Elementary descriptions of fibrations were given in Section 2 of \cite{Weber-2Toposes} in terms of cartesian 2-cells. A similar 2-categorical reformulation of split fibrations, in which cleavages were expressed 2-categorically via the notion of chosen cartesian 2-cell, was given in Section 3 of \cite{Weber-Fam2fun}. Thus for any 2-category $\ca K$, one can define fibrations and split fibrations in $\ca K$. Moreover, given fibrations $f:A \to B$ and $g:C \to D$ in $\ca K$, a morphism of fibrations $(u,v):f \to g$ is a commutative square
\[ \xygraph{{A}="p0" [r] {C}="p1" [d] {D}="p2" [l] {B}="p3" "p0":"p1"^-{u}:"p2"^-{g}:@{<-}"p3"^-{v}:@{<-}"p0"^-{f}} \]
such that post-composition with $u$ sends $f$-cartesian 2-cells to $g$-cartesian 2-cells. When $f$ and $g$ are split fibrations, $(u,v)$ is a strict morphism when post-composing with $u$ preserves chosen cartesian 2-cells.

Thus one has 2-categories $\tn{Fib}(\ca K)$ and $\tn{SFib}(\ca K)$ of fibrations, morphisms thereof and 2-cells; and of split fibrations and strict morphisms respectively, coming with forgetful 2-functors into $\ca K^{[1]}$. Dually an \emph{opfibration} in a general 2-category $\ca K$ is a fibration in the 2-category $\ca K^{\co}$ obtained by reversing 2-cells. We freely use the associated dual notions, such as (chosen) $f$-opcartesian 2-cells below, and we define
\[ \begin{array}{lccr} {\tn{OpFib}(\ca K) = \tn{Fib}(\ca K^{\co})^{\co}} &&&
{\tn{SOpFib}(\ca K) = \tn{SFib}(\ca K^{\co})^{\co}.} \end{array} \]

When $\ca K$ has comma objects, Street \cite{Street-FibrationIn2cats} observed that fibrations can be regarded as pseudo algebras for certain easy to define 2-monads. The comma squares
\[ \xygraph{{\xybox{\xygraph{!{0;(1.5,0):(0,.6667)::} {1_B \downarrow f}="p0" [r] {A}="p1" [d] {B}="p2" [l] {B}="p3" "p0":"p1"^-{q_f}:"p2"^-{f}:@{<-}"p3"^-{1_B}:@{<-}"p0"^-{\Phi_{\ca K}(f)} "p0" [d(.5)r(.4)] :@{=>}[r(.2)]^{\lambda_f}}}}
[r(4)]
{\xybox{\xygraph{!{0;(1.5,0):(0,.6667)::} {f \downarrow 1_B}="p0" [r] {B}="p1" [d] {B}="p2" [l] {A}="p3" "p0":"p1"^-{\Psi_{\ca K}(f)}:"p2"^-{1_B}:@{<-}"p3"^-{f}:@{<-}"p0"^-{p_f}
"p0" [d(.5)r(.4)] :@{=>}[r(.2)]^{\lambda'_f}}}}} \]
describe the effect on objects of the underlying endofunctors of 2-monads $\Phi_{\ca K}$ and $\Psi_{\ca K}$ on $\ca K^{[1]}$. For $f:A \to B$ the component of the unit of $\Phi_{\ca K}$ is of the form $(\eta_f,1_B)$ where $\eta_f$ is unique such that
\[ \begin{array}{lcccr} {\Phi_{\ca K}(f)\eta_{f} = f} && {q_f\eta_{f} = 1_A} && {\lambda_f\eta_f = \id} \end{array} \]
and the $f$-component of the multiplication is of the form $(\mu_f,1_B)$ where $\mu_f$ is unique such that
\[ \begin{array}{lcccr} {\Phi_{\ca K}(f)\mu_{f} = \Phi_{\ca K}^2(f)} && {q_f\mu_f = q_fq_{\Phi_{\ca K}(f)}} && {\lambda_f\mu_f = (\lambda_fq_{\Phi_{\ca K}(f)})\lambda_{\Phi_{\ca K}(f)}.} \end{array} \]
The unit and multiplication of $\Psi_{\ca K}$ are described dually.

We follow the established notation of 2-dimensional monad theory by denoting, for a 2-monad $T$ on a 2-category $\ca K$, $\Algs T$, $\PsAlg T$ and $\tn{Kl}(T)$ the 2-categories of strict $T$-algebras and strict maps, pseudo $T$-algebras and strong{\footnotemark{\footnotetext{Meaning that the coherence data consists of invertible 2-cells}}} maps and the Kleisli 2-category of $T$ respectively. Moreover one has
\[ \begin{array}{lccr} {U^T:\Algs T \longrightarrow \ca K} &&&
{F_T:\ca K \longrightarrow \tn{Kl}(T)} \end{array} \]
the right adjoint part of the Eilenberg-Moore adjunction for $T$, and the left adjoint part of the Kleisli adjunction for $T$ respectively. We proceed now to exhibit explicit descriptions of these for the 2-monads $\Phi_{\ca K}$ and $\Psi_{\ca K}$.

Denote by $\ca K^{[1]}_{\tn{colax}}$ the 2-category of functors $[1] \to \ca K$, colax natural transformations between them, and modifications; and by $\ca K^{[1]}_{\tn{lax}}$ the 2-category of functors $[1] \to \ca K$, lax natural transformations between them, and modifications. Thus a morphism $f \to g$ of $\ca K^{[1]}_{\tn{colax}}$ is the unlabelled data in
\[ \xygraph{{A}="p0" [r] {C}="p1" [d] {D}="p2" [l] {B}="p3" "p0":"p1"^-{}:"p2"^-{g}:@{<-}"p3"^-{}:@{<-}"p0"^-{f} "p0" [d(.55)r(.35)] :@{=>}[r(.3)]^-{}} \]
and for $\ca K^{[1]}_{\tn{lax}}$ the 2-cell is in the opposite direction. Since strict natural transformations are degenerate instances of lax and colax ones, one has canonical inclusions $\ca K^{[1]} \hookrightarrow \ca \ca K^{[1]}_{\tn{colax}}$ and $\ca K^{[1]} \hookrightarrow \ca \ca K^{[1]}_{\tn{lax}}$. The following result is known, though perhaps formulated slightly more globally than usual.
\begin{prop}\label{prop:global-fibrations-monads}
\cite{KockA-FibEMAlg, Street-FibrationIn2cats}
Let $\ca K$ be a 2-category with comma objects.
\begin{enumerate}
\item One has isomorphisms
\[ \begin{array}{llll}
{\PsAlg {\Phi_{\ca K}} \iso \tn{Fib}(\ca K)} &&&
{\Algs {\Phi_{\ca K}} \iso \tn{SFib}(\ca K)} \\
{\PsAlg {\Psi_{\ca K}} \iso \tn{OpFib}(\ca K)} &&&
{\Algs {\Psi_{\ca K}} \iso \tn{SOpFib}(\ca K)} \end{array} \]
commuting with the evident 2-functors into $\ca K^{[1]}$.\label{propcase:algebras-fib-monads}
\item One has isomorphisms
\[ \begin{array}{llll}
{\tn{Kl}(\Phi_{\ca K}) \iso \ca K^{[1]}_{\tn{colax}}} &&&
{\tn{Kl}(\Psi_{\ca K}) \iso \ca K^{[1]}_{\tn{lax}}} \end{array} \]
commuting with the evident 2-functors out of $\ca K^{[1]}$.\label{propcase:kleisli-fib-monads}
\end{enumerate}
\end{prop}
\begin{proof}
In the case $\ca K = \Cat$ of (\ref{propcase:algebras-fib-monads}) is completely standard. See \cite{KockA-FibEMAlg} for a recent discussion. The result for general $\ca K$ follows by a representable argument since the 2-monads $\Phi_{\ca K}$ and $\Psi_{\ca K}$ are described in terms of limits, and the notions of cartesian and opcartesian 2-cell are representable. The isomorphisms (\ref{propcase:kleisli-fib-monads}) are easily exhibited by using the universal property of comma objects and the definition of $\tn{Kl}(\Phi_{\ca K})$ and $\tn{Kl}(\Psi_{\ca K})$.
\end{proof}
The fibre of the codomain 2-functor $\ca K^{[1]} \to \ca K$ over $B \in \ca K$ is exactly the slice 2-category $\ca K/B$, and these 2-monads restrict to the 2-monads $\Phi_{\ca K,B}$ and $\Psi_{\ca K,B}$ on $\ca K/B$ defined originally by Street in \cite{Street-FibrationIn2cats}. Using Proposition \ref{prop:global-fibrations-monads} one then has an explicit description of the algebras of $\Phi_{\ca K,B}$ (resp. $\Psi_{\ca K,B}$) as fibrations (resp. opfibrations) with codomain $B$. Similarly 1-cells $x \to y$ of $\tn{Kl}(\Phi_{\ca K,B})$ (resp. $\tn{Kl}(\Psi_{\ca K,B})$) may be identified with lax triangles
\[ \xygraph{{\xybox{\xygraph{{X}="p0" [r(2)] {Y}="p1" [dl] {B}="p2" "p0":"p1"^-{}:"p2"^-{y}:@{<-}"p0"^-{x} "p0" [d(.5)r(.85)] :@{=>}[r(.3)]^-{}}}}
[r(4)]
{\xybox{\xygraph{{X}="p0" [r(2)] {Y}="p1" [dl] {B.}="p2" "p0":"p1"^-{}:"p2"^-{y}:@{<-}"p0"^-{x} "p0" [d(.5)r(.85)] :@{<=}[r(.3)]^-{}}}}} \]

For any nice symmetric monoidal category $\ca V$ over which one may wish to enrich, one has a notion of monoidal $\ca V$-category, and so in particular taking $\ca V = \Cat$ (with the cartesian tensor product), one has a canonical notion of monoidal 2-category{\footnotemark{\footnotetext{Monoidal bicategories in the most general sense whose underlying bicategory is a 2-category are weaker than this.}}}. In this sense, for any 2-bicategory $\ca B$ and object $B$ therein, the hom $\ca B(B,B)$ is a monoidal 2-category. Following \cite{KellyLack-PropertyLikeStructures, KockA-KZMonads} we define a pseudo monoid in a monoidal 2-category $\ca K$, with unit and multiplication denoted $u:I \to M$ and $m:M \tensor M \to M$, to be \emph{colax idempotent} (resp. \emph{lax idempotent}) when $u \tensor M \ladj m \ladj M \tensor u$ (resp. $M \tensor u \ladj m \ladj u \tensor M$).
\begin{defn}\label{def:pseudo-monad-in-2-bicategory}
Let $\ca B$ be a 2-bicategory and $B \in \ca B$. Then a \emph{pseudo monad} (resp. \emph{2-monad}) on $B$ in $\ca B$ is a pseudo monoid (resp. monoid) in the monoidal 2-category $\ca B(B,B)$. A pseudo monad on $B$ in $\ca B$ is \emph{colax idempotent} (resp. \emph{lax idempotent}) when its corresponding pseudo monoid in $\ca B(B,B)$ is so.
\end{defn}
Writing $[n]$ for the ordinal $\{0 < ... < n\}$ one has a cospan in $\Cat$ as on the left
\[ \xygraph{{\xybox{\xygraph{{[0]}="l" [r] {[1]}="m" [r] {[0]}="r" "l":"m"^-{\delta_1}:@{<-}"r"^-{\delta_0}}}}
[r(3)]
{\xybox{\xygraph{{B}="l" [r] {B^{[1]}}="m" [r] {B}="r" "l":@{<-}"m"^-{d_1}:"r"^-{d_0}}}}
[r(3)]
{\xybox{\xygraph{{B^{[1]}}="p0" [r] {B}="p1" [d] {B}="p2" [l] {B}="p3" "p0":"p1"^-{d_0}:"p2"^-{1_B}:@{<-}"p3"^-{1_B}:@{<-}"p0"^-{d_1} "p0" [d(.5)r(.35)] :@{=>}[r(.3)]^{}}}}
} \]
and cotensoring this with $B$ gives the span in $\ca K$ in the middle, where $Z^A$ denotes the cotensor of $Z \in \ca K$ with the category $A$. This span fits into comma square as depicted on the right in the previous display. To some extent the following result is implicit in the work of Street \cite{Street-FibrationIn2cats, Street-CosmoiOfInternalCats, Street-FibrationInBicats}, and all that we have done is to observe that Street's approach to the fibrations 2-monads exhibits them as polynomial 2-monads. In fact, since the underlying endofunctors come from spans, they are examples of ``linear'' polynomial 2-functors.
\begin{prop}\label{prop:fib-monads-as-poly}
Let $\ca K$ be a 2-category with comma objects and pullbacks and let $B \in \ca K$.
\begin{enumerate}
\item $\Phi_{\ca K,B}$ is the result of applying $\PFun {\ca K}$ to a colax idempotent 2-monad in $\Polyc {\ca K}$ whose underlying endomorphism is
\[ \xygraph{{B}="p1" [r] {B^{[1]}}="p2" [r] {B^{[1]}}="p3" [r] {B.}="p4" "p1":@{<-}"p2"^-{d_0}:"p3"^-{1}:"p4"^-{d_1}} \]
\item $\Psi_{\ca K,B}$ is the result of applying $\PFun {\ca K}$ to a lax idempotent 2-monad in $\Polyc {\ca K}$ whose underlying endomorphism is
\[ \xygraph{{B}="p1" [r] {B^{[1]}}="p2" [r] {B^{[1]}}="p3" [r] {B.}="p4" "p1":@{<-}"p2"^-{d_1}:"p3"^-{1}:"p4"^-{d_0}} \]
\end{enumerate}
\end{prop}
\begin{proof}
The two statements are dual, so we consider just the monad $\Phi_{\ca K,B}$. By the elementary properties of comma squares and pullbacks, one can factor the defining comma square of $\Phi_{\ca K,B}(f) = \Phi_{\ca K}(f)$ as
\[ \xygraph{{1_{\ca K,B} \downarrow f}="p0" [r(1.2)] {B^{[1]}}="p1" [r] {B}="p2" [d] {B}="p3" [l] {B}="p4" [l(1.2)] {A}="p5" "p0":"p1"^-{}:"p2"^-{d_1}:"p3"^-{1_{\ca K,B}}:@{<-}"p4"^-{1_{\ca K,B}}:@{<-}"p5"^-{f}:@{<-}"p0"^-{} "p1":"p4"_-{d_0} "p0":@/^{1.5pc}/"p2"^-{\Phi_{\ca K,B}(f)}
"p0" [d(.5)r(.5)] {\scriptstyle{\tn{pb}}} "p1" [d(.5)r(.65)] :@{=>}[l(.3)]}  \]
and so one has $\Phi_{\ca K,B}(f) = \Sigma_{d_1}\Delta_{d_0}(f)$. This expresses on objects that $\Phi_{\ca K,B}$ is the result of applying $\PFun{\ca K}$ to the endopolynomial of the statement.

Recall that the inclusion $\Delta \hookrightarrow \Cat$ is a cocategory object, and as described in \cite{Lawvere-OrdSumDoct, Street-FibrationInBicats}, the standard presentation of (topologists') $\Delta$ as a subcategory of $\Cat$ has the 2-categorical feature that the successive generating coface and codegeneracy maps are adjoint. Thus the diagram
\begin{equation}\label{diag:KZ-cospan}
\xygraph{!{0;(2,0):} {[0]}="p0" [r] {[1]}="p1" [r] {[2]}="p2" [r] {[3]}="p3" [r(.5)] {......}="p4" "p2" [u] {[0]}="t" [d(2)] {[0]}="b" "p1":"p0"|-{\sigma_0}="a12"
"p2":@<2ex>"p1"^-{\sigma_1}|-{}="a22":"p2"|-{}="a23":@<-2ex>"p1"_-{\sigma_0}|-{}="a24" "a24":@{}"a23"|-{\perp}:@{}"a22"|-{\perp}
"p3":@<4ex>"p2"^-{\sigma_2}|-{}="a32":@<-2ex>"p3"|-{}="a33":"p2"|-{}="a34":@<2ex>"p3"|-{}="a35":@<-4ex>"p2"_-{\sigma_0}|-{}="a36" "a36":@{}"a35"|-{\perp}:@{}"a34"|-{\perp}:@{}"a33"|-{\perp}:@{}"a32"|-{\perp}
"t"(:@/_{2pc}/"p0"_-{t},:@/_{1pc}/"p1"_-{\delta_0=t},:"p2"^-{t},:@/^{1.5pc}/"p3"^-{t}|-{}="tr") "tr" [r(.5)] {...} "b"(:@/^{2pc}/"p0"^-{b},:@/^{1pc}/"p1"^-{\delta_1=b},:"p2"_-{b},:@/_{1.5pc}/"p3"_-{b}|-{}="br") "br" [r(.5)] {...}}
\end{equation}
in which each functor labelled as ``$t$'' picks out the top element of its codomain, and each functor labelled as ``$b$'' picks out a bottom element, is a colax idempotent 2-monad in $\Cospan{\Cat}$. Cotensoring it with an object $B$ in any finitely complete 2-category $\ca K$, gives a colax idempotent 2-monad in $\Span{\ca K}$ and by means of the inclusion $\Span{\ca K} \hookrightarrow \Polyc{\ca K}$, one has a polynomial colax idempotent 2-monad on $B$. The observation that associated 2-monad on $\ca K/B$, obtained via application of $\PFun{\ca K}$, is the 2-monad $\Phi_{\ca K,B}$ is easily verified, and was implicit in Section 2 of \cite{Street-FibrationInBicats}.
\end{proof}
\begin{rem}\label{rem:fibrations-as-coalgebras}
In general $d_1$ is a split fibration and $d_0$ is a split opfibration. In 2-categories $\ca K$, such as the case $\ca K = \Cat$, in which fibrations and opfibrations are exponentiable, it follows that $\Phi_{\ca K,B}$ and $\Psi_{\ca K,B}$ are themselves left adjoints. For such situations fibrations and opfibrations are thus also coalgebras for 2-comonads (ie those obtained from $\Phi_{\ca K,B}$ and $\Psi_{\ca K,B}$ by taking right adjoints), and the forgetful 2-functors $U^{\Phi_{\ca K,B}}$ and $U^{\Psi_{\ca K,B}}$ create all colimits.
\end{rem}

\subsection{Familial 2-functors.}
\label{ssec:fam-2-fun}
We now recall, and to some extent update, the theory of familial 2-functors from \cite{Weber-Fam2fun}. Intuitvely, a familial 2-functor is one that is compatible in an appropriate sense with the theory of fibrations recalled in the previous section. In Section \ref{ssec:fam-from-poly} we will identify conditions on polynomials in a 2-category which ensure that the corresponding 2-functor is familial.

The compatibility of familial 2-functors with the theory of fibrations is expressed formally by the formal theory of monads \cite{Street-FTM}, in which monads in a 2-category $\ca A$ were organised in various useful ways into 2-categories. Since this material is used so extensively in this section, we recall it briefly now.

We denote a monad in a 2-category $\ca A$ as a pair $(A,t)$ where $A \in \ca A$, $t$ is the underlying one-cell of the monad on $A$ in $\ca A$, and we denote the unit and multiplication 2-cells as $\eta^t$ and $\mu^t$ respectively. A \emph{lax morphism} $f : (A,t) \to (B,s)$ in $\ca A$ consists of an arrow $f : A \to B$ in $\ca A$, and a ``coherence'' 2-cell $f^l:sf \to ft$ which satisfies $f^l(\eta^sf) = f\eta^t$ and $f^l(\mu^sf) = (f\mu^t)(f^lt)(sf^l)$. A \emph{colax morphism} $f : (A,t) \to (B,s)$ in $\ca A$ consists of an arrow $f : A \to B$ in $\ca A$, and $f^c:ft \to sf$ satisfying $f^c(f\eta^t) = \eta^sf$ and $f^c(f\mu^t) = (\mu^sf)(sf^c)(f^ct)$. Lax and colax morphisms of monads were called ``monad functors'' and ``monad opfunctors'' in \cite{Street-FTM} respectively.

Given a monad $(A,t)$ in $\ca A$ one may have an associated Eilenberg-Moore object, which is a type of 2-categorical limit, whose universal $0$ and $1$-cell data is denoted as $A^t$ and $u^t : A^t \to A$ respectively. The corresponding colimit notion is that of a Kleisli object, the universal $0$ and $1$-cell data of which is denoted $A_t$ and $f_t : A \to A_t$ respectively. See \cite{Street-FTM} for the precise definitions. When $\ca A = \Cat$, $A^t$ is the category of algebras and $A_t$ is the Kleisli category, and in the general situation $u^t$ has a left adjoint and $f_t$ has a right adjoint. When $\ca A$ is the 2-category of 2-categories, a monad in $\ca A$ is a 2-monad $(\ca K, T)$, and the $0$ and $1$-cell data of the Eilenberg-Moore and Kleisli objects are denoted as
\[ \begin{array}{lccr} {U^T : \Algs T \longrightarrow \ca K} &&&
{F_T : \ca K \longrightarrow \tn{Kl}(T)} \end{array} \]
respectively.

In general, when the 2-category $\ca A$ admits Eilenberg-Moore objects, given monads $(A,t)$ and $(B,s)$ in $\ca A$, and an arrow $f : A \to B$, then 2-cells $f^l:sf \to ft$ providing coherence data of a lax monad morphism are in bijection with liftings $\overline{f}$ to the level of algebras as on the left in
\[ \xygraph{{\xybox{\xygraph{{A^t}="p0" [r] {B^s}="p1" [d] {B}="p2" [l] {A}="p3" "p0":@{.>}"p1"^-{\overline{f}}:"p2"^-{u^s}:@{<-}"p3"^-{f}:@{<-}"p0"^-{u^t}
"p0":@{}"p2"|-{=}}}} [r(3)]
{\xybox{\xygraph{{A_s}="p0" [r] {B_s}="p1" [d] {B}="p2" [l] {A}="p3" "p0":@{.>}"p1"^-{\underline{f}}:@{<-}"p2"^-{f_s}:@{<-}"p3"^-{f}:"p0"^-{f_t}
"p0":@{}"p2"|-{=}}}}} \]
and dually, when $\ca A$ admits Kleisli objects, colax coherence data $f^c:ft \to sf$ is in bijection with extensions of $f$ to $\underline{f}$ as on the right in the previous display.
\begin{rem}\label{rem:split-fib-pb-stable-from-fib-monads}
Let $\ca K$ be a 2-category with comma objects and pullbacks and $f:A \to B$ be a morphism therein. The 2-functor $\Sigma_f:\ca K/A \to \ca K/B$ clearly extends to lax triangles, and so one has $\underline{\Sigma}_f:\tn{Kl}(\Phi_{\ca K,A}) \to \tn{Kl}(\Phi_{\ca K,B})$ such that $\underline{\Sigma}_fF_{\Phi_{\ca K,A}} = F_{\Phi_{\ca K,B}}\Sigma_f$. Such an extension is equivalent to the data of a 2-natural transformation $\sigma_f$ providing the coherence datum for a colax morphism of 2-monads
\[ (\Sigma_f,\sigma_f):(\ca K/A,\Phi_{\ca K,A}) \longrightarrow (\ca K/B,\Phi_{\ca K,B}) \]
and so by taking mates \cite{KellyStreet-ElementsOf2Cats} with respect to $\Sigma_f \ladj \Delta_f$, also to the coherence datum $\delta_f$ for a lax morphism of 2-monads
\[ (\Delta_f,\delta_f):(\ca K/B,\Phi_{\ca K,B}) \longrightarrow (\ca K/A,\Phi_{\ca K,A}). \]
This last is in turn equivalent to lifting $\Delta_f$ to the level of split fibrations, that is to say, to giving a 2-functor $\overline{\Delta}_f:\Algs {\Phi_{\ca K,B}} \to \Algs {\Phi_{\ca K,A}}$ such that $U^{\Phi_A}\overline{\Delta}_f = \Delta_fU^{\Phi_B}$. The 2-functor $\overline{\Delta}_f$ witnesses the pullback stability of split fibrations.
\end{rem}
We recall now the familial 2-functors of \cite{Weber-Fam2fun}. For a 2-functor $T:\ca K \to \ca L$ and an object $X \in \ca K$, we denote by $T_X : \ca K/X \to \ca L/TX$ the 2-functor given on objects by applying $T$ to morphisms into $X$. A \emph{local right adjoint} $\ca K \to \ca L$ is a 2-functor $T:\ca K \to \ca L$ equipped with a left adjoint to $T_X$ for all $X \in \ca K$. Given an adjunction as on the left
\[ \xygraph{*{\xybox{\xygraph{{\ca A}="p0" [r(2)] {\ca B}="p1" "p0":@<-1.2ex>"p1"_-{R}|-{}="bot":@<-1.2ex>"p0"_-{L}|-{}="top" "top":@{}"bot"|-{\perp}}}} [r(5)]
*!(0,.175){\xybox{\xygraph{{\ca A/X}="p0" [r(2)] {\ca B/RX}="p1" "p0":@<-1.2ex>"p1"_-{R_X}|-{}="bot":@<-1.2ex>"p0"_-{}|-{}="top" "top":@{}"bot"|-{\perp}}}}} \]
and $X \in \ca A$, one always has an adjunction as on the right in the previous display. Thus when $\ca K$ has a terminal object $1$, to exhibit $T$ as a local right adjoint, it suffices to give a left adjoint to $T_1$.
\begin{defn}\label{def:Familial}
Let $\ca K$ and $\ca L$ be 2-categories with comma objects, and suppose that $\ca K$ has a terminal object $1$. Then a \emph{familial 2-functor} $\ca K \to \ca L$ consists of local right adjoint $T:\ca K \to \ca L$ together with $\overline{T}_1 : \ca K \to \Algs {\Phi_{T1}}$ such that $U^{\Phi_{T1}}\overline{T}_1 = T_1$.
\end{defn}
Familial 2-functors are exactly those 2-functors which are compatible with fibrations because of
\begin{prop}\label{prop:Fam2fun-fib}
Let $\ca K$ and $\ca L$ be 2-categories with comma objects, suppose that $\ca K$ has a terminal object $1$ and let $T$ be a 2-functor. To give $T$ the structure of a familial 2-functor is to give the coherence data $\phi^T$ for a lax morphism
\[ (T^{[1]},\phi^T) : (\ca K^{[1]},\Phi_{\ca K}) \longrightarrow (\ca L^{[1]},\Phi_{\ca L}) \]
of 2-monads.
\end{prop}
\begin{proof}
It is straight forward to adapt Theorem 6.6 of \cite{Weber-Fam2fun} to provide, given $\overline{T}_1$, the lifting of $T^{[1]}$ to a 2-functor $\Algs {\Phi_{\ca K}} \to \Algs {\Phi_{\ca L}}$. Conversely, given such a lifting one recovers $\overline{T}_1$ by restricting to split fibrations in $\ca K$ into $1$.
\end{proof}
Implicit in Definition 5.2 of \cite{Weber-Fam2fun} is the idea that morphisms of familial 2-functors are simply cartesian 2-natural transformations. In this work we wish to adopt a more specialised notion which is compatible with Propostion \ref{prop:Fam2fun-fib}.
\begin{defn}\label{def:Fam-nat}
Let $S$ and $T$ be familial 2-functors $\ca K \to \ca L$. Then a 2-natural transformation $\alpha:S \to T$ is \emph{familial} when it is cartesian and for all $X \in \ca K$, $\alpha$'s naturality square with respect to the unique map $t_X : X \to 1$ is a morphism of split fibrations $(\alpha_X,\alpha_1) : St_X \to Tt_X$.
\end{defn}
\begin{lem}\label{lem:fam-2-cell}
Let $\alpha:S \to T$ be a familial natural transformation between familial 2-functors $\ca K \to \ca L$, and $f:A \to B$ be a split fibration in $\ca K$. Then the naturality square
\[ \xygraph{!{0;(1.5,0):(0,.6667)::} {SA}="p0" [r] {TA}="p1" [d] {TB}="p2" [l] {SB}="p3" "p0":"p1"^-{\alpha_A}:"p2"^-{Tf}:@{<-}"p3"^-{\alpha_B}:@{<-}"p0"^-{Sf}} \]
is a morphism $(\alpha_A,\alpha_B) : Sf \to Tf$ of split fibrations in $\ca L$.
\end{lem}
\begin{proof}
Recall from Lemma 6.3 of \cite{Weber-Fam2fun} that a 2-cell $\psi$ as indicated on the left in
\[ \xygraph{{\xybox{\xygraph{{X}="p0" [r] {SA}="p1" "p0":@/^{1pc}/"p1"^-{x}|(.4){}="t" "p0":@/_{1pc}/"p1"_-{y}|(.4){}="b" "t":@{}"b"|(.2){}="dom"|(.8){}="cod" "dom":@{=>}"cod"^{\psi}}}} [r(4)]
{\xybox{\xygraph{!{0;(2,0):(0,.5)::} {X}="p0" [r] {SD}="p1" [d] {SA}="p2" [l] {SC}="p3" "p0":"p1"^-{g_2}:"p2"^-{Sh_2}:@{<-}"p3"^-{Sh_1}:@{<-}"p0"^-{g_1} "p3":"p1"|-{S\delta}
"p0" [d(.4)r(.15)] :@{=>}[r(.2)]^-{\psi_1} "p1" [d(.7)l(.35)] :@{=>}[r(.2)]^-{S\psi_2}}}}} \]
is chosen $Sf$-cartesian if and only if $\psi_2$ is chosen $f$-cartesian, where $x = S(h_1)g_1$ and $y = S(h_2)g_2$ are $S$-generic factorisations in the sense of \cite{Weber-Fam2fun}, and $\psi_1$ and $\psi_2$ are unique such that the composite on the right in the previous display equal to $\psi$ and $\psi_1$ is chosen $St_D$-cartesian. We must show that if $\psi$ is chosen $Sf$-cartesian, then $\alpha_A\psi$ is chosen $Tf$-cartesian. In
\[ \xygraph{!{0;(2,0):(0,.5)::}
{X}="p0" [r] {SD}="p1" [d] {SA}="p2" [l] {SC}="p3" "p0":"p1"^-{g_2}:"p2"^-{Sh_2}:@{<-}"p3"^-{Sh_1}:@{<-}"p0"^-{g_1} "p3":"p1"|-{S\delta}
"p0" [d(.4)r(.15)] :@{=>}[r(.2)]^-{\psi_1} "p1" [d(.7)l(.35)] :@{=>}[r(.2)]^-{S\psi_2}
"p1" [d(1)r(.8)] {TD}="q1" [d] {TA}="q2" [l] {TC}="q3" "q1":"q2"^-{Th_2}:@{<-}"q3"^-{Th_1}:"q1"^(.4){T\delta}
"q1" [d(.68)l(.35)] :@{=>}[r(.2)]^-{T\psi_2}
"p1":"q1"^-{\alpha_D} "p2":"q2"|(.45)*=<3pt>{}|(.7)*=<15pt>{} "p3":"q3"_-{\alpha_C}} \]
$\alpha_D\psi_1$ is chosen $Tt_D$-cartesian since $(\alpha_D,\alpha_1)$ is a morphism of split fibrations $St_D \to Tt_D$, and so the result follows by Lemma 6.3 of \cite{Weber-Fam2fun} applied to $T$.
\end{proof}
When $T:\ca K \to \ca L$ is familial we denote by $\overline{T} : \Algs {\Phi_{\ca K}} \to \Algs {\Phi_{\ca L}}$ the lifting of $T^{[1]}$ corresponding to $\phi^T$. Then one obtains the following result immediately from Lemma \ref{lem:fam-2-cell}.
\begin{prop}\label{prop:fam-2-cell-fib}
Let $S$ and $T$ be familial 2-functors and $\alpha : S \to T$ be a cartesian natural transformation between them. Then $\alpha$ is familial if and only if it lifts to a 2-natural transformation $\overline{\alpha}:\overline{S} \to \overline{T}$.
\end{prop}
We freely use and apply the dual notions and results, an opfamilial 2-functor being one which is compatible in the same way with opfibrations. A familial (resp. opfamilial) 2-monad is one whose underlying endo-2-functor, unit and multiplication are familial (resp. familial).
\begin{rem}\label{rem:examples-from-fam2fun}
All the examples familial and opfamilial 2-monads exhibited in \cite{Weber-Fam2fun} turn out to be familial and opfamilial 2-monads in the more restricted sense of this article. For those that arise from polynomials, we shall establish this from general results below.
\end{rem}

\subsection{Familial 2-functors from polynomials.}
\label{ssec:fam-from-poly}
In this section we identify conditions on polynomials in a 2-category and their morphisms, ensuring that the associated 2-functors and 2-natural transformations are familial. The central result to this effect appears in this section as Theorem \ref{thm:fibrational-polynomials}. The task of proving this result breaks up into that of identifying conditions ensuring that 2-functors of the form $\Delta_f$ and $\Pi_f$, and the ``Beck-Chevalley 2-cells'' lift to the level of split fibrations. These conditions are presented in turn in Lemmas \ref{lem:PTJ-fib}-\ref{lem:right-part-poly-2-cell} below. 
\begin{lem}\label{lem:PTJ-fib}
\cite{Johnstone-FibrationsPartialProducts}
Let $f:A \to B$ be a split opfibration in a 2-category $\ca K$ with pullbacks and comma objects.
\begin{enumerate}
\item One has $\underline{\Delta}_f : \tn{Kl}(\Phi_{\ca K,B}) \to \tn{Kl}(\Phi_{\ca K,A})$ such that $\underline{\Delta}_fF_{\Phi_{\ca K,B}} = F_{\Phi_{\ca K,A}}\Delta_f$.\label{lemcase:PTJ-split}
\item If $f$ is a discrete opfibration then $\underline{\Sigma}_f \ladj \underline{\Delta}_f$.\label{lemcase:PTJ-discrete}
\end{enumerate}
\end{lem}
\begin{proof}
(\ref{lemcase:PTJ-split}): Given $(k,\gamma) : x \to y$ in $\tn{Kl}(\Phi_{\ca K,B})$, we pullback to obtain $X_2$ and $Y_2$ in
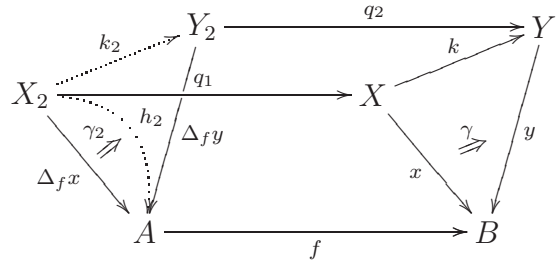
\[ \xygraph{!{0;(.6,0):(0,1.2)::} {X_2}="ltl" [r(3)u] {Y_2}="ltr" [d(3)l] {A}="lb" [r(4)u(2)] {X}="rtl" [r(3)u] {Y}="rtr" [d(3)l] {B}="rb" "ltl":"rtl"^-{q_1}:"rb"|-{}="d1"_-{x}:@{<-}^-{f}"lb":@{<-}"ltl"|-{}="d2"^-{\Delta_fx}:@/^{1.75pc}/@{.>}"lb"|-{}="c22"^-{h_2}
"ltr":"rtr"^-{q_2}:"rb"|-{}="c1"^-{y}:@{<-}"lb":@{<-}"ltr"|-{}="c2"_-{\Delta_fy}|(.67)*=<3pt>{} "rtl":"rtr"^-{k} "ltl":@{.>}"ltr"^-{k_2}
"d1":@{}"c1"|(.35){}="d3"|(.65){}="c3" "d3":@{=>}"c3"^-{\gamma}
"d2":@{}"c22"|(.25){}="d4"|(.75){}="c4" "d4":@{=>}"c4"^-{\gamma_2}} \]
$\gamma_2$ is the chosen $f$-opcartesian lift of $\gamma q_1$, and then $k_2$ is unique such that $\Delta_f(y)k_2 = h_2$ and $q_2k_2 = kq_1$. One defines $\underline{\Delta}_f(k,\gamma) = (k_2,\gamma_2)$. Given a 2-cell $\kappa:(k,\gamma) \to (k',\gamma')$, then by the opcartesianness of $\gamma_2$ one has $\kappa_2:\Delta_f(y)k_2 \to \Delta_f(y)k'_2$ unique such that $\kappa_2\gamma_2 = \gamma'_2$ and $f\kappa_2 = y\kappa q_1$. By the 2-dimensional universal property of the pullback defining $Y_2$, one has $\kappa_3:k_2 \to k'_2$ unique such that $\Delta_f(y)\kappa_3 = \kappa_2$ and $q_2\kappa_3 = \kappa q_1$. We then define $\underline{\Delta}_f(\kappa) = \kappa_3$. The 2-functoriality of these assignations follows easily from the uniqueness of the various lifts in these constructions and the 2-dimensional universal property of pullbacks in $\ca K$.

(\ref{lemcase:PTJ-discrete}): Note that in the above diagram $q_1$ and $q_2$ are the components of the counit of $\Sigma_f \ladj \Delta_f$ at $x$ and $y$ respectively. Thus the commutativity of the diagram ensures that the counit is natural with respect to lax triangles, and 2-naturality follows similarly. It remains to verify the 2-naturality of the unit of $\Sigma_f \ladj \Delta_f$ with respect to lax triangles when $f$ is a discrete opfibration. Denoting the defining pullback of $\Delta_f\Sigma_fx$ as on the left in
\[ \xygraph{{\xybox{\xygraph{{X_2}="p0" [r] {X}="p1" [d] {B}="p2" [l] {A}="p3" "p0":"p1"^-{q_x}:"p2"^-{fx}:@{<-}"p3"^-{f}:@{<-}"p0"^-{\Delta_f\Sigma_fx}:@{}"p2"|-{\tn{pb}}}}}
[r(3.5)u(.1)]
{\xybox{\xygraph{{X}="p0" [r(2)] {Y}="p1" [dl] {A}="p2" "p0":"p1"^-{k}:"p2"^-{y}:@{<-}"p0"^-{x} "p0" [d(.5)r(.85)] :@{=>}[r(.3)]^-{\gamma}}}}} \]
$\eta_x:X \to X_2$ is unique such that $\Delta_f\Sigma_f(x)\eta_x = x$ and $q_x\eta_x = 1_X$. Naturality with respect to $(k,\gamma)$ amounts to checking that $\eta_yk = k_2\eta_x$ and $\gamma = \gamma_2\eta_x$ in
\[ \xygraph{!{0;(.4,0):(0,1.5)::}
{X}="lt1" [r(3)u] {Y}="rt1" [d(2)l(1)] {A}="bt1"
"lt1":"rt1"^(.4){k}:"bt1"^-{y}|-{}="pc1":@{<-}"lt1"^-{x}|-{}="pd1"
"pd1":@{}"pc1"|(.25){}="d1"|(.75){}="c1" "d1":@{=>}"c1"^{\gamma}
"bt1" :[d(1.5)] {B}="bot"^{f} :@{<-}[l(6)] {A}="bt2"^-{f}
"lt1" :@{<-}[l(6)] {X_2}="lt2"_(.4){q_x} "rt1" :@{<-}[l(6)] {Y_2}="rt2"_-{q_y}
"lt2":"rt2"^(.4){k_2}:"bt2"|(.285)*=<3pt>{}^-{\Delta_f\Sigma_fy}|-{}="pc2":@{<-}"lt2"^-{}|-{}="pd2"
"pd2":@{}"pc2"|(.25){}="d2"|(.75){}="c2" "d2":@{=>}"c2"^{\gamma_2}
"lt2" :@{<-}[l(6)] {X}="lt3"_-{\eta_x} "rt2" :@{<-}[l(6)] {Y}="rt3"_-{\eta_y}
"bt2" {}="bt3" "lt3":"rt3"^(.4){k}:"bt3"|(.285)*=<3pt>{}^-{}|-{}="pc3":@{<-}"lt3"^-{x}|-{}="pd3"
"pd3":@{}"pc3"|(.25){}="d3"|(.75){}="c3" "d3":@{=>}"c3"^{\gamma}} \]
in which $\gamma_2$ is unique such that $f\gamma_2 = f\gamma q_x$. Thus $f\gamma_2\eta_x = f\gamma$ and so since $f$ is a discrete opfibration, the uniqueness of lifts implies $\gamma_2\eta_x = \gamma$ and $\Delta_f\Sigma_f(y)k_2\eta_x = ky$. From this last one obtains the first equation of
\[ \begin{array}{lccr} {\Delta_f\Sigma_f(y)\eta_y k = \Delta_f\Sigma_f(y)k_2\eta_x} &&&
{q_yk_2\eta_x = q_y\eta_yk} \end{array} \]
and the second equation is also easily verified, so that by the joint monicness of $(\Delta_f\Sigma_f(y),q_y)$, one has $\eta_yk = k_2\eta_x$. The verification of the 2-dimensional part of naturality proceeds straight forwardly along similar lines, and so is left to the reader.
\end{proof}
\begin{lem}\label{lem:PTJ-exp-fib}
\cite{Johnstone-FibrationsPartialProducts}
Let $f:A \to B$ be an exponentiable split opfibration in a 2-category $\ca K$ with pullbacks and comma objects.
\begin{enumerate}
\item One has $\overline{\Pi}_f : \Algs {\Phi_{\ca K,A}} \to \Algs {\Phi_{\ca K,B}}$ such that $U^{\Phi_{\ca K,B}}\overline{\Pi}_f = \Pi_fU^{\Phi_{\ca K,A}}$.\label{lemcase:PTJ-exp-split}
\item If $f$ is a discrete opfibration then $\overline{\Delta}_f \ladj \overline{\Pi}_f$.\label{lemcase:PTJ-exp-discrete}
\end{enumerate}
\end{lem}
\begin{proof}
(\ref{lemcase:PTJ-exp-split}): By the formal theory of monads the data of the extension $\underline{\Delta}_f$ of Lemma \ref{lem:PTJ-fib}(\ref{lemcase:PTJ-split}) is equivalent to that of a 2-natural transformation $\partial_f$ providing the coherence datum for an oplax morphism
\[ (\Delta_f,\partial_f):(\ca K/B,\Phi_{\ca K,B}) \longrightarrow (\ca K/A,\Phi_{\ca K,A}). \]
Since $f$ is exponentiable, this is in turn equivalent, by taking mates with respect to the adjunction $\Delta_f \ladj \Pi_f$, to the data of a 2-natural transformation $\pi_f$ providing the coherence datum of a lax morphism
\[ (\Pi_f,\pi_f):(\ca K/A,\Phi_{\ca K,A}) \longrightarrow (\ca K/B,\Phi_{\ca K,B}). \]
This last is in turn equivalent to the lifting $\overline{\Pi}_f$ of $\Pi_f$ to the level of split fibrations.

(\ref{lemcase:PTJ-exp-discrete}): The data of the adjunction $\underline{\Sigma}_f \ladj \underline{\Delta}_f$ established in Lemma \ref{lem:PTJ-fib}(\ref{lemcase:PTJ-discrete}) corresponds by the formal theory of monads to the data of an adjunction of oplax morphisms of 2-monads whose underlying 2-adjunction is $\Sigma_f \ladj \Delta_f$. This means that in addition to having the oplax coherence data $\sigma_f$ and $\partial_f$ described above, one has the compatibility of the unit $\eta_f$ and counit $\varepsilon_f$ of the adjunction $\Sigma_f \ladj \Delta_f$ with respect to these coherences, which in explicit terms is the commutativity of
\[ \xygraph{{\xybox{\xygraph{!{0;(2,0):(0,.5)::} {\Phi_{\ca K,A}}="p0" [r] {\Delta_f\Sigma_f\Phi_{\ca K,A}}="p1" [d] {\Delta_f\Phi_{\ca K,B}\Sigma_f}="p2" [l] {\Phi_{\ca K,A}\Delta_f\Sigma_f}="p3" "p0":"p1"^-{\eta_f\Phi_{\ca K,A}}:"p2"^-{\Delta_f\sigma_f}:"p3"^-{\partial_f\Sigma_f}:@{<-}"p0"^-{\Phi_{\ca K,A}\eta_f}}}}
[r(5)d(.05)]
{\xybox{\xygraph{!{0;(2,0):(0,.5)::} {\Sigma_f\Delta_f\Phi_{\ca K,B}}="p0" [r] {\Sigma_f\Phi_{\ca K,A}\Delta_f}="p1" [d] {\Phi_{\ca K,B}\Sigma_f\Delta_f.}="p2" [l] {\Phi_{\ca K,B}}="p3" "p0":"p1"^-{\Sigma_f\partial_f}:"p2"^-{\sigma_f\Delta_f}:"p3"^-{\Phi_{\ca K,B}\varepsilon_f}:@{<-}"p0"^-{\varepsilon_f\Phi_{\ca K,B}}}}}} \]
In terms of string diagrams (which go from top to bottom) in the 2-category of 2-categories (see \cite{JoyalStreet-GeometryTensorCalculus}), the above commutative diagrams are expressed as
\begin{equation}\label{eq:oplax-adjunction}
\xygraph{{\xybox{
\xygraph{
{\xybox{\xygraph{!{0;(.5,0):(0,1.2)::} 
{\scriptstyle{\eta_f}} *\xycircle<6pt,5pt>{-}="p0" [dr]
{\scriptstyle{\sigma_f}} *\xycircle<6pt,5pt>{-}="p1" [dl]
{\scriptstyle{\partial_f}} *\xycircle<6pt,5pt>{-}="p2"
"p0"-"p1"-"p2"
"p1" (-[u(2)r] {\Phi_{\ca K,A}},-[d(2)r] {\Sigma_f})
"p2" (-[l(.75)d] {\Phi_{\ca K,A}},-[r(.75)d] {\Delta_f},-@`{"p2"+(-1,1)}"p0")}}}
[r] {=} [r(1.25)]
{\xybox{\xygraph{!{0;(.5,0):(0,1.2)::} 
{\scriptstyle{\eta_f}} *\xycircle<6pt,5pt>{-}="p0"
"p0" (-[l(.75)d(2)] {\Delta_f},-[r(.75)d(2)] {\Sigma_f},[l(2)u(2)] {\Phi_{\ca K,A}} -[d(4)] {\Phi_{\ca K,A}})}}}}
}}
[r(5)]
{\xybox{
\xygraph{
{\xybox{\xygraph{!{0;(-.5,0):(0,1.2)::} 
{\scriptstyle{\varepsilon_f}} *\xycircle<6pt,5pt>{-}="p0" [dr]
{\scriptstyle{\sigma_f}} *\xycircle<6pt,5pt>{-}="p1" [dl]
{\scriptstyle{\partial_f}} *\xycircle<6pt,5pt>{-}="p2"
"p0"-"p1"-"p2"
"p1" (-[u(2)r] {\Phi_{\ca K,B}},-[d(2)r] {\Sigma_f})
"p2" (-[l(.75)d] {\Phi_{\ca K,B}},-[r(.75)d] {\Delta_f},-@`{"p2"+(-1,1)}"p0")}}}
[r(1.125)] {=} [r(1.125)]
{\xybox{\xygraph{!{0;(-.5,0):(0,1.2)::} 
{\scriptstyle{\varepsilon_f}} *\xycircle<6pt,5pt>{-}="p0"
"p0" (-[l(.75)d(2)] {\Delta_f},-[r(.75)d(2)] {\Sigma_f},[l(2)u(2)] {\Phi_{\ca K,B}} -[d(4)] {\Phi_{\ca K,B}})}}}}}}}
\end{equation}
It suffices now to verify that the lax coherence data $\delta_f$ and $\pi_f$ are compatible with the unit $\eta'_f$ and counit $\varepsilon'_f$ of $\Delta_f \ladj \Sigma_f$ in the same way, which is to say that
\begin{equation}\label{eq:lax-adjunction}
\xygraph{{\xybox{
\xygraph{
{\xybox{\xygraph{!{0;(-.5,0):(0,-1.2)::} 
{\scriptstyle{\eta'_f}} *\xycircle<6pt,5pt>{-}="p0" [dr]
{\scriptstyle{\pi_f}} *\xycircle<6pt,5pt>{-}="p1" [dl]
{\scriptstyle{\delta_f}} *\xycircle<6pt,5pt>{-}="p2"
"p0"-"p1"-"p2"
"p1" (-[u(2)r] {},-[d(2)r] {})
"p2" (-[l(.75)d] {},-[r(.75)d] {},-@`{"p2"+(-1,1)}"p0")}}}
[r] {=} [r(.75)]
{\xybox{\xygraph{!{0;(-.5,0):(0,-1.2)::} 
{\scriptstyle{\eta'_f}} *\xycircle<6pt,5pt>{-}="p0"
"p0" (-[l(.75)d(2)] {},-[r(.75)d(2)] {},[l(1.5)u(2)] {} -[d(4)] {})}}}}
}}
[r(5)]
{\xybox{
\xygraph{
{\xybox{\xygraph{!{0;(.5,0):(0,-1.2)::} 
{\scriptstyle{\varepsilon'_f}} *\xycircle<6pt,5pt>{-}="p0" [dr]
{\scriptstyle{\pi_f}} *\xycircle<6pt,5pt>{-}="p1" [dl]
{\scriptstyle{\delta_f}} *\xycircle<6pt,5pt>{-}="p2"
"p0"-"p1"-"p2"
"p1" (-[u(2)r] {},-[d(2)r] {})
"p2" (-[l(.75)d] {},-[r(.75)d] {},-@`{"p2"+(-1,1)}"p0")}}}
[r] {=} [r]
{\xybox{\xygraph{!{0;(.5,0):(0,-1.2)::} 
{\scriptstyle{\varepsilon'_f}} *\xycircle<6pt,5pt>{-}="p0"
"p0" (-[l(.75)d(2)] {},-[r(.75)d(2)] {},[l(1.5)u(2)] {} -[d(4)] {})}}}}}}}
\end{equation}
because then by the formal theory of monads this corresponds to giving the rest of the data of the required adjunction $\overline{\Delta}_f \ladj \overline{\Pi}_f$. The definition of $\delta_f$ and $\pi_f$ as mates of $\sigma_f$ and $\partial_f$ in explicit terms says that
\begin{equation}\label{eq:oplax-lax-mates}
\xygraph{{\xybox{
\xygraph{
{\xybox{\xygraph{!{0;(.5,0):(0,1.2)::} 
{\scriptstyle{\eta_f}} *\xycircle<6pt,5pt>{-}="p0" [dr]
{\scriptstyle{\sigma_f}} *\xycircle<6pt,5pt>{-}="p1" [dr]
{\scriptstyle{\varepsilon_f}} *\xycircle<6pt,5pt>{-}="p2"
"p0"(-"p1"-"p2",-[d(3)l])
"p1" (-[d(2)l(.75)],-[u(2)r(.75)])
"p2" -[u(3)r]}}}
[r(1.5)] {=} [r]
{\xybox{\xygraph{!{0;(.5,0):(0,1.2)::} 
{\scriptstyle{\delta_f}} *\xycircle<6pt,5pt>{-}="p0"
"p0" (-[l(.75)d(2)],-[r(.75)d(2)],-[l(.75)u(2)],-[r(.75)u(2)])}}}}}}
[r(5)]
{\xybox{
\xygraph{
{\xybox{\xygraph{!{0;(.5,0):(0,1.2)::} 
{\scriptstyle{\eta'_f}} *\xycircle<6pt,5pt>{-}="p0" [dr]
{\scriptstyle{\partial_f}} *\xycircle<6pt,5pt>{-}="p1" [dr]
{\scriptstyle{\varepsilon'_f}} *\xycircle<6pt,5pt>{-}="p2"
"p0"(-"p1"-"p2",-[d(3)l])
"p1" (-[d(2)l(.75)],-[u(2)r(.75)])
"p2" -[u(3)r]}}}
[r(1.5)] {=} [r]
{\xybox{\xygraph{!{0;(.5,0):(0,1.2)::} 
{\scriptstyle{\pi_f}} *\xycircle<6pt,5pt>{-}="p0"
"p0" (-[l(.75)d(2)],-[r(.75)d(2)],-[l(.75)u(2)],-[r(.75)u(2)])}}}}}}}
\end{equation}
and so the left equation of (\ref{eq:lax-adjunction}) is established by the calculation
\[ \xygraph{
{\xybox{\xygraph{!{0;(-.35,0):(0,-1.2)::} 
{\scriptstyle{\eta'}} *\xycircle<4pt,6pt>{-}="p0" [dr]
{\scriptstyle{\pi}} *\xycircle<4pt,4pt>{-}="p1" [dl]
{\scriptstyle{\delta}} *\xycircle<4pt,4pt>{-}="p2"
"p0"-"p1"-"p2"
"p1" (-[u(2)r] {},-[d(2)r] {})
"p2" (-[l(.75)d] {},-[r(.75)d] {},-@`{"p2"+(-1,1)}"p0")}}}
[r(.8)] {=} [r(1.2)]
{\xybox{\xygraph{!{0;(.35,0):(0,1.2)::} 
{\scriptstyle{\eta'}} *\xycircle<4pt,6pt>{-}="p0" [dr]
{\scriptstyle{\partial}} *\xycircle<4pt,4pt>{-}="p1" [dr]
{\scriptstyle{\varepsilon'}} *\xycircle<4pt,6pt>{-}="p2" [ur]
{\scriptstyle{\eta'}} *\xycircle<4pt,6pt>{-}="p3"
"p0" [d(2)] {\scriptstyle{\eta}} *\xycircle<4pt,4pt>{-}="p4"
"p1" [d(2)] {\scriptstyle{\sigma}} *\xycircle<4pt,4pt>{-}="p5"
"p2" [d(2)] {\scriptstyle{\varepsilon}} *\xycircle<4pt,4pt>{-}="p6"
"p0"-"p1"-"p2"-"p3" "p4"-"p5"-"p6"
"p0" -@`{"p1"+(-2,0),"p5"+(-2,0)}[d(5)l]
"p1" (-[u(2)r],-"p5")
"p3" -@`{"p3"+(1,-1)}"p6"
"p4" -[d(3)l(.25)]
"p5" -[d(2)l(.5)]}}}
[r(1.1)] {=} [r(1.1)]
{\xybox{\xygraph{!{0;(.35,0):(0,1.2)::} 
{\scriptstyle{\eta'}} *\xycircle<4pt,6pt>{-}="p0" [r(2)d]
{\scriptstyle{\partial}} *\xycircle<4pt,4pt>{-}="p1"
"p0" [d] {\scriptstyle{\eta}} *\xycircle<4pt,4pt>{-}="p4" [dr]
{\scriptstyle{\sigma}} *\xycircle<4pt,4pt>{-}="p5" [dr]
{\scriptstyle{\varepsilon}} *\xycircle<4pt,4pt>{-}="p6"
"p0"-"p1" "p4"-"p5"-"p6"
"p0" -@`{"p4"+(-1,0),"p5"+(-2,0)}[d(4)l]
"p1" (-[u(2)r(.5)],-"p5",-@`{"p5"+(2,0)}"p6")
"p4" -[d(3)l(.25)]
"p5" -[d(2)l(.5)]}}}
[r(.9)] {=} [r(.9)]
{\xybox{\xygraph{!{0;(.35,0):(0,1.2)::}
{\scriptstyle{\eta'}} *\xycircle<4pt,6pt>{-}="p0" [d]
{\scriptstyle{\eta}} *\xycircle<4pt,4pt>{-}="p1" [dr]
{\scriptstyle{\varepsilon}} *\xycircle<4pt,4pt>{-}="p2"
"p0" (-@`{"p1"+(-1,0)}[d(3)l],-@`{"p0"+(1,-.5),"p2"+(1,.5)}"p2")
"p1" (-[d(2)l(.5)],-"p2" [r(.75)d] -[u(4)])}}}
[r(.9)] {=} [r(.8)]
{\xybox{\xygraph{!{0;(-.35,0):(0,-1.2)::} 
{\scriptstyle{\eta'}} *\xycircle<4pt,6pt>{-}="p0"
"p0" (-[l(.75)d(2)] {},-[r(.75)d(2)] {},[l(1.5)u(2)] {} -[d(4)] {})}}}} \]
which uses (\ref{eq:oplax-adjunction}), (\ref{eq:oplax-lax-mates}) and the adjunction triangle equations. The right equation in (\ref{eq:lax-adjunction}) follows from a similar calculation.
\end{proof}
Recall that given a commutative square in $\ca K$ as on the left
\begin{equation}\label{eq:Beck-Chev}
\xygraph{!{0;(3,0):}
{\xybox{\xygraph{!{0;(1.2,0):} {A}="tl" [r] {C}="tr" [d] {D}="br" [l] {B}="bl" "tl":"tr"^-{h}:"br"^-{g}:@{<-}"bl"^-{k}:@{<-}"tl"^-{f}}}} [r]
{\xybox{\xygraph{!{0;(1.2,0):} {\ca K/A}="tl" [r] {\ca K/C}="tr" [d] {\ca K/D}="br" [l] {\ca K/B}="bl" "tl":"tr"^-{\Sigma_h}:@{<-}"br"^-{\Delta_g}:@{<-}"bl"^-{\Sigma_k}:"tl"^-{\Delta_f} [d(.5)r(.35)] :@{=>}[r(.3)]^-{\alpha}}}} [r]
{\xybox{\xygraph{!{0;(1.2,0):} {\ca K/A}="tl" [r] {\ca K/C}="tr" [d] {\ca K/D}="br" [l] {\ca K/B}="bl" "tl":@{<-}"tr"^-{\Delta_h}:"br"^-{\Pi_g}:"bl"^-{\Delta_k}:@{<-}"tl"^-{\Pi_f} [d(.5)r(.35)] :@{<=}[r(.3)]^-{\beta}}}}}
\end{equation}
one has a canonical 2-natural transformation $\alpha$ is obtained from $\Sigma_k\Sigma_f = \Sigma_g\Sigma_h$ and the adjunctions $\Sigma_f \ladj \Delta_f$ and $\Sigma_g \ladj \Delta_g$. If $f$ and $g$ are exponentiable, then by taking right adjoints one has a canonical 2-natural transformation $\beta$ as on the right.
\begin{lem}\label{lem:Beck-Chev-lax-slices}
If in a 2-category $\ca K$ with comma objects and pullbacks, the left-most square in (\ref{eq:Beck-Chev}) underlies a morphism of split opfibrations $f \to g$, then $\alpha$ extends to a 2-natural transformation $\underline{\alpha}$ as on the left
\[ \xygraph{{\xybox{\xygraph{!{0;(2,0):(0,.5)::} {\tn{Kl}(\Phi_{\ca K,A})}="tl" [r] {\tn{Kl}(\Phi_{\ca K,C})}="tr" [d] {\tn{Kl}(\Phi_{\ca K,D})}="br" [l] {\tn{Kl}(\Phi_{\ca K,B})}="bl" "tl":"tr"^-{\underline{\Sigma}_h}:@{<-}"br"^-{\underline{\Delta}_g}:@{<-}"bl"^-{\underline{\Sigma}_k}:"tl"^-{\underline{\Delta}_f} [d(.5)r(.4)] :@{=>}[r(.2)]^-{\underline{\alpha}}}}}
[r(4)]
{\xybox{\xygraph{!{0;(2,0):(0,.5)::}
{\Algs {\Phi_{\ca K,A}}}="tl" [r] {\Algs {\Phi_{\ca K,C}}}="tr" [d] {\Algs {\Phi_{\ca K,D}}}="br" [l] {\Algs {\Phi_{\ca K,B}}}="bl"
"tl":@{<-}"tr"^-{\overline{\Delta}_h}:"br"^-{\overline{\Pi}_g}:"bl"^-{\overline{\Delta}_k}:@{<-}"tl"^-{\overline{\Pi}_f} [d(.5)r(.4)] :@{<=}[r(.2)]^-{\overline{\beta}}
}}}} \]
and if $f$ and $g$ are exponentiable, then $\beta$ lifts to a 2-natural transformation $\overline{\beta}$ as on the right.
\end{lem}
\begin{proof}
By the formal theory of monads $\alpha$ extends to $\underline{\alpha}$ if and only if $\alpha$ underlies a monad 2-cell, that is to say, satisfies axioms of compatibility with the colax monad morphism structures on $\Sigma_h\Delta_f$ and $\Delta_g\Sigma_k$. When $f$ and $g$ are exponentiable this in turn is equivalent, by the calculus of mates, to asking that $\beta$ satisfy compatibility with respect to the lax monad morphism structures on $\Pi_f\Delta_h$ and $\Delta_k\Pi_g$, which is moreover equivalent to the existence of the lifting $\overline{\beta}$. Thus it suffices to exhibit $\underline{\alpha}$, in other words, that $\alpha$ is 2-natural with respect to lax triangles.

For the 1-dimensional part of this naturality note that the component of $\alpha$ at $x:X \to B$ is as shown on the right
\[ \xygraph{{\xybox{\xygraph{!{0;(1.5,0):(0,.6667)::} {X_2}="p0" [r] {X}="p1" [d] {B}="p2" [d] {D.}="p3" [l] {C}="p4" [u] {A}="p5" "p0":"p1"^-{x_3}:"p2"^-{x}:"p3"^-{k}:@{<-}"p4"^-{g}:@{<-}"p5"^-{h}:@{<-}"p0"^-{x_2} "p5":"p2"_-{f}:@{}"p0"|-{\tn{pb}}}}}
[r(3.5)]
{\xybox{\xygraph{!{0;(1.5,0):(0,.6667)::} {X_2}="p0" [r] {X}="p1" [d] {B}="p2" [d] {D}="p3" [l] {C}="p4" [u] {A}="p5" [r(.5)u(.2)] {X_3}="p6" "p0":"p1"^-{x_3}:"p2"^-{x}:"p3"^-{k}:@{<-}"p4"^-{g}:@{<-}"p5"^-{h}:@{<-}"p0"^-{x_2}"p6" (:@{<.}"p0"|-{\alpha_x},:"p1"|-{x_5},:"p4"|-{x_4},:@{}"p3"_-{\tn{pb}})}}}} \]
Denoting by $(k_2,\psi_2)$ the result of pulling back $(k,\psi)$ along $f$
\[ \xygraph{!{0;(.4,0):(0,1.25)::} {X_2}="ltl" [r(3)u] {Y_2}="ltr" [d(3)l] {A}="lb" [r(4)u(2)] {X}="rtl" [r(3)u] {Y}="rtr" [d(3)l] {B}="rb" "ltl":"rtl"^-{}:"rb"|-{}="d1"_-{x}:@{<-}^-{f}"lb":@{<-}"ltl"|-{}="d2"^-{x_2}
"ltr":"rtr"^-{}:"rb"|-{}="c1"^-{y}:@{<-}"lb":@{<-}"ltr"|-{}="c2"_-{y_2}|(.67)*=<3pt>{} "rtl":"rtr"^(.4){k} "ltl":"ltr"^(.4){k_2}
"d1":@{}"c1"|(.3){}="d3"|(.7){}="c3" "d3":@{=>}"c3"^-{\psi}
"d2":@{}"c2"|(.3){}="d4"|(.7){}="c4" "d4":@{=>}"c4"^-{\psi_2}} \]
as in Lemma \ref{lem:PTJ-fib}, the desired naturality comes down to the commutativity of the prism $(X_2,Y_2,Y_3,X_3,C,A)$ in
\[ \xygraph{!{0;(.4,0):(0,1.25)::} {X_3}="ltl" [r(3)u] {Y_3}="ltr" [d(3)l] {C}="lb" [r(4)u(2)] {X}="rtl" [r(3)u] {Y}="rtr" [d(3)l] {D}="rb" "ltl":"rtl"^(.6){x_5}:"rb"|-{}="d1"_-{kx}:@{<-}^-{g}"lb":@{<-}"ltl"|-{}="d2"^-{x_4}
"ltr":"rtr"^-{y_5}:"rb"|-{}="c1"^-{ky}:@{<-}"lb":@{<-}"ltr"|-{}="c2"_-{y_4}|(.67)*=<3pt>{} "rtl":"rtr"^(.4){k} "ltl":"ltr"^(.4){k_3}
"d1":@{}"c1"|(.3){}="d3"|(.7){}="c3" "d3":@{=>}"c3"^-{k\psi}
"d2":@{}"c2"|(.3){}="d4"|(.7){}="c4" "d4":@{=>}"c4"^-{\psi_3}
"lb" [l(6)] {A}="llb" "ltl" [l(6)] {X_2}="lltl" "ltr" [l(6)] {Y_2}="lltr"
"lltl" (:"llb"|-{}="d5",:"lltr"^(.4){k_2}:"llb"|(.33)*=<3pt>{}|-{}="c5")
"lltl":"ltl"^(.6){\alpha_x} "lltr":"ltr"^-{\alpha_y} "llb":"lb"_-{h}
"d5":@{}"c5"|(.3){}="d6"|(.7){}="c6" "d6":@{=>}"c6"^-{\psi_2}} \]
in which $(k_3,\psi_3)$ is the result of pulling back $(k,k\psi)$ along $g$. Since $g\psi_3\alpha_x = gh\psi_2$, and both $\psi_3\alpha_x$ and $h\psi_2$ are chosen $g$-opcartesian, it follows that $\psi_3\alpha_x = h\psi_2$. Taking codomains of this gives $y_4k_3\alpha_x = y_4\alpha_yk_2$, and $y_5k_3\alpha_x = y_5\beta_yk_2$, so $k_3\alpha_x = \alpha_yk_2$. The uniqueness part of opcartesianness can be used to verify 2-naturality.
\end{proof}
Given a commutative triangle as on the left
\begin{equation}\label{eq:right-part-poly-2-cell} 
\xygraph{{\xybox{\xygraph{{A}="p0" [r(2)] {B}="p1" [dl] {C}="p2" "p0":"p1"^-{h}:"p2"^-{g}:@{<-}"p0"^-{f}}}}
[r(3.5)d(.05)]
{\xybox{\xygraph{{\ca K/A}="p0" [r(2)] {\ca K/B}="p1" [dl] {\ca K/C}="p2" "p0":@{<-}"p1"^-{\Delta_h}:"p2"^-{\Sigma_g}:@{<-}"p0"^-{\Sigma_f} "p0" [d(.5)r(.85)] :@{=>}[r(.3)]^-{\kappa}}}}}
\end{equation}
one has a canonical 2-natural transformation as on the right obtained from $\Sigma_g\Sigma_h = \Sigma_f$ and $\Sigma_h \ladj \Delta_h$. When moreover $f$ and $g$ are split fibrations, since split fibrations compose, the 2-functors $\Sigma_f$ and $\Sigma_g$ lift to 2-functors
\[ \begin{array}{lccr}
{\overline{\Sigma}_f : \Algs {\Phi_{\ca K,A}} \longrightarrow \Algs {\Phi_{\ca K,C}}} &&&
{\overline{\Sigma}_g : \Algs {\Phi_{\ca K,B}} \longrightarrow \Algs {\Phi_{\ca K,C}}} \end{array} \]
respectively.
\begin{lem}\label{lem:right-part-poly-2-cell}
The 2-natural transformation (\ref{eq:right-part-poly-2-cell}) lifts to a 2-natural transformation
\[ \xygraph{!{0;(1.25,0):(0,.8)::} {\Algs {\Phi_{\ca K,A}}}="p0" [r(2)] {\Algs {\Phi_{\ca K,B}}}="p1" [dl] {\Algs {\Phi_{\ca K,C}}}="p2" "p0":@{<-}"p1"^-{\overline{\Delta}_h}:"p2"^-{\overline{\Sigma}_g}:@{<-}"p0"^-{\overline{\Sigma}_f} "p0" [d(.5)r(.85)] :@{=>}[r(.3)]^-{\overline{\kappa}}} \]
when $f$ and $g$ each have the structure of a split fibration.
\end{lem}
\begin{proof}
Let $p:E \to B$ be a split fibration. We have to show that the component $\kappa_p$ is a morphism of split fibrations, that is to say, that post-composition with $\kappa_p$ preserves chosen cartesian 2-cells. So we consider a chosen $fp_2$-cartesian 2-cell $\psi$ as in
\[ \xygraph{!{0;(2,0):(0,.5)::} {X}="p0" [r] {E_2}="p1" [r] {E}="p2" [d] {B}="p3" [l(.5)d] {C}="p4" [l(.5)u] {A}="p5"
"p0":@/^{1pc}/"p1"^-{x_1}|-{}="t" "p5":"p3"_-{h} "p0":@/_{1pc}/"p1"_-{x_2}|-{}="b":"p2"^-{\kappa_p}:"p3"^-{p}:"p4"^-{g}:@{<-}"p5"^-{f}:@{<-}"p1"^-{p_2} "t":@{}"b"|(.25){}="d"|(.75){}="c" "d":@{=>}"c"^-{\psi}
"p1":@{}"p3"|-{\tn{pb}}} \]
and we must show that $\kappa_p\psi$ is chosen $gp$-cartesian. By the way in which one describes the cleavage for a composite of split fibrations, to say that $\kappa_p\psi$ is chosen $gp$-cartesian is to say that $\kappa_p\psi$ is chosen $p$-cartesian and $p\kappa_p\psi$ is chosen $g$-cartesian; and to say that $\psi$ is chosen $fp_2$-cartesian is to say that $\psi$ is chosen $p_2$-cartesian and $p_2\psi$ is chosen $f$-cartesian. By the way in which one describes the cleavage of a morphism resulting from pulling back a split fibration, to say that $\psi$ is chosen $p_2$-cartesian is to say that $\kappa_p\psi$ is chosen $p$-cartesian. Since $h$ is a morphism of split fibrations and $p_2\psi$ is chosen $f$-cartesian, $hp_2\psi$ is chosen $g$-cartesian, and so the result follows.
\end{proof}
\begin{thm}\label{thm:fibrational-polynomials}
Let $\ca K$ be a 2-category with pullbacks and comma objects. Denote by $P:\ca K/I \to \ca K/J$ the polynomial 2-functor associated to the polynomial on the left
\[ \xygraph{{\xybox{\xygraph{{I}="p0" [r] {E}="p1" [r] {B}="p2" [r] {J}="p3" "p0":@{<-}"p1"^-{s}:"p2"^-{p}:"p3"^-{t}}}}
[r(5)]
{\xybox{\xygraph{!{0;(1.5,0):(0,.5)::} {I}="p0" [ur] {E_1}="p1" [r] {B_1}="p2" [dr] {J}="p3" [dl] {B_2}="p4" [l] {E_2}="p5" "p0":@{<-}"p1"^-{s_1}:"p2"^-{p_1}:"p3"^-{t_1}:@{<-}"p4"^-{t_2}:@{<-}"p5"^-{p_2}:"p0"^-{s_2}
"p1":"p5"_-{f_2} "p2":"p4"^-{f_1}
"p1":@{}"p4"|-{\tn{pb}} "p0" [r(.5)] {\scriptstyle{=}} "p3" [l(.5)] {\scriptstyle{=}}}}}} \]
and by $\phi:P \to Q$ the 2-natural transformation associated to the morphism of polynomials indicated on the right in the previous display. Then
\begin{enumerate}
\item If $p$ has the structure of a split opfibration and $t$ has the structure of a split fibration, then $P$ lifts to a 2-functor $\overline{P}:\Algs {\Phi_{\ca K,I}} \to \Algs {\Phi_{\ca K,J}}$ such that $U^{\Phi_{\ca K,J}}\overline{P} = PU^{\Phi_{\ca K,I}}$.\label{thmcase:fib-poly}
\item If in the context of (\ref{thmcase:fib-poly}) $I$ is discrete, then $P$ is a familial 2-functor.\label{thmcase:fam-poly}
\item If $(f_2,f_1):p_1 \to p_2$ is a morphism of split opfibrations and $(f_1,1_J):t_1 \to t_2$ is a morphism of split fibrations, then $\phi$ lifts to a 2-natural transformation $\overline{\phi}:\overline{P} \to \overline{Q}$ such that $U^{\Phi_{\ca K,J}}\overline{\phi} = \phi U^{\Phi_{\ca K,I}}$.\label{thmcase:morphism-fib-poly}
\item If in the context of (\ref{thmcase:morphism-fib-poly}) $I$ is discrete, then $\phi$ is a familial 2-natural transformation.\label{thmcase:morphism-fam-poly}
\end{enumerate}
\end{thm}
\begin{proof}
When $I$ is discrete $\Phi_{\ca K,I}$ is the identity 2-monad, and so (\ref{thmcase:fam-poly}) follows from (\ref{thmcase:fib-poly}) by the definition of familial 2-functor, and (\ref{thmcase:morphism-fam-poly}) follows from (\ref{thmcase:morphism-fib-poly}) by the definition of familial 2-natural transformation. To obtain (\ref{thmcase:fib-poly}), define $\overline {P} = \overline{\Sigma}_t\overline{\Pi}_p\overline{\Delta}_p$, where $\overline{\Pi}_p$ exists by Lemma \ref{lem:PTJ-fib} and the remarks immediately following that result, and $\overline{\Sigma}_t$ exists by the composability of split fibrations. To obtain (\ref{thmcase:morphism-fib-poly}), note that $\phi$ may be obtained as the composite
\[ \xygraph{!{0;(2,0):(0,.4)::} {\ca K/I}="p0" [ur] {\ca K/E_1}="p1" [r] {\ca K/B_1}="p2" [dr] {\ca K/J.}="p3" [dl] {\ca K/B_2}="p4" [l] {\ca K/E_2}="p5" "p0":"p1"^-{\Delta_{s_1}}:"p2"^-{\Pi_{p_1}}|(.4){}="t1":"p3"^-{\Sigma_{t_1}}|(.3){}="t2":@{<-}"p4"^-{\Sigma_{t_2}}|(.7){}="b2":@{<-}"p5"^-{\Pi_{p_2}}|(.6){}="b1":@{<-}"p0"^-{\Delta_{s_2}}
"p1":@{<-}"p5"_-{\Delta_{f_2}} "p2":@{<-}"p4"_-{\Delta_{f_1}}
"p0" [r(.5)] {\iso} "t1":@{}"b1"|(.4){}="d1"|(.6){}="c1" "d1":@{=>}"c1"^{\beta} "t2":@{}"b2"|(.35){}="d2"|(.65){}="c2" "d2":@{=>}"c2"^{\kappa}}  \]
The unnamed isomorphism is obtained by adjunction from the identity $\Sigma_{f_2}\Sigma_{s_1} = \Sigma_{s_2}$ which clearly extends to the level of lax slices, and so this isomorphism lifts to the level of split fibrations by the same argument as that given in the first paragraph of the proof of Lemma \ref{lem:Beck-Chev-lax-slices}. By Lemma \ref{lem:Beck-Chev-lax-slices} $\beta$ also lifts to the level of split fibrations, and by Lemma \ref{lem:right-part-poly-2-cell} $\kappa$ does too, and so (\ref{thmcase:morphism-fib-poly}) follows.
\end{proof}
\begin{rem}\label{rem:fib-poly}
In subsequent work we shall also use the dual of the above result, in which $\ca K$ is replaced by $\ca K^{\tn{co}}$. For instance the dual version of (\ref{thmcase:fam-poly}) says that if $I$ is discrete, $p$ has the structure of a split fibration and $t$ has the structure of a split opfibration, then $P$ is opfamilial. Note also that when $J$ is discrete, $t$, $t_1$ and $t_2$ are automatically split (op)fibrations in a unique way and $f_1$ is a morphism thereof.
\end{rem}

\subsection{Preservation of sifted colimits.}
\label{ssec:sifted-colim-preservation}
Two dimensional monad theory becomes a particularly powerful framework when applied to codescent object preserving 2-monads. For instance, as explained in \cite{Bourke-Thesis, Lack-Codescent}, a codescent object preserving 2-monad $T$ on a 2-category of the form $\Cat(\ca E)$ automatically satisfies the conditions of Power's general coherence theorem \cite{Lack-Codescent, Power-GeneralCoherenceResult}. Thus knowing $T$ preserves codescent objects implies a coherence theorem for $T$-algebras in this case. In particular one recovers the usual coherence theorems for monoidal, braided monoidal and symmetric monoidal categories in this way, without any combinatorial analysis. Moreover as discussed in \cite{Weber-CodescCrIntCat}, knowing that a 2-monad preserves codescent objects is one desirable condition leading to the ability to compute internal algebra classifiers involving $T$.

Codescent objects are particular instances of a class of 2-categorical colimits called sifted colimits, which include also reflexive coequalisers and filtered colimits. Recall that any weight $J:\ca C^{\op} \to \Cat$ (with $\ca C$ a small 2-category) determines, by virtue of the cocompleteness of $\Cat$ as a 2-category, a functor
\[ J * - : [\ca C,\Cat] \longrightarrow \Cat \]
given on objects by taking colimits in $\Cat$ weighted by $J$, and $J$ is a \emph{sifted weight} when $J * -$ preserves finite products. A \emph{sifted colimit} is a weighted colimit whose weight is sifted. In other words, sifted colimits are exactly those colimits which in $\Cat$ commute with finite products. Sufficient conditions on a polynomial in $\Cat$ so that its corresponding polynomial 2-functor preserves sifted colimits is provided by
\begin{thm}\label{thm:polynomials-sifted-colims}
The polynomial 2-functor associated to a polynomial
\[ \xygraph{{I}="p1" [r] {A}="p2" [r] {B}="p3" [r] {J}="p4" "p1":@{<-}"p2"^-{s}:"p3"^-{p}:"p4"^-{t}} \]
in $\Cat$ such that $I$ is discrete and $p$ is a discrete fibration or a discrete opfibration with finite fibres, preserves sifted colimits.
\end{thm}
\begin{proof}
It suffices to show that $\Pi_p\Delta_s$ preserves sifted colimits since $\Sigma_t$ as a left adjoint preserves all colimits. We consider the case where $p$ is a discrete fibration; the proof for the case where $p$ is a discrete opfibration is similar. By Theorem \ref{thm:fibrational-polynomials}(\ref{thmcase:fib-poly}) applied in the case $\ca K = \Cat^{\co}$, one has the commutative diagram
\[ \xygraph{!{0;(3,0):(0,.3333)::} {\Cat/I}="p0" [ur] {\Algs {\Psi_{\ca K,E}}}="p1" [r] {\Algs {\Psi_{\ca K,B}}}="p2" [d] {\Cat/B.}="p3" [l] {\Cat/E}="p4" "p0":"p1"^-{\overline{\Delta}_s}:"p2"^-{\overline{\Pi}_p}:"p3"^-{U^{\Psi_{\ca K,B}}}:@{<-}"p4"^-{\Pi_p}:@{<-}"p0"^-{\Delta_s} "p1":"p4"^-{U^{\Psi_{\ca K,E}}}} \]
By Remark \ref{rem:fibrations-as-coalgebras} $U^{\Psi_{\ca K,E}}$ and $U^{\Psi_{\ca K,B}}$ create all colimits, and so it suffices to show that $\overline{\Pi}_p$ preserves sifted colimits.

Since $\Algs {\Psi_{\ca K,E}}$ and $\Algs {\Psi_{\ca K,B}}$ are 2-equivalent to the functor 2-categories $[E,\Cat]$ and $[B,\Cat]$, and in these terms $\overline{\Delta}_p$ corresponds to the process of precomposition with $p$, by Lemma \ref{lem:PTJ-exp-fib}(\ref{lemcase:PTJ-exp-discrete}) one may identify $\overline{\Pi}_p$ with the process of right Kan extension along $p$. But since $p$ is a discrete fibration, such right Kan extensions are computed simply by taking products over the fibres of $p$. Since these products are finite, and limits and colimits in $[B,\Cat]$ are componentwise, the result follows by the definition of ``sifted colimit''.
\end{proof}
\begin{exams}\label{exs:finite-fam-pres-sifted-colimits}
By Theorem \ref{thm:polynomials-sifted-colims} the finite product completion and finite coproduct completion endofunctors of $\Cat$ described in Example \ref{ex:Fam-Ps-Mnd} preserve sifted colimits.
\end{exams}
\begin{exams}\label{exs:Batanin-Berger}
In part 3 of \cite{BataninBerger-HtyThyOfAlgOfPolyMnd} Batanin and Berger exhibit various flavours of ``operad'' as algebras of polynomial monads over $\Set$. The underlying endo-polynomial in all of their examples is of the form
\begin{equation}\label{eq:BatBer-polys}
\xygraph{{I}="p0" [r] {E}="p1" [r] {B}="p2" [r] {I}="p3"
"p0":@{<-}"p1"^-{s}:"p2"^-{p}:"p3"^-{t}}
\end{equation}
where $p$ is a function with finite fibres. The algebras of the associated monad on $\Set/I$ are the particular flavour of operad under consideration in $\Set$, however for the homotopical aspects of that work, it is also important to regard (\ref{eq:BatBer-polys}) as a componentwise discrete polynomial in $\Cat$. The corresponding 2-monad on $\Cat/I$ is familial and opfamilial by Theorem \ref{thm:fibrational-polynomials} and preserves sifted colimits by Theorem \ref{thm:polynomials-sifted-colims}.
\end{exams}

\section{Examples of polynomial 2-monads on $\Cat$}
\label{sec:SMonCat-2-monad}

In this section we exhibit the 2-monads
\begin{enumerate}
\item $\MCMnd$ for monoidal categories,
\item $\SMCMnd$ for symmetric monoidal categories,
\item $\BMCMnd$ for braided monoidal categories,
\item $\FCMnd$ for categories with finite coproducts, and
\item $\FPMnd$ for categories with finite products
\end{enumerate}
on $\Cat$, as polynomial 2-monads. These are all well-known cartesian 2-monads which were, for instance, basic examples for Kelly and his collaborators \cite{BWellKellyPower-2DMndThy, Kelly-ClubsDoctrines, KellyLack-PropertyLikeStructures}, in the establishment of 2-dimensional monad theory. We recall them as such in Section \ref{ssec:Cart-2-Monads}, and then in Section \ref{ssec:ex-endo-as-poly} exhibit their underlying endofunctors as polynomial. In Section \ref{ssec:bagdomains} we reexpress Johnstone's idea of bagdomain data, as a way of indirectly exhibiting the unit and multiplication of a polynomial monad. We apply this method in Section \ref{ssec:exhibiting-examples} to our examples. Many other examples are exhibited in \cite{Weber-OpPoly2Mnd}.

\subsection{The examples as cartesian 2-monads.}
\label{ssec:Cart-2-Monads}
Let $X$ be a category. Then the objects of $\MCMnd(X)$, $\SMCMnd(X)$, $\BMCMnd(X)$, $\FCMnd(X)$ and $\FPMnd(X)$ are the same, namely they are finite sequences $(x_1,...,x_n)$ of objects of $X$. Denoting $\underline{n}$ as the discrete category $\{1,...,n\}$ with $n$ objects, we regard a sequence $(x_1,...,x_n)$ as a functor $x : \underline{n} \to X$.

A morphism $x \to y$ in $\FCMnd(X)$ consists of a function $\phi$ and a natural transformation $\overline{\phi}$ as on the left in
\[ \xygraph{{\xybox{\xygraph{{\underline{m}}="p0" [r(2)] {{\underline{n}}}="p1" [dl] {X}="p2" "p0":"p1"^-{\phi}:"p2"^-{y}:@{<-}"p0"^-{x} "p0" [d(.5)r(.85)] :@{=>}[r(.3)]^-{\overline{\phi}}}}}
[r(4.5)]
{\xybox{\xygraph{!{0;(1,0):(0,1.5)::}
{x_1}="t1" [r] {x_2}="t2" [r] {x_3}="t3" [r] {x_4}="t4"
"t1" [r(.5)d] {y_1}="b1" [r] {y_2}="b2" [r] {y_3}="b3"
"t1"-"b2"|(.3)@{>}_(.3){\overline{\phi}_1}
"t2"-"b1"|(.35)@{>}^(.3){\overline{\phi}_2}
"t3"-"b2"|(.6)@{>}_(.6){\overline{\phi}_3}
"t4"-"b2"|(.5)@{>}^(.5){\overline{\phi}_4}}}}} \]
An alternative point of view is that such a morphism is a function ($\phi$) decorated by the morphisms of $X$ (the components of $\overline{\phi}$) as indicated in the example on the right. Since it is useful to be able to reason precisely from this latter point of view, we make a brief notational digression.

We denote by $\S$ the category whose objects are natural numbers, and whose morphisms $m \to n$ are functions $\underline{m} \to \underline{n}$. The category $\S$ is a skeleton of the category of finite sets and functions. We denote a sequence $(x_1,...,x_n)$ of objects of $X$ alternatively as $(x_k)_{1{\leq}k{\leq}n}$, or as $(x_k)_k$, as convenience dictates. The result of concatenating a sequence of sequences $((x_{k,l})_{1{\leq}l{\leq}n_k})_{1{\leq}k{\leq}m}$, is denoted $(x_{k,l})_{k,l}$. Implicit in this notation is the identification
\[ \{(k,l) \, : \, 1 \leq k \leq m, 1 \leq l \leq n_k\} = \{1,...,n_1 +...+ n_m\} \]
via the lexicographic ordering of the former. Given morphisms $\phi : m_1 \to m_2$ and $\phi_k : n_{1,k} \to n_{2,\phi(k)}$ in $\S$ for $1 \leq k \leq m_1$, we denote by $\phi(\phi_k)_k$ the morphism $\Sigma_k n_{1,k} \to \Sigma_k n_{2,k}$ given by $(\phi(\phi_k)_k)(k,l) = (\phi(k),\phi_k(l))$. In these terms a general morphism of $\FCMnd(X)$, is of the form
\[ (\phi,(\overline{\phi}_k)_{1{\leq}k{\leq}m}) : (x_k)_{1{\leq}k{\leq}m} \longrightarrow (y_k)_{1{\leq}k{\leq}n} \]
where $\phi : m \to n$ is in $\S$, and $\overline{\phi}_k : x_k \to y_{\phi k}$ is in $X$ for $1 \leq k \leq m$. We refer to the datum $\phi$ of such a morphism as the \emph{indexing function}.

The unit for $\FCMnd$ is given by the full inclusion of sequences of length $1$. On objects the components $\mu_X : \FCMnd^2(X) \to \FCMnd(X)$ are given by concatenation of sequences, and on morphisms by
\[ \mu_X(\phi,(\phi_k,(\overline{\phi}_{k,l})_l)_k) = (\phi(\phi_k)_k,(\overline{\phi}_{k,l})_{k,l}), \]
which in intuitive terms, is just substitution of decorated functions. With the notation provided it is straight forward to verify that $(\FCMnd,\eta,\mu)$ is a cartesian 2-monad on $\Cat$.

We have already considered a version of this 2-monad in Example \ref{ex:Fam-Ps-Mnd}. The difference is that here, we have taken a skeleton of the category of finite sets to index our families, and have more carefully book-kept the combinatorics. It is straight forward to check that $\FCMnd$ is lax idempotent{\footnotemark{\footnotetext{Or in older language, $\FCMnd$ is a Kock-Z\"oberlein 2-monad of colimit like variance.}}} and thus to exhibit the pseudo algebras of either version as being the same, namely, categories with finite coproducts. The virtue of the more pain-staking approach taken here, is that $\FCMnd$ is a 2-monad rather than just a pseudo monad.

One defines the 2-monad $\FPMnd$ as $\FPMnd(X) = \FCMnd(X^{\op})^{\op}$. In direct terms, a morphism of $\FPMnd(X)$ is of the form
\[ (\phi,(\overline{\phi}_k)_{1{\leq}k{\leq}m}) : (x_k)_{1{\leq}k{\leq}m} \longrightarrow (y_k)_{1{\leq}k{\leq}n} \]
where $\phi : n \to m$ is in $\S$, and $\overline{\phi}_k : x_{\phi k} \to y_k$ is in $X$ for $1 \leq k \leq n$. This is a colax idempotent cartesian 2-monad, and its pseudo algebras are categories with finite products.

The morphisms of the category $\SMCMnd(X)$ are defined to be those of $\FCMnd(X)$ (or equally well, of $\FPMnd(X)$) whose underlying indexing functions are bijections. It is clear that $\SMCMnd$ is a sub-2-monad of $\FCMnd$, and that the inclusion $\iota^S_C : \SMCMnd \hookrightarrow \FCMnd$ is a cartesian monad morphism. It follows that $\SMCMnd$ is also a cartesian 2-monad. It is straight forward to verify directly that pseudo $\SMCMnd$-algebras are exactly unbiased{\footnotemark{\footnotetext{This means that an $n$-ary tensor product is defined for all $n$, and invertible coherence maps are given. The coherence theorems one has available enable one to identify these with symmetric monoidal categories defined in the usual biased way, with a unit and binary product, and so we regard pseudo $\SMCMnd$-algebras as symmetric monoidal categories, disregarding the biased/unbiased distinction.}}} symmetric monoidal categories, and strict algebras are exactly symmetric strict monoidal categories. Moreover the various notions of $\SMCMnd$-morphism -- lax, colax, pseudo and strict -- correspond to symmetric lax, colax, strong and strict monoidal functors respectively. Similarly one defines the cartesian 2-monad $\MCMnd$ for monoidal categories, by defining the morphisms of $\MCMnd(X)$ as those of $\FCMnd(X)$ whose indexing functions are identities.

As explained in \cite{JoyalStreet-GeometryTensorCalculus} the 2-monad $\BMCMnd$ for braided monoidal categories, there denoted $\mathbb{B} \wr (-)$, is given similarly as for $\SMCMnd$, except that the indexing bijections are replaced by indexing braids, in the definition of the morphisms of $\BMCMnd(X)$. The process of taking the underlying permutation of a braid, gives a cartesian monad morphism $\pi : \BMCMnd \to \SMCMnd$, and so $\BMCMnd$ is also a cartesian 2-monad. Once again the different types of algebras and algebra morphisms of $\BMCMnd$ reconcile in the expected way with braided (strict) monoidal categories and braided (lax, colax, strong or strict) monoidal functors. One also has a cartesian monad morphism $\iota^M_B : \MCMnd \to \BMCMnd$, whose components can be regarded as the identity on objects inclusion which regards morphisms of $\MCMnd(X)$ as identity braids whose strings are decorated by the morphisms of $X$. The composite $\pi\iota^M_B$ is denoted $\iota^M_S$, and its components regards morphisms of $\MCMnd(X)$ as identity permutations whose strings are decorated by the morphisms of $X$.

To summarise, one has a diagram of cartesian 2-monads and cartesian monad morphisms as on the left
\begin{equation}\label{eq:examples-diagram}
\xygraph{*{\xybox{\xygraph{!{0;(1,0):(0,.5)::} {\MCMnd}="p0" [r] {\BMCMnd}="p1" [r] {\SMCMnd}="p2" [ur] {\FCMnd}="p3" [d(2)] {\FPMnd}="p4" "p0":"p1"_-{\iota^M_B}:"p2"_-{\pi}(:"p3"^-{\iota^S_C},:"p4"_-{\iota^S_P})
"p0":@/^{1.5pc}/"p2"^-{\iota^M_S}}}}
[r(5.5)]
*!(0,.1){\xybox{\xygraph{!{0;(1,0):(0,.5)::} {\N}="p0" [r] {\B}="p1" [r] {\P}="p2" [ur] {\S}="p3" [d(2)] {\S^{\op}}="p4" "p0":"p1"^-{}:"p2"^-{}(:"p3"^-{},:"p4"^-{}) "p0":@/^{1.5pc}/"p2"^-{}}}}}
\end{equation}
and the result of evaluating at $1$ is denoted as on the right. Each monad morphism here is componentwise an identity on objects, and so in particular each functor on the right is the identity on objects. On the right one has $\MCMnd(1) = \N$ the natural numbers, $\BMCMnd(1) = \B$ the braid category, $\SMCMnd(1) = \P$ the permutation category, and so on. Recall, the non-empty hom-sets of $\P$ are $\P(n,n) = \Sigma_n$, where $\Sigma_n$ is the group of permutations of $n$ elements, and similarly, $\B(n,n)$ is $\tn{Br}_n$, the $n$-th braid group. In this section we shall establish that the left diagram of (\ref{eq:examples-diagram}) is a diagram of polynomial 2-monads and morphisms thereof.

\subsection{Underlying endofunctors as polynomial functors.}
\label{ssec:ex-endo-as-poly}
We denote by $\S_*$ the coslice $1/\S$, whose objects may be regarded as pairs $(i,n)$ where $n \in \N$ and $1 \leq i \leq n$, with $i$ regarded as a ``chosen basepoint''. A morphism $(i,m) \to (j,n)$ of $\S_*$ is thus a function $f : \underline{m} \to \underline{n}$ such that $fi = j$. By definition, $\S_*$ comes with a forgetful functor $U^{\S} : \S_* \to \S$, and $U^{\S}$ is a discrete opfibration with finite fibres. The fibre over $n \in \S$ may be identified with the set $\underline{n} = \{1,...,n\}$. Since $U^{\S}$ is a discrete opfibration, and $(U^{\S})^{\op}$ is a discrete fibration, $U^{\S}$ and $(U^{\S})^{\op}$ are exponentiable functors.
\begin{lem}\label{lem:FCMnd-FPMnd-polyendos}
The underlying endofunctors of $\FCMnd$ and $\FPMnd$ are the result of applying $\PFun{\Cat}$ to
\[ \xygraph{*{\xybox{\xygraph{{1}="p0" [r] {\S_*}="p1" [r] {\S}="p2" [r] {1}="p3" "p0":@{<-}"p1"^-{}:"p2"^-{U^{\S}}:"p3"^-{}}}}
[r(4.5)]
*{\xybox{\xygraph{{1}="p0" [r] {\S_*^{\op}}="p1" [r(1.25)] {\S^{\op}}="p2" [r] {1}="p3" "p0":@{<-}"p1"^-{}:"p2"^-{(U^{\S})^{\op}}:"p3"^-{}}}}} \]
respectively.
\end{lem}
\begin{proof}
We give the proof for $\FCMnd$, the case of $\FPMnd$ follows similarly. Provisionally we write $P$ for the 2-functor corresponding to the polynomial on the left in the above display. By definition one has
\begin{equation}\label{eq:obj-map-Cfin} 
\xygraph{!{0;(1.75,0):(0,.5)::}
{1}="p1" [r] {\S_*}="p2" [r] {\S}="p3" [r] {1}="p4"
"p1" [u] {X}="s" "p2" [u] {X \times \S_*}="mdpb" [u] {P_*X}="tldpb" [r] {PX}="trdpb"
"p1":@{<-}"p2"_-{}:"p3"_-{U^{\S}}:"p4"_-{} "mdpb"(:"s"_-{}:"p1"_-{},:"p2"^-{p_X},:@{<-}"tldpb"_-{}:"trdpb"^-{}(:"p3"^-{q_X},:"p4"^-{}|-{}="codeq"))
"mdpb" ([r(.5)] {\scriptstyle{\tn{dpb}}}, [l(.5)d(.5)] {\scriptstyle{\tn{pb}}})}
\end{equation}
and we now proceed to identify $PX$ with $\FCMnd(X)$. An object of $PX$ may be regarded as a functor $h:[0] \to PX$, and thus also as a pair $n:[0] \to \S$ together with $h:[0] \to PX$ such that $qh = n$. By $\Delta_{U^{\S}} \ladj \Pi_{U^{\S}}$ and since $q_X = \Pi_{U^{\S}}(p_X)$, such an $h$ is in bijection with $k : (U^{\S})^{-1}\{n\} \to X \times \S_*$ over $\S_*$, but this in turn is just an $n$-tuple $(x_1,...,x_n)$ of objects of $X$. Thus an object of $PX$ is a pair $(n,x)$ where $n \in \N$ and $x : \underline{n} \to X$. Similarly regarding arrows of $PX$ as functors $[1] \to PX$ and using the adjointness $\Delta_{U^{\S}} \ladj \Pi_{U^{\S}}$ in the same way, one finds that a morphism $(m,x) \to (n,y)$ of $PX$ consists of $\phi : \underline{m} \to \underline{n}$ in $\S$ together with $\overline{\phi}$ as in
\[ \xygraph{{\underline{m}}="p0" [r(2)] {\underline{n}}="p1" [dl] {X}="p2" "p0":"p1"^-{\phi}:"p2"^-{y}:@{<-}"p0"^-{x} "p0" [d(.5)r(.85)] :@{=>}[r(.3)]^-{\overline{\phi}}} \]
and so $P$ and $\FCMnd$ agree on objects. Note that $P_*X$ also has an easy explicit description. Namely, an object is a triple $(i,n,x)$ where $(n,x) \in PX$ and $1 \leq i \leq n$, and a morphism $(i,m,x) \to (j,n,y)$ consists of $\phi$ and $\overline{\phi}$ as above making $(\phi,\overline{\phi}) : (m,x) \to (n,y)$ a morphism of $PX$, such that $\phi i = j$. Moreover, all the functors participating in (\ref{eq:obj-map-Cfin}) also admit straight forward direct descriptions in these terms.

We must now verify that $Pf = \FCMnd(f)$ for $f : X \to Y$. Note that $Pf$ is induced from $f$ by the pullbacks and distributivity pullbacks that go into defining $PX$ and $PY$, whereas $\FCMnd(f)$ is the functor described by composing with $f$, that is on objects one has $\FCMnd(f)(n,x) = (n,fx)$. Let us define ${\FCMnd}_*(f) : P_*X \to P_*Y$ to be the functor given on objects by $(i,n,x) \mapsto (i,n,fy)$. By the uniqueness aspects of the universal properties of the pullbacks and distributivity pullbacks involved in defining  $PY$, it suffices to show that
\[ \xygraph{!{0;(1.5,0):(0,.5)::} {X \times \S_*}="p0" [r] {P_*X}="p1" [r] {PX}="p2" [dr] {\S}="p3" [dl] {PY}="p4" [l] {P_*Y}="p5" [l] {Y \times \S_*}="p6" "p0":@{<-}"p1"^-{}:"p2"^-{}:"p3"^-{q_X}:@{<-}"p4"^-{q_Y}:@{<-}"p5"^-{}:"p6"^-{}:@{<-}"p0"^-{f \times 1_{\S_*}} "p1":"p5"_-{{\FCMnd}_*(f)} "p2":"p4"_-{\FCMnd(f)}} \]
commutes. Given that everything in this diagram has been made so explicit, this is a straight forward calculation.
\end{proof}
The functor $(\iota^S_C)_1 : \P \to \S$ is the inclusion of the maximal subgroupoid of $\S$, and similarly $(\iota^S_P)_1$ is also a maximal subgroupoid inclusion. We define $\P_*$ as the maximal subgroupoid of $\S_*$. It comes with a forgetful functor $U^{\P} : \P_* \to \P$. Equivalently, $\P_* = 1/\P$. The functor $U^{\P}$ fits into pullback squares in
\[ \xygraph{*{\xybox{\xygraph{!{0;(1.25,0):(0,1)::}
{\P_*}="p0" [r] {\P}="p1" [d] {\S}="p2" [l] {\S_*}="p3" "p0":"p1"^-{U^{\P}}:"p2"^-{(\iota^S_C)_1}:@{<-}"p3"^-{U^{\S}}:@{<-}"p0"^-{}:@{}"p2"|-{\tn{pb}} "p0" [d(.5)l] {1}="l" "p1" [d(.5)r] {1}="r" "l"(:@{<-}"p0",:@{<-}"p3") "r"(:@{<-}"p1",:@{<-}"p2")}}}
[r(5.5)]
*!(0,.02){\xybox{\xygraph{!{0;(1.25,0):(0,1)::} {\P_*}="p0" [r] {\P}="p1" [d] {\S^{\op}}="p2" [l] {\S_*^{\op}}="p3" "p0":"p1"^-{U^{\P}}:"p2"^-{(\iota^S_P)_1}:@{<-}"p3"^-{(U^{\S})^{\op}}:@{<-}"p0"^-{}:@{}"p2"|-{\tn{pb}} "p0" [d(.5)l] {1}="l" "p1" [d(.5)r] {1}="r" "l"(:@{<-}"p0",:@{<-}"p3") "r"(:@{<-}"p1",:@{<-}"p2")}}}} \]
By Propositions \ref{prop:homfunctors-bi-fibs} and \ref{prop:poly-along-cart-2cat-case} the underlying endofunctor of $\SMCMnd$ is a polynomial 2-functor, and one may regard these morphisms of polynomials as corresponding to the 2-natural transformations $\iota^S_C$ and $\iota^S_P$. Similarly by pulling back along $\pi_1$ and $(\iota^M_S)_1$, one defines $U^{\B} : \B_* \to \B$ and $U^{\N} : \N_* \to \N$, exhibits $\tnb{B}$ and $\tnb{M}$ as polynomial 2-functors, and $\pi$, $\iota^M_S$ and $\iota^M_B$ as arising from morphisms of polynomials. To summarise, we have
\begin{cor}\label{cor:examples-as-polyendos}
At the level of the underlying endofunctors, the left diagram of (\ref{eq:examples-diagram}) is in the essential image of $\PFun{\Cat}$.
\end{cor}
To establish that (\ref{eq:examples-diagram}) is a diagram of polynomial 2-\emph{monads} and morphisms thereof, one way to proceed would be to write down the polynomial monad structures for $\FCMnd$ and $\FPMnd$ explicitly, and then use Proposition \ref{prop:poly-along-cart-2cat-case}. To avoid computations involving composites of such polynomials, for instance when verifying monad axioms, we use an alternative approach. This approach is described in general in the next section, and applied to our examples in Section \ref{ssec:exhibiting-examples}.

\subsection{Bagdomain data as polynomial monad structures.}
\label{ssec:bagdomains}
The construction of the finite coproduct and finite product completion were seen to underlie polynomial \emph{pseudo} monads in $\Cat$ in Example \ref{ex:Fam-Ps-Mnd}. However as explained in Section \ref{ssec:Cart-2-Monads} these constructions underlie 2-monads, and a framework for these is described in this section. It is little more than an explicit encoding of some of the developments of Section 2 of \cite{Johnstone-VariationsBagdomain}, in the language of polynomials.

Recall that when a category $\ca E$ has pullbacks the codomain functor $\tn{cod} : \ca E^{[1]} \to \ca E$ is a fibration, and that a morphism of $\ca E^{[1]}$ is $\tn{cod}$-cartesian if and only if its underlying square in $\ca E$ is a pullback square.
\begin{defn}\label{def:p-fibration}
Let $\ca E$ be a category with pullbacks and $p:E \to B$ be a morphism therein. Then a \emph{$p$-fibration} is a $\tn{cod}$-cartesian arrow into $p$, and a \emph{morphism of $p$-fibrations} is a morphism in $\ca E^{[1]} \downarrow p$ between $p$-fibrations.
\end{defn}
As such a $p$-fibration consists of an arrow $f:X \to Y$ of $\ca E$ together with $(u,v)$ fitting into a pullback square
\[ \xygraph{{\xybox{\xygraph{{X}="p0" [r] {Y}="p1" [d] {B}="p2" [l] {E}="p3" "p0":"p1"^-{f}:"p2"^-{v}:@{<-}"p3"^-{p}:@{<-}"p0"^-{u}:@{}"p2"|-{\tn{pb}}}}}
[r(3)] {\xybox{\xygraph{{X_1}="p0" [r] {Y_1}="p1" [d] {Y_2}="p2" [l] {X_2}="p3" "p0":"p1"^-{f_1}:"p2"^-{v_3}:@{<-}"p3"^-{f_2}:@{<-}"p0"^-{u_3}:@{}"p2"|-{\tn{pb}}}}}} \]
as on the left. In this context we say that $(u,v)$ is a \emph{$p$-fibration structure} on $f$. For a morphism of $p$-fibrations $(f_1,u_1,v_1) \to (f_2,u_2,v_2)$ one has morphisms $u_3$ and $v_3$ as on the right in the previous display such that $u_3u_2 = u_1$ and $v_3v_2 = v_1$. Clearly, $u_3$ is uniquely determined by $v_3$ and the universal property of the pullback containing $(f_2,u_2,v_2)$, and so in minimalistic terms, $(u_3,v_3)$ amounts to $v_3:Y_1 \to Y_2$ over $B$.
\begin{exam}\label{ex:p-for-monoid-monad}
Let $\ca E = \Set$ and regard $U^{\N} : \N_* \to \N \, \in \, \Set$. Explicitly, $\N_* = \{(i,j) \, : \, i \in \N, 1 \leq j \leq i\}$
and $U^{\N}(i,j) = i$. Observe that for $n \in \N$, $|(U^{\N})^{-1}\{n\}| = n$. To give a function $f:X \to Y$ the structure of a $U^{\N}$-fibration is by definition to give functions $u$ and $v$ fitting into a pullback square
\[ \xygraph{{X}="p0" [r] {Y}="p1" [d] {\N}="p2" [l] {\N_*}="p3" "p0":"p1"^-{f}:"p2"^-{v}:@{<-}"p3"^-{U^{\N}}:@{<-}"p0"^-{v}:@{}"p2"|-{\tn{pb}}} \]
and so $f$ must have finite fibres since $U^{\N}$ does. Moreover for $y \in Y$ such that $f^{-1}\{y\} = n$, $u$ restricts to a bijection $u_y$ between $f^{-1}\{y\}$ and $(U^{\N})^{-1}\{n\}$, and so amounts to a linear order on $f^{-1}\{y\}$. Conversely given a function $f:X \to Y$ with finite fibres and a linear order on each fibre, one determines $v$ by the formula $vy = |f^{-1}\{y\}|$, and then $u$ by taking $ux = (n,i)$, where $n = |f^{-1}\{fx\}|$ and $i$ is $x$'s position in the linear order on $f^{-1}\{fx\}$. Thus a $U^{\N}$-fibration is a function with finite fibres together with a linear order on each fibre.
\end{exam}
Recall that a \emph{double category} (resp. \emph{double functor}) is a category (resp. functor) internal to $\CAT$. Given a category $\ca E$ we denote by $\ca D(\ca E)$ the double category as on the left
\[ \xygraph{*!(0,.75)-=(4.5,1){\xybox{\xygraph{!{0;(2,0):(0,1)::} {\ca E^{[2]}}="p0" [r] {\ca E^{[1]}}="p1" [r] {\ca E}="p2"
"p2":"p1"|-{\id_{\ca E}} "p1":@<1.5ex>"p2"^-{\tn{dom}_{\ca E}} "p1":@<-1.5ex>"p2"_-{\tn{cod}_{\ca E}} "p0":@<1.5ex>"p1"^-{} "p0":"p1"|-{\tn{comp}_{\ca E}} "p0":@<-1.5ex>"p1"_-{}}}}
[r(4.5)]
*-=(1.5,1){\xybox{\xygraph{{X_1}="p0" [r] {X_0}="p1" "p0":@<1ex>"p1"^-{d_1} "p0":@<-1ex>"p1"_-{d_0}}}}} \]
whose objects are those of $\ca E$, vertical arrows and horizontal arrows are morphisms of $\ca E$, and squares are commutative squares in $\ca E$. Any double category $X$ has, by forgetting horizontal identities and compositions, an underlying graph internal to $\CAT$, consisting of the source and target functors as on the right in the previous display. Our conventions are that $X_0$ is the category of objects and vertical arrows, and $X_1$ is the category of horizontal arrows and squares between them.

In the context of Definition \ref{def:p-fibration} we define
\[ \ca U_p : \ca D(p) \longrightarrow \ca D(\ca E) \]
a graph morphism internal to $\CAT$, as follows. The category $\ca D(p)_0$ of objects and vertical arrows of $\ca D(p)$ is $\ca E$, and $(\ca U_p)_0 = 1_{\ca E}$. The category $\ca D(p)_1$ of horizontal arrows and squares of $\ca D(p)$, is the category of $p$-fibrations and their morphisms. Every $p$-fibration has an underlying arrow of $\ca E$, and every morphism of $p$-fibrations has an underlying square, the assignations of which provide $(\ca U_p)_1$.
\begin{thm}\label{thm:bagdomain-data}
Let $\ca E$ be a category with finite limits and $p:E \to B$ an exponentiable morphism therein. There is a bijection between the following types of data:
\begin{enumerate}
\item Unit and multiplication 2-cells in $\Polyc {\ca E}$ making
\[ \xygraph{{1}="p0" [r] {E}="p1" [r] {B}="p2" [r] {1}="p3" "p0":@{<-}"p1"^-{}:"p2"^-{p}:"p3"^-{}} \]
the underlying endoarrow of a monad in $\Polyc {\ca E}$.\label{thmcase:bagdomain-1}
\item Double category structures on $\ca D(p)$ making $\ca U_p$ a double functor.
\end{enumerate}
\end{thm}
\begin{proof}
We denote by $P:1 \to 1$ the endoarrow of (\ref{thmcase:bagdomain-1}). To give the data of a 2-cell $u:1_1 \to P$ in $\Polyc {\ca E}$ is to give $(u_1,u_2)$ fitting into a pullback square
\[ \xygraph{{\xybox{\xygraph{{1}="p0" [r] {1}="p1" [d] {B}="p2" [l] {E}="p3" "p0":"p1"^-{}:"p2"^-{u_1}:@{<-}"p3"^-{p}:@{<-}"p0"^-{u_2}:@{}"p2"|-{\tn{pb}}}}}
[r(2.5)u(.125)]
{\xybox{\xygraph{{X}="p0" [r] {X}="p1" [d] {1}="p2" [l] {1}="p3" "p0":"p1"^-{1_X}:"p2"^-{}:@{<-}"p3"^-{}:@{<-}"p0"^-{}:@{}"p2"|-{\tn{pb}}}}}
[r(2.5)d(.07)]
{\xybox{\xygraph{{X}="p0" [r] {X}="p1" [d] {Y}="p2" [l] {Y}="p3" "p0":"p1"^-{1_X}:"p2"^-{f}:@{<-}"p3"^-{1_Y}:@{<-}"p0"^-{f}:@{}"p2"|-{\tn{pb}}}}}} \]
as on the left. Given this data any identity arrow $1_X:X \to X$ acquires the structure of a $p$-fibration, namely that which makes the square on the middle in the previous display a morphism $1_X \to 1_1$ of $p$-fibrations. Given a morphism $f:X \to Y$ of $\ca E$, with respect to these $p$-fibration structures on identity arrows, the square on the right in the previous display is clearly a morphism of $p$-fibrations. Thus one has a functor $\id_p:\ca E \to \ca D(p)_1$ such that $(\ca U_p)_1\id_p = \id_{\ca E}$. Conversely, such a functor amounts to assigning a $p$-fibration structure to each identity map in such a way that for all $f:X \to Y$ the right square above is a morphism of $p$-fibrations. Thus in particular there is a $p$-fibration structure on $1_1$, which amounts to $(u_1,u_2)$ as on the left in the previous display. The processes just described are easily verified to exhibit a bijection between 2-cells $u:1_1 \to P$ of $\Polyc {\ca E}$ and functors $\id_p$ as above.

To give a functor $\tn{comp}_p:\ca D(p)_1 \times_{\ca E} \ca D(p)_1 \to \ca D(p)_1$ such that $(\ca U_p)_1\tn{comp}_p = \tn{comp}_{\ca E}(\ca U_p)_2$, where $(\ca U_p)_2$ is induced in the evident way using $(\ca U_p)_1$, is to give, for each composable pair $f:X \to Y$, $g:Y \to Z$ in $\ca E$ together with $p$-fibration structures on $f$ and $g$, a $p$-fibration structure on the composite $gf$, and that this assignation be functorial. This functoriality means that given morphisms of $p$-fibrations $(u,v):f \to f'$ and $(v,w):g \to g'$ as in
\[ \xygraph{{X}="p0" [r] {Y}="p1" [r] {Z}="p2" [d] {Z'}="p3" [l] {Y'}="p4" [l] {X'}="p5" "p0":"p1"^-{f}:"p2"^-{g}:"p3"^-{w}:@{<-}"p4"^-{g'}:@{<-}"p5"^-{f'}:@{<-}"p0"^-{u} "p1":"p4"^{v}} \]
the composite square underlies a $p$-fibration morphism $(u,w):gf \to g'f'$.

Suppose that a 2-cell $m:P \comp P \to P$ in $\Polyc {\ca E}$ is given. Then given also a composable pair $f:X \to Y$, $g:Y \to Z$ in $\ca E$ together with $p$-fibration structures on $f$ and $g$, one has
\[ \xygraph{{1}="b1" [r] {E}="b2" [r] {B}="b3" [r] {1}="b4" [r] {E}="b5" [r] {B}="b6" [r] {1}="b7" "b4" [u] {B \times E}="p1" [u] {F}="dl" ([r(1.5)] {B^{(2)}}="dr", [l(1.5)] {E^{(2)}}="p2")
"b2" [r(.5)d] {X}="t1" "b4" [d] {Y}="t2" "b5" [r(.5)d] {Z}="t3"
"b1":@{<-}"b2"_-{}:"b3"_-{p}:"b4"_-{}:@{<-}"b5"_-{}:"b6"_-{p}:"b7"_-{} "dl":"p1"_-{}(:"b3"_-{},:"b5"^-{\pi_E}) "b2":@{<-}"p2"_-{}:"dl"_-{}:"dr"_-{}:"b6"^(.7){} "b1":@{<-}"p2"^-{} "dr":"b7"^-{}
"t1":"t2"_-{f}:"t3"_-{g} "t1":"b2" "t2" (:"b3",:"b5") "t3":"b6"
"b3" [u(1.25)] {\scriptstyle{\tn{pb}}} "b5" [u(1.25)] {\scriptstyle{\tn{dpb}}} "b4" [u(.5)] {\scriptstyle{\tn{pb}}} "t2" ([u(.45)l(1.1)] {\scriptstyle{\tn{pb}}}, [u(.45)r(1.1)] {\scriptstyle{\tn{pb}}})
"t2":@/_{1pc}/@{.>}"p1" "t2":@/^{1pc}/@{.>}"dl" "t3":@{.>}"dr" "t1":@{.>}"p2"} \]
in which the lower pullback squares witness the $p$-fibration structures on $f$ and $g$. By the universal property of $B \times E$ one induces the unique arrow $Y \to B \times E$ commuting with the morphisms into $B$ and $E$. Thus the bottom right pullback is around $(p,\pi_E)$, and so by the distributivity pullback one induces the morphisms $Y \to F$ and $Z \to B^{(2)}$. Finally using the top right pullback one induces $X \to E^{(2)}$. The squares $(X,Y,F,E^{(2)})$ and $(Y,Z,B^{(2)},F)$ so arising are pullbacks, and so the composite pullback
\[ \xygraph{!{0;(1.5,0):(0,.6667)::} 
{X}="p0" [r] {Y}="p1" [r] {Z}="p2" [d] {B^{(2)}}="p3" [l] {F}="p4" [l] {E^{(2)}}="p5" "p0":"p1"^-{f}:"p2"^-{g}:"p3"^-{}:@{<-}"p4"^-{}:@{<-}"p5"^-{}:@{<-}"p0"^-{} "p1":"p4"^{}
"p0":@{}"p4"|-{\tn{pb}} "p1":@{}"p3"|-{\tn{pb}}
"p5" [d] {E}="p6" [r(2)] {B}="p7"
"p3":"p7"^{m_1}:@{<-}"p6"^-{p}:@{<-}"p5"^{m_2}
"p5":@{}"p7"|-{\tn{pb}}} \]
exhibits a $p$-fibration structure on $gf$, and so we have described the object map of the functor $\tn{comp}_p$.

The composite $p^{(2)}:E^{(2)} \to B^{(2)}$ is the middle map of the composite polynomial $P \comp P$. When $f$ and $g$ are the morphisms $E^{(2)} \to F$ and $F \to B^{(2)}$ respectively, the top row of vertical arrows in the previous display are identities, and so one has a $p$-fibration structure on $p^{(2)}$. Moreover for general $f$ and $g$, the composite of the top pullbacks of the previous display exhibits a morphism $gf \to p^{(2)}$ of $p$-fibrations. From
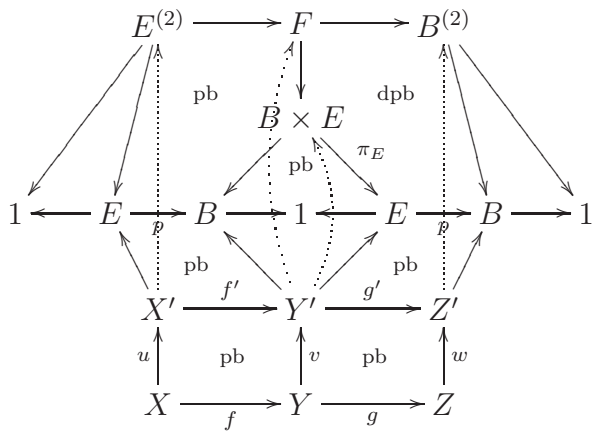
\[ \xygraph{{1}="b1" [r] {E}="b2" [r] {B}="b3" [r] {1}="b4" [r] {E}="b5" [r] {B}="b6" [r] {1}="b7" "b4" [u] {B \times E}="p1" [u] {F}="dl" ([r(1.5)] {B^{(2)}}="dr", [l(1.5)] {E^{(2)}}="p2")
"b2" [r(.5)d] {X'}="t1" "b4" [d] {Y'}="t2" "b5" [r(.5)d] {Z'}="t3"
"b1":@{<-}"b2"_-{}:"b3"_-{p}:"b4"_-{}:@{<-}"b5"_-{}:"b6"_-{p}:"b7"_-{} "dl":"p1"_-{}(:"b3"_-{},:"b5"^-{\pi_E}) "b2":@{<-}"p2"_-{}:"dl"_-{}:"dr"_-{}:"b6"^(.7){} "b1":@{<-}"p2"^-{} "dr":"b7"^-{}
"t1":"t2"^-{f'}:"t3"^-{g'} "t1":"b2" "t2" (:"b3",:"b5") "t3":"b6"
"b3" [u(1.25)] {\scriptstyle{\tn{pb}}} "b5" [u(1.25)] {\scriptstyle{\tn{dpb}}} "b4" [u(.5)] {\scriptstyle{\tn{pb}}} "t2" ([u(.45)l(1.1)] {\scriptstyle{\tn{pb}}}, [u(.45)r(1.1)] {\scriptstyle{\tn{pb}}})
"t2":@/_{1pc}/@{.>}"p1" "t2":@/^{1pc}/@{.>}"dl" "t3":@{.>}"dr" "t1":@{.>}"p2"
"t1" [d] {X}="t4" "t2" [d] {Y}="t5" "t3" [d] {Z}="t6"
"t4"(:"t5"_-{f}(:"t6"_-{g}:"t3"_{w},:"t2"_{v}),:"t1"^{u})
"t1":@{}"t5"|-{\tn{pb}} "t2":@{}"t6"|-{\tn{pb}}} \]
in which $(u,v)$ and $(v,w)$ are morphisms of $p$-fibrations, $(u,w)$ is a morphism $gf \to g'f'$ of $p$-fibrations, thus giving the arrow map for $\tn{comp}_p$. In this way from a 2-cell $m:P \comp P \to P$ in $\Polyc {\ca E}$, one obtains the functor $\tn{comp}_p$ such that $(\ca U_p)_1\tn{comp}_p = \tn{comp}_{\ca E}(\ca U_p)_2$.

Conversely given $\tn{comp}_p$, the pullbacks appearing in the formation of $P \comp P$ exhibit $p$-fibration structures on $E^{(2)} \to F$ and $F \to B^{(2)}$, and applying $\tn{comp}_p$ to this composable pair gives a $p$-fibration structure to $p^{(2)}$, which amounts to the components $(m_1,m_2)$ of a 2-cell $m:P \comp P \to P$. It is straight forward to verify that the processes described here exhibit a bijection between such 2-cells and such functors $\tn{comp}_p$. The straight forward verification that the unit and associative laws for $u$ and $m$ correspond with the unit and associative laws for horizontal composition in the double category $\ca D(p)$ is left to the reader.
\end{proof}
It was the data of Theorem \ref{thm:bagdomain-data}(\ref{thmcase:bagdomain-1}) that Johnstone named \emph{bagdomain data} in Definition 2.1 of \cite{Johnstone-VariationsBagdomain}.
\begin{rem}\label{rem:bicategories-structured-polynomials}
In the context of Theorem \ref{thm:bagdomain-data}, one can consider polynomials as on the left
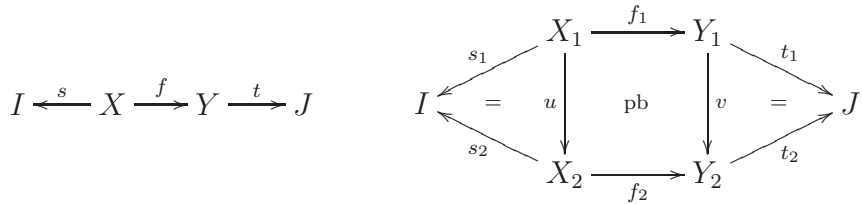
\[ \xygraph{{\xybox{\xygraph{{I}="p0" [r] {X}="p1" [r] {Y}="p2" [r] {J}="p3" "p0":@{<-}"p1"^-{s}:"p2"^-{f}:"p3"^-{t}}}}
[r(5)d(.075)]
{\xybox{\xygraph{!{0;(1.5,0):(0,.5)::} {I}="p0" [ur] {X_1}="p1" [r] {Y_1}="p2" [dr] {J}="p3" [dl] {Y_2}="p4" [l] {X_2}="p5" "p0":@{<-}"p1"^-{s_1}:"p2"^-{f_1}:"p3"^-{t_1}:@{<-}"p4"^-{t_2}:@{<-}"p5"^-{f_2}:"p0"^-{s_2}
"p1":"p5"_-{u} "p2":"p4"^-{v}
"p1":@{}"p4"|-{\tn{pb}} "p0" [r(.5)] {\scriptstyle{=}} "p3" [l(.5)] {\scriptstyle{=}}}}}} \]
in which the middle map $f$ has the structure of a $p$-fibration, morphisms thereof as on the right in which the middle pullback square is a morphism of $p$-fibrations. We call such a polynomial a \emph{$p$-structured polynomial} from $I$ to $J$. By definition the process of pulling back in $\ca E$ carries along $p$-fibration structures, and by Theorem \ref{thm:bagdomain-data} one has a composition of $p$-fibrations. Thus composition in $\Polyc {\ca E}$ can be extended to a composition of $p$-structured polynomials. In this way one has a bicategory $\Polyc p$, whose 1-cells are $p$-structured polynomials, together with a strict homomorphism of bicategories $\Polyc p \to \Polyc {\ca E}$, and the polynomial
\[ \xygraph{{1}="p0" [r] {E}="p1" [r] {B}="p2" [r] {1}="p3"
"p0":@{<-}"p1"^-{}:"p2"^-{p}:"p3"^-{}} \]
is by definition terminal in the hom category $\Polyc p(1,1)$. The monad structure it acquires is by definition that of Theorem \ref{thm:bagdomain-data}. Thus for any category $\ca E$ with finite limits, any monad in $\Polyc {\ca E}$ on $1$ arises in a manner analogous to the process described in Example \ref{ex:Fam-Ps-Mnd}.
\end{rem}
\begin{exam}\label{ex:monoid-monad-p-fib}
By concatenating linear orders on fibres, $U^{\N}$-fibrations as characterised in Example \ref{ex:p-for-monoid-monad} can be composed, and this composition is functorial with respect to morphisms of $U^{\N}$-fibrations. The polynomial monad one gets by Theorem \ref{thm:bagdomain-data} is the monoid monad on $\Set$. Regarding $U^{\N}$ as a functor between discrete categories and arguing the same way, one recovers the monad $\tnb{M}$ on $\Cat$.
\end{exam}

\subsection{Exhibiting the examples.}
\label{ssec:exhibiting-examples}
The foregoing discussion is adapted to the 2-categorical context in the following way. For an exponentiable morphism $p:E \to B$ in a 2-category $\ca K$ with pullbacks, we define $p$-fibrations and morphisms thereof as in Definition \ref{def:p-fibration}. Given $p$-fibrations $(f_1,u_1,v_1)$ and $(f_2,u_2,v_2)$ and morphisms thereof as in
\[ \xygraph{*!(0,1)=(0,1.75){\xybox{
\xygraph{!{0;(1.75,0):} 
{\xybox{\xygraph{{X_1}="p0" [r] {Y_1}="p1" [d] {B}="p2" [l] {E}="p3" "p0":"p1"^-{f_1}:"p2"^-{v_1}:@{<-}"p3"^-{p}:@{<-}"p0"^-{u_1}:@{}"p2"|-{\tn{pb}}}}}
[r]
{\xybox{\xygraph{{X_2}="p0" [r] {Y_2}="p1" [d] {B}="p2" [l] {E}="p3" "p0":"p1"^-{f_2}:"p2"^-{v_2}:@{<-}"p3"^-{p}:@{<-}"p0"^-{u_2}:@{}"p2"|-{\tn{pb}}}}}
[d(.015)r]
{\xybox{\xygraph{{X_1}="p0" [r] {Y_1}="p1" [d] {Y_2}="p2" [l] {X_2}="p3" "p0":"p1"^-{f_1}:"p2"^-{v_3}:@{<-}"p3"^-{f_2}:@{<-}"p0"^-{u_3}:@{}"p2"|-{\tn{pb}}}}}
[r]
{\xybox{\xygraph{{X_1}="p0" [r] {Y_1}="p1" [d] {Y_2}="p2" [l] {X_2}="p3" "p0":"p1"^-{f_1}:"p2"^-{v_4}:@{<-}"p3"^-{f_2}:@{<-}"p0"^-{u_4
}:@{}"p2"|-{\tn{pb}}}}}
[r(1.25)]
{\xybox{\xygraph{!{0;(1.25,0):(0,.8)::}
{X_1}="p0" [r] {Y_1}="p1" [d] {Y_2}="p2" [l] {X_2}="p3" "p0":"p1"^-{f_1}:@/^{1pc}/"p2"^(.4){v_4}|(.4){}="pcr":@{<-}"p3"^-{f_2}:@/_{1pc}/@{<-}"p0"_(.6){u_4}|(.6){}="pcl" "p1":@/_{1pc}/"p2"_(.6){v_3}|(.6){}="pdr" "p3":@/^{1pc}/@{<-}"p0"^(.4){u_3}|(.4){}="pdl"
"pdl":@{}"pcl"|(.25){}="dl"|(.75){}="cl" "dl":@{=>}"cl"^-{\alpha}
"pdr":@{}"pcr"|(.25){}="dr"|(.75){}="cr" "dr":@{=>}"cr"^-{\beta}}}}}
}}} \]
a 2-cell $(u_3,v_3) \to (u_4,v_4)$ is a pair $(\alpha,\beta)$ such that the cylinder on the right commutes, and $u_2\alpha = \id$ and $v_2\beta = \id$. In minimalistic terms using the 2-dimensional universal property of pullbacks, $(\alpha,\beta)$ is determined by the 2-cell $\beta$ over $B$.

We call a category (resp. functor) internal to $\TwoCAT$ a \emph{double 2-category} (resp. \emph{double 2-functor}). For any 2-category $\ca K$ one has the double 2-category $\ca D(\ca K)$
\[ \xygraph{!{0;(2,0):(0,1)::} {\ca K^{[2]}}="p0" [r] {\ca K^{[1]}}="p1" [r] {\ca K.}="p2"
"p2":"p1"|-{\id_{\ca K}} "p1":@<1.5ex>"p2"^-{\tn{dom}_{\ca K}} "p1":@<-1.5ex>"p2"_-{\tn{cod}_{\ca K}} "p0":@<1.5ex>"p1"^-{} "p0":"p1"|-{\tn{comp}_{\ca K}} "p0":@<-1.5ex>"p1"_-{}} \]
In the situation where $\ca K$ has pullbacks and $p$ is as above, the graph morphism $\ca U_p : \ca D(p) \to \ca D(\ca K)$ internal to $\TwoCAT$ is defined analogously to the definition given above with $\ca D(p)_0 = \ca K$, $(\ca U_p)_0 = 1_{\ca K}$, $\ca D(p)$ is the 2-category of $p$-fibrations as just defined and $(\ca U_p)_1$ is the evident forgetful 2-functor. The following result is proved in the same way as Theorem \ref{thm:bagdomain-data} with the 2-dimensional universal properties one now has providing the additional required information.
\begin{thm}\label{thm:2-bagdomain-data}
Let $\ca K$ be a 2-category with finite limits and $p:E \to B$ be an exponentiable morphism therein. There is a bijection between the following types of data:
\begin{enumerate}
\item Unit and multiplication 2-cells in $\Polyc {\ca K}$ making
\[ \xygraph{{1}="p0" [r] {E}="p1" [r] {B}="p2" [r] {1}="p3" "p0":@{<-}"p1"^-{}:"p2"^-{p}:"p3"^-{}} \]
the underlying endoarrow of a 2-monad in $\Polyc {\ca K}$.
\item Double 2-category structures on $\ca D(p)$ making $\ca U_p$ a double 2-functor.
\end{enumerate}
\end{thm}
There is also a version of this result useful for exhibiting morphisms of polynomial 2-monads. Given a pullback square
\[ \xygraph{!{0;(1.5,0):(0,.6667)::} {E_1}="p0" [r] {B_1}="p1" [d] {B_2}="p2" [l] {E_2}="p3" "p0":"p1"^-{p_1}:"p2"^-{v}:@{<-}"p3"^-{p_2}:@{<-}"p0"^-{u}:@{}"p2"|-{\tn{pb}}} \]
in $\ca K$ in which $p_1$ and $p_2$ are exponentiable, composing with this pullback square is the effect on objects of a 2-functor $\ca D(u,v)_1 : \ca D(p_1) \to \ca D(p_2)$. Defining $\ca D(u,v)_0 : \ca K \to \ca K$ to be the identity, one has a graph morphism $\ca D(u,v) : \ca D(p_1) \to \ca D(p_2)$ internal to $\TwoCAT$ over $\ca D(\ca K)$. It is straight forward to extend the proof of Theorem \ref{thm:2-bagdomain-data} to a proof of
\begin{thm}\label{thm:2-bagdomain-data-monad-morphism-version}
In the context just described, there is a bijection between the following types of data:
\begin{enumerate}
\item Unit and multiplication 2-cells on in $\Polyc {\ca K}$ making
\[ \xygraph{!{0;(1.5,0):(0,.5)::} {1}="p0" [ur] {E_1}="p1" [r] {B_1}="p2" [dr] {1}="p3" [dl] {E_2}="p4" [l] {B_2}="p5" "p0":@{<-}"p1"^-{}:"p2"^-{p_1}:"p3"^-{}:@{<-}"p4"^-{}:@{<-}"p5"^-{p_2}:"p0"^-{}
"p1":"p5"_-{u} "p2":"p4"^-{v}
"p1":@{}"p4"|-{\tn{pb}}} \]
the underlying endoarrow of a morphism of 2-monads in $\Polyc {\ca K}$.
\item Double 2-category structures on $\ca D(p_1)$ and $\ca D(p_2)$ making $\ca U_{p_1}$, $\ca U_{p_2}$ and $\ca D(u,v)$ double 2-functors.
\end{enumerate}
\end{thm}
We now have sufficiently many tools to enable us to witness the diagram (\ref{eq:examples-diagram}) as being a diagram of polynomial 2-monads. First we shall understand some of the relevant double 2-categories of $p$-fibrations for appropriate $p$. Recall from the definition of $U^{\P} : \P_* \to \P$ given in Section \ref{ssec:ex-endo-as-poly}, that $\P$ is the permutation category, and that $\P_*$ has objects those of $\N_*$ as described explicitly in Example \ref{ex:p-for-monoid-monad}. An arrow $(m,i) \to (n,j)$ of $\P_*$ can only exist when $m = n$, in which case it is a permutation $\rho \in \Sigma_n$ such that $\rho i = j$. The forgetful functor $U^{\P}$ is a discrete fibration of groupoids with finite fibres.
\begin{lem}\label{lem:UPFib}\ \\
\begin{enumerate}
\item To give a $U^{\P}$-fibration is to give a functor $f:X \to Y$ which is a discrete fibration and a discrete opfibration with finite fibres, together with a linear order on each fibre.\label{lcase:UPFib-0-cells}
\item Given $U^{\P}$-fibrations $f_1 : X_1 \to Y_1$ and $f_2 : X_2 \to Y_2$, a to give a morphism $f_1 \to f_2$ of $U^{\P}$-fibrations is to give a pair $(h,k)$ of functors fitting into a pullback square as on the left
\[ \xygraph{{\xybox{\xygraph{{X_1}="p0" [r] {Y_1}="p1" [d] {Y_2}="p2" [l] {X_2}="p3" "p0":"p1"^-{f_1}:"p2"^-{k}:@{<-}"p3"^-{f_2}:@{<-}"p0"^-{h}:@{}"p2"|-{
\tn{pb}}}}}
[r(4)]
{\xybox{\xygraph{!{0;(2,0):(0,.5)::} {X_1}="p0" [r] {Y_1}="p1" [d] {Y_2.}="p2" [l] {X_2}="p3" "p0":"p1"^-{f_1}:@/^{1pc}/"p2"^(.4){k_2}|(.4){}="pcr":@{<-}"p3"^-{f_2}:@/_{1pc}/@{<-}"p0"_(.6){h_2}|(.6){}="pcl" "p1":@/_{1pc}/"p2"_(.6){k_1}|(.6){}="pdr" "p3":@/^{1pc}/@{<-}"p0"^(.4){h_1}|(.4){}="pdl"
"pdl":@{}"pcl"|(.25){}="dl"|(.75){}="cl" "dl":@{=>}"cl"^-{\alpha}
"pdr":@{}"pcr"|(.25){}="dr"|(.75){}="cr" "dr":@{=>}"cr"^-{\beta}}}}} \]
such that for all $y \in Y_1$, $h|_{f_1^{-1}\{y\}} : f_1^{-1}\{y\} \to f_2^{-1}\{ky\}$ is order preserving.\label{lcase:UPFib-1cells}
\item Given $f_1$ and $f_2$ as in (\ref{lcase:UPFib-1cells}), and morphisms $(h_1,k_1)$ and $(h_2,k_2)$ of $U^{\P}$-fibrations, to give a 2-cell $(h_1,k_1) \to (h_2,k_2)$ in $\ca D(U^{\P})_1$ is to give a pair $(\alpha,\beta)$ fitting into a commutative cylinder as on the right in the previous display.
\label{lcase:UPFib-2-cells}
\end{enumerate}
\end{lem}
\begin{proof}
In each case we shall explain how to go from $U^{\P}$-fibrations, morphisms or 2-cells thereof to the data described in the statement, and how to go back, leaving to the reader the straight forward task of showing that these processes give the required bijections.

(\ref{lcase:UPFib-0-cells}): To give a functor $f:X \to Y$ the structure of a $U^{\P}$-fibration is by definition to give functors $u$ and $v$ fitting into a pullback square
\begin{equation}\label{eq:pbsquare-UPfibn}
\xygraph{{X}="p0" [r] {Y}="p1" [d] {\P.}="p2" [l] {\P_*}="p3" "p0":"p1"^-{f}:"p2"^-{v}:@{<-}"p3"^-{U^{\P}}:@{<-}"p0"^-{u}:@{}"p2"|-{\tn{pb}}}
\end{equation}
Thus $f$ is a discrete fibration and a discrete opfibration with finite fibres since $U^{\P}$ is, and these properties on a functor are pullback stable. For $y \in Y$, writing $n = vy$, because of the pullback (\ref{eq:pbsquare-UPfibn}) on objects, $u$ restricts to a bijection between $f^{-1}\{y\}$ and $(U^{\P})^{-1}\{n\}$, this latter set being $\{(n,1),...,(n,n)\}$ in explicit terms. But to give such a bijection is to give a linear order on $f^{-1}\{y\}$, by taking the $i$-th element of $f^{-1}\{y\}$ to be the element sent to $(n,i)$ by the bijection.

Conversely suppose one has a functor $f:X \to Y$ which is a discrete fibration and a discrete opfibration with finite fibres, together with a linear order on each fibre. For $y \in Y$ one can define $vy = |f^{-1}\{y\}|$. Putting $n = vy$ and denoting the linearly ordered set $f^{-1}\{y\}$ as $\{x_1 < ... < x_n\}$, for $1 \leq i \leq n$ we define $ux_i = (n,i)$. For $\beta : y \to y'$ in $Y$, the unique lifting property $f$ enjoys by being a discrete opfibration provides a function $f^{-1}\{y\} \to f^{-1}\{y'\}$, the unique lifting property $f$ enjoys by being a discrete fibration provides a function $f^{-1}\{y'\} \to f^{-1}\{y\}$, and the uniqueness of these lifting properties ensures that these functions are mutually inverse. Let $n = vy = vy'$ and denote by $\{x_1 < ... < x_n\}$ and $\{x'_1 < ... < x'_n\}$ the linearly ordered sets $f^{-1}\{y\}$ and $f^{-1}\{y'\}$ respectively. The lifting properties just described give us, for $1 \leq i \leq n$, a morphism $\alpha_i : x_i \to x'_{\rho i}$ of $X$, and the above bijection $f^{-1}\{y\} \to f^{-1}\{y'\}$ is given by $x_i \mapsto x'_{\rho i}$. Thus $\rho \in \Sigma_n$, and we define $v\beta = \rho$ and $u\alpha_i$ to be $\rho : (n,i) \to (n,\rho i)$.

(\ref{lcase:UPFib-1cells}): We write $(u_1,v_1)$ and $(u_2,v_2)$ for the morphisms which exhibit the $U^{\P}$-fibration structures of $f_1$ and $f_2$ respectively. By definition a morphism $f_1 \to f_2$ of $U^{\P}$-fibrations determines $(h,k)$ fitting into a pullback square as in the statement, and the equation $u_2h = u_1$ restricted to $f^{-1}\{y\}$, implies that $h|_{f_1^{-1}\{y\}}$ is order preserving.

Conversely suppose that one has $(h,k)$ as in the statement. We must verify that $u_2h = u_1$ and $v_2k = v_1$. For $y \in Y_1$ one has
\[ v_1y = |f_1^{-1}\{y\}| = |f_2^{-1}\{ky\}| = v_2ky  \]
in which the first and third equalities follow from the definitions of $v_1$ and $v_2$, and the second equality follows by the pullback of the statement on objects. Writing $x_{1,i}$ (resp. $x_{2,i}$) for the $i$-th element of $f_1^{-1}\{y\}$ (resp. $f_2^{-1}\{y\}$), one has
\[ u_1x_{1,i} = (n,i) = u_2x_{2,i} = u_2hx_i \]
where $n = v_1y = v_2ky$ and $x_{2,i}$ is the $i$-th element of $f_2^{-1}\{ky\}$, in which the first and second equalities follow by the definition of $u_1$ and $u_2$, and the third equality follows since $f_2h = kf_1$ and $h$ is order preserving on the fibres of $f_1$.

For $\beta : y \to y'$ in $Y_1$, we denote by $x_{1,i}$, $x_{2,i}$, $x'_{1,i}$ and $x'_{2,i}$ the $i$-th element of $f_1^{-1}\{y\}$, $f_2^{-1}\{ky\}$, $f_1^{-1}\{y'\}$ and $f_2^{-1}\{ky'\}$ respectively. As in (\ref{lcase:UPFib-0-cells}) we have  $\alpha_{1,i} : x_{1,i} \to x'_{1,\rho_1 i}$ such that $f_1\alpha_{1,i} = \beta$, and $\alpha_{2,i} : x_{2,i} \to x'_{2,\rho_2 i}$ such that $f_2\alpha_{2,i} = k\beta$, where $n = v_1y = v_1y'$ and $\rho_1$, $\rho_2 \in \Sigma_n$. We must show that $v_2k\beta = v_1\beta$, that is that $\rho_1 = \rho_2$, and that $u_2h\alpha_{1,i} = u_1\alpha_{1,i}$ for all $1 \leq i \leq n$. But since $h\alpha_{1,i} : hx_{1,i} \to hx'_{1,\rho_1i}$, $hx_{1,i} = x_{2,i}$, $hx'_{1,\rho_1i} = x'_{2,\rho_1i}$, and $f_2h\alpha_{1,i} = k\beta$, these equations follow from the uniqueness of lifts for $f_2$.

(\ref{lcase:UPFib-2-cells}): By definition a 2-cell $(h_1,k_1) \to (h_2,k_2)$ of  $\ca D(U^{\P})_1$ consists of $(\alpha,\beta)$ of the statement making the cylinder commute, and moreover verifying $u_2\alpha = \id$ and $v_2\beta = \id$. It suffices to show that these last two equations are automatic.

Let $y \in Y_1$ and denote by $\{x_1 < ... < x_n\}$ the linearly ordered fibre $f_1^{-1}\{y\}$. We have $\beta_y : k_1y \to k_2y$ in $Y_2$, and for $1 \leq i \leq n$, we have $\alpha_{x_i} : h_1x_i \to h_2x_i$ in $X_2$. By the commutativity of the cylinder $f_2\alpha_{x_i} = \beta_y$. By the definition of $v_2$ on arrows, one has $v_2\beta_y = \id$, and so by the definition of $u_2$ on arrows, one has $u_2\alpha_{x_i} = \id$ as required.
\end{proof}
Thanks to Lemma \ref{lem:UPFib} we have an explicit description of graph $\ca D(U^{\P})$ internal to $\TwoCAT$ and of the internal graph morphism $\ca U_{U^{\P}}$, which involves forgetting the $U^{\P}$-fibration structures. Similarly one can understand
\[ \begin{array}{lccr} {\ca U_{U^{\S}} : \ca D(U^{\S}) \longrightarrow \ca D(\Cat)} &&& {\ca U_{U^{\S^{\op}}} : \ca D(U^{\S^{\op}}) \longrightarrow \ca D(\Cat)} \end{array} \]
by analysing what $U^{\S}$-fibrations (resp. $U^{\S^{\op}}$-fibrations) and their morphisms amount to. To give a functor $f : X \to Y$ a $U^{\S}$-fibration (resp. $U^{\S^{\op}}$-fibration) structure, is to exhibit it as a discrete opfibration (resp. discrete fibration) with finite fibres, and to give a linear order on each fibre. Via the pullback squares relating $U^{\P}$, $U^{\S}$ and $U^{\S^{\op}}$, one has internal graph morphisms $\ca D(\iota^S_C) : \ca D(U^{\P}) \to \ca D(U^{\S})$ and $\ca D(\iota^S_P) : \ca D(U^{\P}) \to \ca D(U^{\S^{\op}})$ over $\ca D(\Cat)$.
\begin{thm}\label{thm:Cat-examples}
The diagram
\[ \xygraph{!{0;(1,0):(0,.5)::} {\MCMnd}="p0" [r] {\BMCMnd}="p1" [r] {\SMCMnd}="p2" [ur] {\FCMnd}="p3" [d(2)] {\FPMnd}="p4" "p0":"p1"_-{\iota^M_B}:"p2"_-{\pi}(:"p3"^-{\iota^S_C},:"p4"_-{\iota^S_P})
"p0":@/^{1.5pc}/"p2"^-{\iota^M_S}} \]
described in (\ref{eq:examples-diagram}) Section \ref{ssec:Cart-2-Monads}, is a diagram of polynomial 2-monads and morphisms thereof.
\end{thm}
\begin{proof}
Any identity functor has a unique $U^{\P}$-fibration structure. Given $U^{\P}$-fibrations $f:X \to Y$ and $g:Y \to Z$, the composite $gf$ is a discrete fibration and discrete opfibration whose fibres are finite. For any $z \in Z$, an element of the fibre $(gf)^{-1}\{z\}$ may be identified as a pair $(x,y)$, where $y \in g^{-1}\{z\}$ and $x \in f^{-1}\{y\}$. Defining $(x_1,y_1) \leq (x_2,y_2)$ if and only if $y_1 < y_2$ or $y_1 = y_2$ and $x_1 \leq x_2$, provides $(gf)^{-1}\{z\}$ with a linear order. Thus there is an evident composition of $U^{\P}$-fibrations. Thus $\ca D(U^{\P})$ acquires a double 2-category structure making $\ca U_{U^{\P}} : \ca D(U^{\P}) \to \ca D(\ca K)$ a double 2-functor. Moreover $\ca U_{U^{\S}}$, $\ca U_{U^{\S^{\op}}}$, $\ca D(\iota^S_C)$ and $\ca D(\iota^S_P)$ are easily witnessed as double 2-functors, thanks to our explicit understanding of $U^{\P}$-fibrations, $U^{\S}$-fibrations and $U^{\S^{\op}}$-fibrations. Thus $\FCMnd$, $\FPMnd$ and $\SMCMnd$ are polynomial 2-monads and $\iota^S_C$ and $\iota^S_P$ are morphisms thereof, by Theorems \ref{thm:2-bagdomain-data} and \ref{thm:2-bagdomain-data-monad-morphism-version}. To exhibit $\tnb{B}$, $\tnb{M}$, $\pi$ and $\iota^M_B$ as polynomial, we appeal to Proposition \ref{prop:poly-along-cart-2cat-case} and Corollary \ref{cor:examples-as-polyendos}, using the fact that $\tnb{S}$ is a polynomial 2-monad.
\end{proof}
Having exhibited these examples as polynomial, Theorems \ref{thm:fibrational-polynomials} and \ref{thm:polynomials-sifted-colims} enable us to read off some of their categorical properties.
\begin{cor}\label{cor:properties-of-2-monads}\ \\
\begin{enumerate}
\item All the 2-monads of Theorem \ref{thm:Cat-examples} are sifted colimit preserving.
\item \cite{Weber-Fam2fun} $\MCMnd$, $\BMCMnd$ and $\SMCMnd$ are familial and opfamilial, $\FCMnd$ is familial, and $\FPMnd$ is opfamilial.
\item $\iota^S_C$ is a familial 2-monad morphism, $\iota^S_P$ is an opfamilial 2-monad morphism, and all the other morphisms of Theorem \ref{thm:Cat-examples} are both familial and opfamilial.
\end{enumerate}
\end{cor}

\section*{Acknowledgements}
\label{sec:Acknowledgements}
The author would like to acknowledge Michael Batanin, Richard Garner, Martin Hyland, Joachim Kock and Ross Street for interesting discussions on the subject of this paper. The author would also like to acknowledge the financial support of the Australian Research Council grant No. DP130101172.


\end{document}